\documentclass[11pt]{amsart}

\pdfoutput=1

\usepackage[text={420pt,660pt},centering]{geometry}

\usepackage{esint,amssymb} 
\usepackage{graphicx}
\usepackage{MnSymbol}
\usepackage{mathtools} 
\usepackage[colorlinks=true, pdfstartview=FitV, linkcolor=blue, citecolor=blue, urlcolor=blue,pagebackref=false]{hyperref}
\usepackage{microtype}

\usepackage{bm}
\usepackage{dsfont}
\usepackage{mathrsfs}
\usepackage{cancel} \usepackage{enumitem} 

\newtheorem{proposition}{Proposition}
\newtheorem{theorem}[proposition]{Theorem}
\newtheorem{lemma}[proposition]{Lemma}
\newtheorem{corollary}[proposition]{Corollary}

\theoremstyle{remark}

\theoremstyle{definition}
\newtheorem{definition}[proposition]{Definition}

\numberwithin{equation}{section}
\numberwithin{proposition}{section}
\numberwithin{figure}{section}
\numberwithin{table}{section}

\newcommand{\N}{\mathbb{N}}

\newcommand{\R}{\mathbb{R}}

\newcommand{\E}{\mathbb{E}}

\newcommand{\eps}{\varepsilon}

\renewcommand{\ge}{\geqslant}
\renewcommand{\leq}{\leqslant}
\renewcommand{\geq}{\geqslant}

\renewcommand{\subset}{\subseteq}
\renewcommand{\bar}{\overline}
\renewcommand{\tilde}{\widetilde}

\renewcommand{\hat}{\widehat}

\renewcommand{\d}{\mathrm{d}}

\newcommand{\var}{\mathbb{V}\mathrm{ar}}

\newcommand{\la}{\left\langle}
\newcommand{\ra}{\right\rangle}

\renewcommand{\H}{\mathsf{H}}

\newcommand{\Pout}{\mathcal{P}}
\renewcommand{\div}{\mathsf{div}}

\newcommand{\tz}{\tilde z}
\newcommand{\ty}{\tilde y}
\newcommand{\bF}{\bar F}

\newcommand{\htwo}{{h_2}}
\newcommand{\pr}[1]{{(#1)}}
\newcommand{\rp}[1]{{|#1)}}
\newcommand{\gau}{{\mathbf{g}}}
\newcommand{\dash}{{\text{--}}}
\newcommand{\Int}{\mathsf{int}\,}
\newcommand{\rv}[1]{\textcolor{black}{#1}} \newcommand{\rvv}[1]{\textcolor{red}{#1}}

\begin{document}

\author[Hong-Bin Chen]{Hong-Bin Chen}
\address[Hong-Bin Chen]{Courant Institute of Mathematical Sciences, New York University, New York, New York, USA}
\email{hbchen@cims.nyu.edu}

\author[Jiaming Xia]{Jiaming Xia}
\address[Jiaming Xia]{Department of Mathematics, University of Pennsylvania, Philadelphia, Pennsylvania, USA}
\email{xiajiam@sas.upenn.edu}


\title{Free energy of multi-layer generalized linear models}

\begin{abstract}
We compute the high-dimensional limit of the free energy associated with a multi-layer generalized linear model. Under certain technical assumptions, we identify the limit in terms of a variational formula. The approach is to first show that the limit is a solution to a Hamilton--Jacobi equation whose initial condition is related to the limiting free energy of a model with one fewer layer. Then, we conclude with an iteration.
\end{abstract}

\maketitle

\section{Introduction}

\subsection{Setting and main result}\label{s.setting}
Let us describe the model.
For $n\in\N$, let $X$ be an $\R^{n}$-valued random vector with distribution $P_X$, serving as the original signal. Fix any $L\in\N$ as the number of layers. For $l\in\{0,1,2,\dots,L\}$, let $n_l = n_l(n)\in\N$ be the dimension of the signal at the $l$-th layer. We assume that $n_0 =n $ and
\begin{align}\label{e.n_l/n}
    \lim_{n\to\infty}\frac{n_l}{n} = \alpha_l>0, 
\end{align}
for some $\alpha_l>0$. In particular, we have that $\alpha_0 =1 $.

For each $l\in\{1,2,\dots, L\}$, let
\begin{itemize}
    \item $\varphi_l: \R\times \R^{k_l}\to\R$ be a measurable function for some fixed $k_l\in\N$ (independent of~$n$);
    \item $(A^{(l)}_j)_{1\leq j\leq n_l}$ be 
    a finite sequence of $\R^{k_l}$-valued random vectors, all together with law $P_{A^\pr{l}}$;
    \item $\Phi^\pr{l}$ be an $n_l\times n_{l-1}$ random matrix with law $P_{\Phi^\pr{l}}$.
\end{itemize}
For $l'\geq l$, we also write
\begin{align}\label{e.A^[l,l'],Phi^[l,l']}
    A^{[l,l']}=\left(A^\pr{m}\right)_{l\leq m\leq l'},\qquad \Phi^{[l,l']} = \left(\Phi^\pr{m}\right)_{l\leq m\leq l'},
\end{align}
and denote their laws by $P_{A^{[l,l']}}$ and $P_{\Phi^{[l,l']}}$, respectively.

Starting with $X^\pr{0}=X$, we iteratively define, for each $l\in \{1,\dots,L\}$, 
\begin{align*}
    X^{(l)}_j= \varphi_l\left(\frac{1}{\sqrt{n_{l-1}}}\sum^{n_{l-1}}_{k=1}\Phi^{(l)}_{jk}X^{(l-1)}_k,\ A^{(l)}_j\right),\quad \forall 1\leq j\leq n_l.
\end{align*}
Viewing the action of $\varphi_l$ component-wise, we also write
\begin{align}\label{e.X^l}
     X^{(l)}=\varphi_l\left(\frac{1}{\sqrt{n_{l-1}}}\Phi^{(l)}X^{(l-1)},\ A^{(l)}\right).
\end{align}
For $\beta\ge 0$, the observable is given by
\begin{align}\label{e.Y^circ}
    Y^\circ = \sqrt{\beta}X^{(L)}+Z
\end{align}
where $Z$ is an $n_L$-dimensional standard Gaussian vector. The inference task is to recover $X$ based on the knowledge of $Y^\circ$, $(\varphi_l)_{1\leq l\leq L}$ and $\Phi^{[1,L]}$.

Using~\eqref{e.X^l} iteratively, we can find a deterministic function $\zeta_{L-1}$ such that $X^\pr{L-1}=\zeta_{L-1}(X,A^{[1,L-1]},\Phi^{[1,L-1]})$. We introduce the shorthand notation:
\begin{align}\label{e.x^L-1}
    x^\pr{L-1}=\zeta_{L-1}\left(x,a,\Phi^{[1,L-1]}\right), \quad \forall x \in \R^n,\ a=\left(a^\pr{1},\dots, a^\pr{L-1}\right)\in \prod_{l=1}^{L-1}\R^{n_l\times k_l}.
\end{align}
We emphasize that $x^\pr{L-1}$ is random due to the presence of $\Phi^{[1,L-1]}$ and also depends on the input $x$ and $a$. By Bayes' rule, the law of $(X,A^{[1,L-1]})$ conditioned on $(Y^\circ,\Phi^{[1,L]})$ is given by
\begin{align*}
    \frac{1}{\mathcal Z_{\beta,L,n}^\circ} \Pout_{\beta,L,n}\bigg(Y^\circ\bigg|\frac{1}{\sqrt{n_{L-1}}}\Phi^\pr{L}x^\pr{L-1}\bigg)\d P_X(x) \d P_{A^{[1,L-1]}}(a)
\end{align*}
where
\begin{gather}
    \Pout_{\beta,L,n}(y|z) = \int e^{-\frac{1}{2}|y - \sqrt{\beta}\varphi_L(z,a^\pr{L})|^2 }\d P_{A^\pr{L}}\left(a^\pr{L}\right),\quad \forall y,z \in \R^{n_L},\label{e.P_beta(y|x)}\\
    \mathcal Z_{\beta,L,n}^\circ = \int \Pout_{\beta,L,n}\bigg(Y^\circ\bigg|\frac{1}{\sqrt{n_{L-1}}}\Phi^\pr{L}x^\pr{L-1}\bigg)\d P_X(x) \d P_{A^{[1,L-1]}}(a).\label{e.Z^circ}
\end{gather}
The normalizing factor $\mathcal Z^\circ_{\beta,L,n}$ is called the partition function. The central object to study is the free energy
\begin{align}\label{e.F^circ}
    F^\circ_{\beta,L,n} = \frac{1}{n} \log \mathcal Z_{\beta,L,n}^\circ.
\end{align}

To compute the limit of $\E F^\circ_{\beta,L,n}$ as $n\to\infty$, we make the following assumptions:
\begin{enumerate}[start=1,label={\rm(H\arabic*)}]
\item \label{i.assump_iid_bdd} $X$ has i.i.d. entries, and the law of $X_1$ is supported on $[-1,1]$, independent of $n$ and satisfies that $X_1\neq 0$ with positive probability;
\item \label{i.assump_varphi_l_2^l_diff} for every $l\in\{1,\dots,L\}$, $\varphi_l$ is bounded, not identically zero, and continuously differentiable with bounded derivatives up to the $2^l$-th order;
\item \label{i.assump_Phi_Gaussian} for every $l\in\{1,\dots,L\}$, $\Phi^\pr{l}$ consists of independent standard Gaussian entries, and $(A^\pr{l}_j)_{1\leq j\leq n_l}$ consists of i.i.d.\ $\R^{k_l}$-valued random vectors with a fixed law and bounded a.s.
\end{enumerate}

To state the main result, we need more definitions. Throughout this work, we set
\begin{align}\label{e.R_+}
    \R_+=[0,\infty).
\end{align}
For every $l\in\{0,1,\dots,L\}$ and $n\in\N$, define
\begin{align}\label{e.rho_l}
    \rho_{l,n}= \frac{1}{n_l}\E\left|X^\pr{l}\right|^2.
\end{align}
Due to Lemma~\ref{l.cvg_rho} to be proved later, the following limit exists
\begin{gather}\label{e.rho_cvg}
    \lim_{n\to\infty}\rho_{l,n} = \rho_l
\end{gather}
for some $\rho_l>0$ with explicit expression.
Let $P_{X_1}$ be the law of $X_1$ and $Z'_1$ be a standard Gaussian random variable. Set
\begin{align}\label{e.Psi_0}
    \Psi_0 (r) = \E \log \int_\R e^{rX_1 x_1 + \sqrt{r}Z'_1 x_1 - \frac{r}{2}|x_1|^2}\d P_{X_1}(x_1), \quad\forall r\in\R_+.
\end{align}
For every $l\in\{1,\dots,L\}$, $\rho\geq 0$ and $h=(h_1,h_2)\in [0,\rho]\times\R_+$, define
\begin{align}
    &\Psi_l(h;\rho) \notag
    \\
    & =  \E\log\int\tilde\Pout_{h_2,l}\left(\sqrt{h_2}\varphi_l\left(\sqrt{h_1}V_1+\sqrt{\rho-h_1}W_1,A_1^\pr{l}\right)+Z_1\Big|\sqrt{h_1}V_1+\sqrt{\rho-h_1}w\right)\d P_{W_1}(w), \label{e.Psi_l}
\end{align}
where $V_1,W_1,Z_1$ are independent standard Gaussian random variables and
\begin{align}\label{e.tilde_Pout}
    \tilde\Pout_{h_2,l}(y|z) = \int_{\R^{k_l}} e^{-\frac{1}{2}|y - \sqrt{h_2}\varphi_l(z,a^\pr{l}_1)|^2 }\d P_{A^\pr{l}_1}\left(a^\pr{l}_1\right),\quad\forall y,z\in \R.
\end{align}
Now, we are ready to state the main result.

\begin{theorem}\label{t}
Under assumptions~\ref{i.assump_iid_bdd}--\ref{i.assump_Phi_Gaussian}, it holds that
\begin{align}\label{e.limit}
    \lim_{n\to\infty} \E F^\circ_{\beta,L,n}  = \sup_{z^\pr{L}}\inf_{y^\pr{L}}\sup_{z^\pr{L-1}}\inf_{y^\pr{L-1}}\cdots \sup_{z^\pr{1}}\inf_{y^\pr{1}} \phi_L\left(\beta; y^\pr{1},\cdots,y^\pr{L};z^\pr{1},\cdots,z^\pr{L}\right)
\end{align}
where $\sup_{z^\pr{l}}$ is taken over $z^\pr{l}\in\R_+\times [0,\frac{\alpha_{l-1}\rho_{l-1}}{2}]$, $\inf_{y^\pr{l}}$ is taken over $y^\pr{l}\in [0,\rho_{l-1}]\times \R_+$, and 
\begin{align}\label{e.phi_L}
    &\phi_L\left(\beta; y^\pr{1},\cdots,y^\pr{L};z^\pr{1},\cdots,z^\pr{L}\right)\\
    &\qquad\qquad = \alpha_L\Psi_L\left(y^\pr{L}_1,\beta;\rho_{L-1}\right)+\sum_{l=1}^{L-1}\alpha_l\Psi_l\left(y^\pr{l}_1,y^\pr{l+1}_2;\rho_{l-1}\right) + \Psi_0\left(y^\pr{1}_2\right)\notag
    \\
    &\qquad\qquad\qquad+ \sum_{l=1}^L\left(-y^\pr{l}\cdot z^\pr{l}+\frac{2}{\alpha_{l-1}} z^\pr{l}_1z^\pr{l}_2\right)+\sum_{l=2}^L \frac{\alpha_{l-1}}{2}\left(1+\rho_{l-1}y^\pr{l}_2\right).\notag
\end{align}

\end{theorem}

\rv{Note that $\phi_L$ is a linear combination of relatively simple functions, which is a result of the independence of $\Phi^l$ and $A^l$ between layers. This independence also gives rise to the recursive relation:
\begin{align*}
    \phi_L(\cdots) = \phi_{L-1}\left(y^\pr{L}_2; y^\pr{1},\cdots,y^\pr{L-1};z^\pr{1},\cdots,z^\pr{L-1}\right)  + \alpha_L\Psi_L\left(y^\pr{L}_1,\beta;\rho_{L-1}\right) -y^\pr{L}\cdot z^\pr{L}
    \\
     \quad  + \frac{2}{\alpha_{L-1}}z^\pr{L}_1z^\pr{L}_2 +  \frac{\alpha_{L-1}}{2}\left(1 + \rho_{L-1}y^\pr{L}_2\right).
\end{align*}
Our proof exploits this structure and is built on induction on layers relating the free energy of the $L$-layer model to that of the $(L-1)$-layer model (see Section~\ref{s.iteration}). This inductive step is achieved by the Hamilton--Jacobi equation technique, which we use as a black box (see Proposition~\ref{p.cvg}). By inputting the free energy of the $(L-1)$-layer model as the initial data of a certain Hamilton--Jacobi equation, we get the free energy of the $L$-layer model as the solution of the equation.
}

\subsection{Comments on assumptions}
The nonzero assumptions in~\ref{i.assump_iid_bdd} and~\ref{i.assump_varphi_l_2^l_diff} are reasonable in the setting of statistical inference where only non-constant signals are interesting. They are also purely technical in order to ensure that $\rho_l$ in~\eqref{e.rho_cvg} is nonzero and thus some domain (defined in~\eqref{e.def_Omega_r}) we work on is non-degenerate. In general, one can always consider a reduced model obtained from the original one by starting from the first layer after which all layers including itself contain nonzero signals.
Alternatively, small constants can be added to fulfill the nonzero assumptions, and the effect of these constants is traceable through explicit formulae.

The assumption that $X$ has i.i.d.\ entries in~\ref{i.assump_iid_bdd} and the assumption on the differentiability of $\varphi_l$ in~\ref{i.assump_varphi_l_2^l_diff} are mainly used in deriving concentration results in Section~\ref{s.aux}. 
We believe that results similar to Theorem~\ref{t} are still valid under different or weaker assumptions. For instance, when $X$ is uniformly distributed on the centered $n$-sphere with radius $\sqrt{n}$, concentration results needed here are expected to hold. The high order of differentiability in~\ref{i.assump_varphi_l_2^l_diff} is needed in an iterative application of the Gaussian integration by parts due to the presence of multiple layers. We remark that in the $2$-layer setting, a careful treatment only needs $\varphi_1$ and $\varphi_2$ to be twice continuously differentiable with bounded derivatives, as done in \cite{gabrie2019entropy}, while~\ref{i.assump_varphi_l_2^l_diff} requires $\varphi_2$ to be continuously differentiable up to the fourth order. Since we are considering general cases, we resort to~\ref{i.assump_varphi_l_2^l_diff} for convenience.

On the other hand, many results in this work do not require assumptions as strong as~\ref{i.assump_iid_bdd} and~\ref{i.assump_varphi_l_2^l_diff}.
Hence, whenever possible, we will instead assume the following, together with~\ref{i.assump_Phi_Gaussian}:
\begin{enumerate}[start=1,label={\rm(h\arabic*)}]
    \item \label{i.assump_w_|X|<sqrt_n} for every $n\in\N$, $|X|\leq \sqrt{n}$ a.s.;
    \item
    \label{i.assump_w_varphi}
    for every $l\in\{1,\dots,L\}$, $\varphi_l$ is bounded and twice continuously differentiable with bounded derivatives.
\end{enumerate}

\subsection{Related works}

Generalized linear models are relevant in many fields including signal processing, statistical learning, and neural networks. Its multi-layer setup models a type of feed-forward neural network, which captures some of the key features of deep learning. For more details on these connections, we refer to \cite{barbier2019optimal, gabrie2019entropy} and references therein.
Recent progress in rigorous studies of information-theoretical aspects of these models has been made using methods originated from statistical physics. The mutual information of a model, a key quantity in these investigations, is related to the free energy via a simple additive relation. Therefore, the high-dimensional limit of the free energy is the central object in these approaches. Variational formulae for the free energy have been rigorously proven in the one-layer setting in \cite{barbier2019optimal} and the two-layer setting in \cite{gabrie2019entropy}.

The two works just mentioned above employed the powerful adaptive interpolation method introduced in \cite{barbier2019adaptive,barbier2019adaptive2}, which can be seen as an evolution from the classic interpolation method in statistical physics. This new method has proven to be successful and versatile in treating many different models and settings \cite{dia2016mutual,bmm,luneau2020high,luneau2019mutual,reeves2020information}.

The approach adopted in this work is based on identifying an enriched version of the original free energy with a solution to a certain Hamilton--Jacobi equation determined by the model. This approach was first introduced in \cite{HJinfer,HJrank} and has been applied also to the study of spin glass models \cite{parisi,mourrat2020extending,bipartite,mourrat2020free}. Similar considerations in physics also appeared in \cite{genovese2009mechanical, guerra2001sum,  barra2010replica, barra2013mean}. 

In treating statistical inference problems, two notions of solutions have been considered. One is the viscosity solution used in \cite{HJinfer,HBJ,chen2021statistical}, and the other is the weak solution in \cite{HJrank,HB1,HBJ}. In this paper, we take the latter approach due to the convenience and simplicity of dealing with boundary conditions under the notion of weak solutions.

\rv{Compared with \cite{HJrank,HB1,HBJ}, the novelty here lies in an iterative argument to treat the multi-layer setting, and the usage of the Hamilton--Jacobi equation technique as a black box.
On a technical level, different from \cite{HJrank,HB1,HBJ}, the equation considered here is defined over a domain where the range of spacial variables depends on time. Accordingly, treatments used previously have to be adjusted.
}

\rv{The feature of the Hamilton--Jacobi equation technique as a black box (input an initial data and output the solution) is crucial to this work and potentially to solving other models with iterative structures. We believe that other existing methods that lack this feature could be unwieldy in such applications. We hope that our work well demonstrates this feature.
}

The rest of the paper is organized as follows. In Section~\ref{s.approx}, we enrich the model and derive that the enriched free energy satisfies an approximate Hamilton--Jacobi equation. We also record some basic properties of the derivatives of free energy. In Section~\ref{s.hj}, we define weak solutions and prove their existence and uniqueness. In particular, the existence is furnished by a variational formula known as the Hopf formula. Using these, we prove the key convergence result of the enriched free energy in Section~\ref{s.cvg}, which is used in an iterative argument to prove Theorem~\ref{t}. Lastly, we collect auxiliary results in Section~\ref{s.aux}, including the convergence in~\eqref{e.rho_cvg}, concentration of the norm of $X^\pr{L}$, and concentration of the free energy.

\subsection*{Acknowledgement}

We warmly thank Jean-Christophe Mourrat for many helpful comments and discussions. \rv{We thank the referees for the comments that help us improve the paper considerably.}

\section{Approximate Hamilton--Jacobi equations}\label{s.approx}

In this section, we enrich the model and derive that the associated free energy satisfies an approximate Hamilton--Jacobi equation, which is stated in Proposition~\ref{p.approx_hj}. We also record basic properties of derivatives of the free energy in Lemma~\ref{l.der_est}.

\subsection{Enrichment}
Recall the notation $\R_+$ in~\eqref{e.R_+} and $\rho_{l,n}$ defined in~\eqref{e.rho_l}. For $\rho>0$, we define
\begin{align}\label{e.def_Omega_r}
    \Omega_\rho = \{(t,h_1,h_2)\in\R^3_+: h_1\leq \rho(1-t),\ t\leq 1   \}
\end{align}
where there is no restriction on $h_2$.
For $(t,h)\in \Omega_{\rho_{L-1,n}}$, define
\begin{gather}
    S = \sqrt{\frac{t}{n_{L-1}}}\Phi^\pr{L} X^\pr{L-1} + \sqrt{h_1}V + \sqrt{\rho_{L-1,n}- \rho_{L-1,n} t - h_1}W, \label{e.cap_S_mu}\\
    s = \sqrt{\frac{t}{n_{L-1}}}\Phi^\pr{L} x^\pr{L-1} + \sqrt{h_1}V + \sqrt{\rho_{L-1,n}- \rho_{L-1,n} t - h_1}w,  \label{e.s_mu}\\
    Y = \sqrt{\beta}\varphi_L\left(S,A^\pr{L}\right)+Z,\label{e.Y}\\
    Y' = \sqrt{h_2}X^\pr{L-1}+Z', \label{e.Y'}
\end{gather}
where $w\in \R^{n_L}$, $x^\pr{L-1}$ is given in~\eqref{e.x^L-1}, $V,W$ are independent $n_L$-dimensional standard Gaussian vectors, $Z$ is given in~\eqref{e.Y^circ}, and $Z'$ is an $n_{L-1}$-dimensional standard Gaussian vector. Due to~\eqref{e.x^L-1} and~\eqref{e.s_mu}, $s$ depends on $(x,w,a,\Phi^{[1,L]},V)$.

\rv{Let us explain the motivation. The variables $(t,h_1)$ allow us to interpolate between $\sqrt{\frac{1}{n_{L-1}}}\Phi^\pr{L}X^\pr{L-1}$ and Gaussian variables through $S$. Note that when $(t,h_1)=(1,0)$, we have $S = \sqrt{\frac{1}{n_{L-1}}}\Phi^\pr{L}X^\pr{L-1}$ and $Y=Y^\circ$ recovering the original noisy observation. In addition to $Y$, we add a scalar Gaussian channel $Y'$ with magnitude parameterized by $h_2$, which is a simple noisy observation to analyze when conditioned on $X^\pr{L-1}$. Varying these parameters, we can interpolate between the original model at $(1,0,0)$ and a simple model at $(0,\rho_{L-1,n},h_2)$.}

Recall $\Pout_{\beta,L,n}$ given in~\eqref{e.P_beta(y|x)}.
We introduce the following Hamiltonian
\begin{align}\label{e.H_beta_L_n_(t,h)}
    H_{\beta,L,n}(x,w,a) = \log \Pout_{\beta,L,n}\big(Y\big|s\big)+ \sqrt{h_2}Y'\cdot x^\pr{L-1}-\frac{h_2}{2}\left|x^\pr{L-1}\right|^2,
\end{align}
where $x\in\R^n$, $w\in\R^{n_L}$, $a$ and $x^\pr{L-1}$ are given in~\eqref{e.x^L-1}.
Define the associated partition function
\begin{align}\label{e.Z_n}
    \mathcal Z_{\beta,L,n} = \int e^{H_{\beta,L,n}(x,w,a)}\d P_X(x)\d P_W(w) \d P_{A^{[1,L-1]}}(a).
\end{align}
\rv{One can readily verify using Bayes' rule that
\begin{align*}
    \frac{1}{\mathcal Z_{\beta,L,n}}e^{H_{\beta,L,n}(x,w,a)}\d P_X(x)\d P_W(w) \d P_{A^{[1,L-1]}}(a)
\end{align*}
is the law of $(X,A^{[1,L-1]},W)$ conditioned on $(Y,Y',\Phi^{[1,L]})$, which is useful in deriving the Nishimori identity in Section~\ref{s.prelim}. Note that at $(t,h_1,h_2)=(1,0,0)$ this is the conditional law associated with the original model (see the display above~\eqref{e.P_beta(y|x)}).}
We consider the corresponding free energy
\begin{align}\label{e.F_beta_L_n_(t,h)}
    F_{\beta,L,n} = \frac{1}{n}\log\mathcal Z_{\beta,L,n}
\end{align}
and $\bar F_{\beta,L,n} = \E F_{\beta,L,n}$ where $\E$ is over $Y, Y', V, \Phi^{[1,L]}$ (recall that $x^\pr{L-1}$ depends on $\Phi^{[1,L-1]}$ as in~\eqref{e.x^L-1}).
The domain of $F_{\beta,L,n}$ is $\Omega_{\rho_{L-1,n}}$ defined in~\eqref{e.def_Omega_r}.

We often make the dependence of $F_{\beta,L,n}$ on $(t,h)\in \Omega_{\rho_{L-1,n}}$ explicit, and write $F_{\beta,L,n}(t,h)$. 
Comparing with the definitions of $\mathcal Z^\circ_{\beta,L,n}$ in~\eqref{e.Z^circ} and $F^\circ_{\beta,L,n}$ in~\eqref{e.F^circ}, we can verify that $\mathcal Z^\circ_{\beta,L,n}= \mathcal Z_{\beta,L,n}(1,0)$ and $F^\circ_{\beta,L,n}=F_{\beta,L,n}(1,0)$ evaluated at $t=1,h=0$. Hence, we view $F_{\beta,L,n}$ as the free energy associated with an enriched model. Note that the following holds
\begin{align}\label{e.EF^circ=bar_F}
    \E F^\circ_{\beta,L,n} = \bar F_{\beta,L,n}(1,0).
\end{align}

Throughout this work, we interpret $t$ as the ``temporal variable'' and $h=(h_1,h_2)$ as the ``spacial variable''. Moreover, we use the short hand notation $\partial_i = \partial_{h_i}$ for $i=1,2$, and denote by $\nabla = (\partial_1,\partial_2)$ the gradient operator. Define $\H_L:\R^2\to\R$ by
\begin{align}\label{e.H}
    \H_L(p) = \frac{2}{\alpha_{L-1}} p_1p_2.
\end{align}
The main goal is to prove the following proposition. 

\begin{proposition}\label{p.approx_hj}

Assume~\ref{i.assump_w_|X|<sqrt_n},~\ref{i.assump_w_varphi} and~\ref{i.assump_Phi_Gaussian} for some $L\in \N$.
For every $\beta\geq 0$ and every $n\in\N$, the function $(t,h)\mapsto \bF_{\beta,L,n}(t,h)$ is differentiable in $\Omega_{\rho_{L-1,n}}\setminus\{h_1 = \rho_{L-1,n}(1-t)\}$ and there is a constant $C$ such that, for all $(t,h)\in \Omega_{\rho_{L-1,n}}\setminus\{h_1 = \rho_{L-1,n}(1-t)\}$,
\begin{align*}
    \left|\partial_t \bar F_{\beta,L,n} - \H_L\left(\nabla \bar F_{\beta,L,n}\right)\right|\leq C\bigg(\frac{1}{n}\partial^2_2\bar F_{\beta,L,n} +  \E\big(\partial_2 F_{\beta,L,n} -  \partial_2\bar F_{\beta,L,n}\big)^2
\bigg)^\frac{1}{2}+a_n,
\end{align*}
where
\begin{align}\label{e.a_n_bd}
    a_n\leq C \left( n \E\left(\frac{\left|X^\pr{L-1}\right|^2}{n_{L-1}}-\rho_{L-1,n} \right)^2 \right)^\frac{1}{2}\left(\E\left(F_{\beta,L,n} - \bar F_{\beta,L,n}\right)^2\right)^\frac{1}{2} + C\left|\frac{n_{L-1}}{n}-\alpha_{L-1}\right|.
\end{align}

\end{proposition}
\rv{Since the right-hand side of~\eqref{e.a_n_bd} vanishes as $n\to\infty$ in a heuristic sense, this} suggests that the limiting Hamilton--Jacobi equation should be
\begin{align}\label{e.HJ}
    \partial_t f - \H_L(\nabla f) =0,
\end{align}
which will be studied in the next section.

\subsection{Proof of Proposition~\ref{p.approx_hj}}

Recall $\Pout_{\beta,L,n}$ defined in~\eqref{e.P_beta(y|x)}.
For simplicity of notation, we write $H=H_{\beta,L,n}$, $\mathcal{Z} = \mathcal{Z}_{\beta,L,n}$, $ F=F_{\beta,L,n}$, $\Pout=\Pout_{\beta,L,n}$ and $\rho = \rho_{L-1,n}$.
For any measurable function $g:\R^n\times \R^{n_L}\times (\prod_{l=1}^{L-1}\R^{n_l\times k_l})\to\R$, we define
\begin{align*}
    \la g(x,w,a)\ra = \frac{1}{\mathcal Z}\int g(x,w,a) e^{H(x,w,a)} \d P_X(x) \d P_W(w) \d P_{A^{[1,L-1]}}(a).
\end{align*}
In other words, $\la \, \cdot \, \ra$ is the Gibbs measure with Hamiltonian $H$ and reference measure $\d P_X(x) \d P_W(w) \d P_{A^{[1,L-1]}}(a)$.

\subsubsection{Preliminaries}\label{s.prelim}

We will repeatedly use two basic tools in our computations: the Gaussian integration by parts and the Nishimori identity. The simplest form of Gaussian integration by parts can be stated as follows. For a standard Gaussian random variable $U$ and a differentiable function $g:\R\to\R$ satisfying $\E|g'(U)|<\infty$, it holds that
\begin{align*}
    \E [Ug(U)] = \E g'(U),
\end{align*}
which can be seen easily by rewriting the expectation as an integration with respect to the Gaussian density and performing the classic integration by parts. For this work, a straightforward extension of the above to standard Gaussian vectors is sufficient. 

Using the definition of $H$ and Bayes' rule, we can see that the conditioned law of $X,W,A^{[1,L-1]}$ on $Y,Y',V,\Phi^{[1,L]}$ is given exactly by  the Gibbs measure $\la\,\cdot\,\ra$, namely,
\begin{align*}
    \la g\left(x,w,a,Y,Y',V,\Phi^{[1,L]}\right)\ra = \E \left[g\left(X,W,A^{[1,L-1]},Y,Y',V,\Phi^{[1,L]}\right)\Big| Y,Y',V,\Phi^{[1,L]}\right],
\end{align*}
for a suitable measurable function $g$. The above immediately implies the Nishimori identity that, for suitable $g$,
\begin{align*}
    \E \la g\left(x,w,a,Y,Y',V,\Phi^{[1,L]}\right)\ra = \E g\left(X,W,A^{[1,L-1]}, Y,Y',V,\Phi^{[1,L]}\right).
\end{align*}
Independent copies of $(x,w,a)$ with respect to the Gibbs measure are called replicas and are often denoted as $(x',w',a')$, $(x'',w'',a'')$, etc. When multiple replicas are present, the above identity can be extended in a straightforward way allowing us to replace one set of the replicas by $(X,W,A^{[1,L-1]})$, and vice versa. For instance, we have that
\begin{align*}
    \E \la g\left(x,w,a,x',w',a',Y,Y',V,\Phi^{[1,L]}\right)\ra = \E \la g\left(x,w,a,X,W,A^{[1,L-1]}, Y,Y',V,\Phi^{[1,L]}\right)\ra.
\end{align*}

\subsubsection{Computation of $\partial_t\bar F$}
Recall $H(x,w,a)$ in~\eqref{e.H_beta_L_n_(t,h)} and let us also write
\begin{align*}H(x,w,a;y,y') =  \log \Pout (y| s) + \sqrt{h_2}y'\cdot x^\pr{L-1} - \frac{h_2}{2}\left|x^\pr{L-1}\right|^2.
\end{align*}
Hence, we have that $H(x,w,a) = H(x,w,a;Y,Y')$, and for each fixed $x,w,y,y'$, the only randomness of $H(x,w,a;y,y')$ comes from $\Phi^{[1,L]}$ (in $s$ and $x^\pr{L-1}$) and $V$ (in $s$).

We can verify that the conditioned law of $(Y,Y')$ given $(\Phi^{[1,L]},V)$ is given by
\begin{align}\label{e.cond_law_Y_Y'}
    \left(\frac{1}{(2\pi)^\frac{n_L}{2}}\int  e^{H(x, w,a; y,y')}\d P_X(x)\d P_W(w)\d P_{A^{[1,L-1]}}(a)\right)\d y \d y',
\end{align}
where we recall that $W$ is Gaussian. Recall the partition function~\eqref{e.Z_n} and we introduce
\begin{align*}
    \mathcal Z(y,y') = \int  e^{H(x,w,a;y,y')}\d P_X(x)\d P_W(w)\d P_{A^{[1,L-1]}}(a).
\end{align*}
Then, note that $\mathcal Z = \mathcal Z(Y,Y')$ and the only randomness of $\mathcal Z(y,y')$ is from $\Phi^{[1,L]}$ and $V$.

We introduce the shorthand notation
\begin{align}\label{e.tilde_P}
    \d \tilde P_{y,y'} =\frac{1}{(2\pi)^\frac{n_L}{2}}\d y\,\d y'\, \d P_X(\tilde x)\,\d P_W(\tilde w)\,\d P_{A^{[1,L-1]}}(\tilde a)
\end{align}
which is a measure that integrates $y,y'$ and all variables with tildes $\tilde x, \tilde w,\tilde a$. Using these and~\eqref{e.F_beta_L_n_(t,h)}, we can write that
\begin{align}\label{e.bar_F_in_tilde_P}
    \bar F = \frac{1}{n}\E\bigg[\int e^{H(\tilde x,\tilde w,\tilde a;y,y')}\log \mathcal Z(y,y')\d \tilde P_{y,y'}\bigg]
\end{align}
where the expectation $\E$ is taken over the remaining randomness, namely, $\Phi^{[1,L]}$ and $V$.
To lighten the notation further, we write $H(\tilde\dash;y,y')= H(\tilde x,\tilde w,\tilde a;y,y')$ and $H(\dash;y,y')= H( x, w,a;y,y')$.

Due to the dependence of $H(\tilde\dash;y,y')$ and $\mathcal Z(y,y')$ on $t$, differentiating $\bar F$ as in~\eqref{e.bar_F_in_tilde_P} with respect to $t$ yields that
\begin{align}
    \partial_t \bar F &= \frac{1}{n}\E\bigg[\int \d \tilde P_{y,y'} \Big(\partial_t H(\tilde\dash; y,y')\Big)e^{H(\tilde\dash;y,y')}\log \mathcal Z(y,y')\bigg] \notag
    \\
    &\quad  + \frac{1}{n} \E \bigg[\Big\langle \partial_t H(\dash;y,y')\big|_{y=Y,\, y'=Y'} \Big\rangle \bigg]\notag
    \\
    &=  \mathtt I_t + \mathtt {II}_t. \label{e.def_II_t}
\end{align}
Here on the second line, the Gibbs measure is the one associated with the Hamiltonian~\eqref{e.H_beta_L_n_(t,h)} and thus only integrates over the variables $x,w,a$. To evaluate the above, we define
\begin{align}\label{e.u_y(x)}
    u_y(x) = \log \Pout(y|x)
\end{align}
and denote by $\nabla u_y$ and $\Delta u_y$ the gradient and Laplacian of $u_y$ with differentiation in $x$, respectively.
Then, using~\eqref{e.s_mu} and~\eqref{e.H_beta_L_n_(t,h)}, we can compute that
\begin{align}
    \partial_t H(\dash;y,y') &= \big(\partial_t s\big)\cdot \nabla u_{y}(s) \notag
    \\
    &= \frac{1}{2} \bigg(\frac{1}{\sqrt{tn_{L-1}}}\Phi^\pr{L} x^\pr{L-1}- \frac{\rho}{\sqrt{\rho(1-t)-h_1}}w\bigg)\cdot \nabla u_{y}(s).\label{e.d_t_H}
\end{align}
We write $\tilde s$ and $\tilde x^\pr{L-1}$ to be $s$ and $x^\pr{L-1}$, respectively, with $x,w,a$ therein replaced by $\tilde x,\tilde w,\tilde a$.
Hence, we have that $\mathtt I_t$ is equal to
\begin{align*}
    \frac{1}{2n}\E\bigg[\int \d \tilde P_{y,y'} \bigg(\frac{1}{\sqrt{tn_{L-1}}}\Phi^\pr{L} \tilde x^\pr{L-1}- \frac{\rho}{\sqrt{\rho(1-t)-h_1}}\tilde w\bigg)\cdot \nabla u_{y}(\tilde s)e^{H(\tilde\dash;y,y')}\log \mathcal Z(y,y')\bigg].
\end{align*}
Recall that $\tilde s$ and $H(\tilde\dash;y,y')$ depend on $\Phi^\pr{L}$ and $\tilde w$, and that $\mathcal Z(y,y')$ depends on $\Phi^\pr{L}$.
Since $\tilde w$ under $\d \tilde P_{y,y'}$ and $\Phi^\pr{L}$ under $\E$ are standard Gaussian vectors, we \rv{perform the Gaussian integration by parts on $\mathtt I_t$ with respect to $\tilde w$ and $\Phi^\pr{L}$ inside the parentheses of the above display. We label by $a'_n$ collectively the terms with $\log \mathcal Z(y,y')$ intact. Hence, we get}
\begin{align}
    \mathtt I_t  = a'_n +\frac{1}{2n} \E\bigg[\frac{1}{\mathcal Z(y,y')}\int \d \tilde P_{y,y'}\d P_X(x)\d P_W(w)\d P_{A^{[1,L-1]}}(a) \notag
    \\
    \left(\frac{1}{n_{L-1}}\tilde x^\pr{L-1} \cdot x^\pr{L-1}\right)\big(\nabla u_{y}(\tilde s)\cdot \nabla u_{y}(s) \big) e^{H(\tilde\dash;y,y')}  e^{H(\dash;y,y')}\bigg] \notag
    \\
    =a'_n+\frac{1}{2}\E\bigg \langle \bigg(\frac{1}{n_{L-1}}X^\pr{L-1}\cdot x^\pr{L-1}\bigg) \bigg(\frac{1}{n} \nabla u_{Y}(S)\cdot \nabla u_{Y}(s)\bigg)\bigg\rangle \label{e.I_t_comp}
\end{align}
where
\begin{align}
    a'_n & = \frac{1}{2n}\E\left[\int \d \tilde P_{y,y'} \left(\frac{1}{n_{L-1}} \left|\tilde x^\pr{L-1}\right|^2-\rho\right)\big(\Delta u_{y}(\tilde s) +|\nabla u_{y}(\tilde s)|^2\big) e^{H(\tilde\dash;y,y')}\log \mathcal Z(y,y')\right]\notag\\
    & = \frac{1}{2n}\E\left[\left(\frac{1}{n_{L-1}} \left|X^\pr{L-1}\right|^2-\rho\right)\big(\Delta u_{Y}(S) +|\nabla u_{Y}(S)|^2\big)\log \mathcal Z(Y,Y')\right]\label{e.a'}.
\end{align}
Here, in deriving~\eqref{e.I_t_comp} and~\eqref{e.a'}, we used~\eqref{e.tilde_P} and the observation that replacing $\tilde x,\tilde w,\tilde a$ by $X,W,A^{[1,L-1]}$ in $\tilde x^\pr{L-1}, \tilde s$ yields $X^\pr{L-1}, S$.
We claim that $\mathtt{II}_t=0$ and postpone its proof \rv{to Section~\ref{s.eval_II_t_II_h_1}}. Then, combining the above gives that
\begin{align}
    \partial_t \bar F = \frac{1}{2}\E\bigg \langle \bigg(\frac{1}{n_{L-1}}X^\pr{L-1}\cdot x^\pr{L-1}\bigg) \bigg(\frac{1}{n} \nabla u_{Y}(S)\cdot \nabla u_{Y}(s)\bigg)\bigg\rangle +a'_n.\label{e.d_tF_n}
\end{align}

\subsubsection{Computation of $\partial_1\bar F$}

Similarly, by~\eqref{e.bar_F_in_tilde_P}, we have that
\begin{align}
    \partial_1 \bar F &= \frac{1}{n}\E\bigg[\int \d \tilde P_{y,y'}\Big(\partial_1 H(\tilde\dash; y,y')\Big)e^{H(\tilde\dash;y,y')}\log \mathcal Z(y,y')\bigg] \notag
    \\
    &\quad+\frac{1}{n} \E \bigg[\Big\langle \partial_1 H(\dash;y,y')\big|_{y=Y,\, y'=Y'} \Big\rangle \bigg] \notag
    \\
    &=  \mathtt I_{h_1} + \mathtt {II}_{h_1}. \label{e.II_h1}
\end{align}
To compute $\mathtt I_{h_1}$, we start with
\begin{align}\label{eqn.partialh1H}
    \partial_1 H(\dash;y,y') &=   \big(\partial_1 s\big)\cdot \nabla u_{y}(s) \nonumber\\
    &= \frac{1}{2} \bigg(\frac{1}{\sqrt{h_1}}V- \frac{1}{\sqrt{\rho(1-t)-h_1}}w\bigg)\cdot \nabla u_{y}(s),
\end{align}
which gives that
\begin{align*}\mathtt I_{h_1} = \frac{1}{2n}\E\bigg[\int \d \tilde P_{y,y'}\bigg(\frac{1}{\sqrt{h_1}}V- \frac{1}{\sqrt{\rho(1-t)-h_1}}\tilde w\bigg)\cdot \nabla u_{y}(\tilde s)e^{H(\tilde\dash;y,y')}\log \mathcal Z(y,y')\bigg].
\end{align*}
Using Gaussian integration by parts on $V$ and $\tilde w$, we obtain that
\begin{align*}
    \mathtt I_{h_1} = \frac{1}{2n}\E\bigg[\int \d \tilde P_{y,y'}\big(1-1\big)\Big(\Delta u_{y}(\tilde s) +|\nabla u_{y}(\tilde s)|^2\Big) e^{H(\tilde\dash;y,y')}\log \mathcal Z(y,y')\bigg]
    \\
    +\frac{1}{2n} \E\bigg[\frac{1}{\mathcal Z(y,y')}\int \d \tilde P_{y,y'}\d P_X(x)\d P_W(w)\d P_{A^{[1,L-1]}}(a)
    \\  
    \qquad\qquad\nabla u_{y}(\tilde s)\cdot \nabla u_{y}(s)e^{H(\tilde\dash;y,y')}e^{H(\dash;y,y')}\bigg] 
    \\
    = \frac{1}{2n} \E \la \nabla u_{Y}(S)\cdot\nabla u_{Y}(s) \ra.
\end{align*}
Here, in the last equality, we used the same argument as in obtaining~\eqref{e.I_t_comp}.
Again, we claim that $\mathtt{II}_{h_1}=0$ and postpone its proof \rv{to Section~\ref{s.eval_II_t_II_h_1}}. This together with the above yields that
\begin{align}\label{e.d_h_1F_n}
    \partial_1\bar F = \frac{1}{2} \E \la  \frac{1}{n}\nabla u_{Y}(S)\cdot\nabla u_{Y}(s) \ra.  
\end{align}

\subsubsection{Computation of $\partial_2 \bar F$}
Using~\eqref{e.H_beta_L_n_(t,h)},~\eqref{e.Z_n} and~\eqref{e.F_beta_L_n_(t,h)}, we can compute that
\begin{align}\label{e.1st_der_F}
    &\partial_2 F = \frac{1}{n}\big\langle \partial_2 H(x,w,a) \big\rangle 
    = \frac{1}{2n}\bigg\langle 2X^\pr{L-1}\cdot x^\pr{L-1} + \frac{1}{\sqrt{h_2}}Z'\cdot x^\pr{L-1} -x^\pr{L-1}\cdot x^\pr{L-1}\bigg\rangle.
\end{align}
Using Gaussian integration by parts on $Z'$ and the Nishimori identity, we get that
\begin{align}
    \partial_2 \bar F &=\frac{1}{2n}\E \bigg\langle 2X^\pr{L-1}\cdot x^\pr{L-1} + \left( x^\pr{L-1}-{x'}^\pr{L-1}\right)\cdot x^\pr{L-1} -x^\pr{L-1}\cdot x^\pr{L-1} \bigg\rangle  \notag\\
    &= \frac{1}{2n}\E \Big\langle X^\pr{L-1}\cdot x^\pr{L-1} \Big\rangle, \label{e.d_h_2F_n}
\end{align}
where ${x'}^\pr{L-1}$ is a replica of $x^\pr{L-1}$ obtained by replacing $x,a$ in~\eqref{e.x^L-1} by replicas $x',a'$.

\subsubsection{Deriving the equation}
By~\eqref{e.H},~\eqref{e.d_h_1F_n} and~\eqref{e.d_h_2F_n}, we have
\begin{align*}
    \left|\H_L(\nabla \bar F) - \frac{1}{2}\E\la \frac{1}{n_{L-1}}X^\pr{L-1}\cdot x^\pr{L-1}\ra \E \la \frac{1}{n} \nabla u_{Y}(S)\cdot \nabla u_{Y}(s)\ra\right|\\
    = \left|1-\frac{\alpha_{L-1}n}{n_{L-1}}\right|\left|\H_L(\nabla \bar F)\right|.
\end{align*}
By~\eqref{e.partial_1_F_n} and~\eqref{e.partial_2_F_n} both proved later and assumption~\eqref{e.n_l/n}, the above is bounded by $C|\frac{n_{L-1}}{n}-\alpha_{L-1}|$. This along with~\eqref{e.d_tF_n} implies that
\begin{align*}
    \left| \partial_t \bar F- \H_L(\nabla \bar F)\right|\leq \frac{1}{2} \sqrt{b_n}+ |a'_n| + C\left|\frac{n_{L-1}}{n}-\alpha_{L-1}\right|.
\end{align*}
where
\begin{align*}
    b_n = \var_{\E\la\,\cdot\,\ra}\left[ \frac{1}{n_{L-1}}X^\pr{L-1}\cdot x^\pr{L-1}\right]\var_{\E\la\,\cdot\,\ra} \left[\frac{1}{n} \nabla u_{Y}(S)\cdot \nabla u_{Y}(s)\right]
\end{align*}
with variances taken with respect to $\E\la\,\cdot\,\ra$. Then, the desired result follows, once we prove that
\begin{gather}
    |a'_n| \leq C \left( n \E\left(\frac{\left|X^\pr{L-1}\right|^2}{n_{L-1}}-\rho \right)^2 \right)^\frac{1}{2} \left(\E\left(F - \bar F\right)^2\right)^\frac{1}{2},   \label{e.a'_n_est}
    \\
    \var_{\E\la\,\cdot\,\ra} \left[\frac{1}{n} \nabla u_{Y}(S)\cdot \nabla u_{Y}(s)\right] \leq C,       \label{e.var(u)}
    \\
    \var_{\E\la\,\cdot\,\ra}\left[ \frac{1}{n_{L-1}}X^\pr{L-1}\cdot x^\pr{L-1}\right] \leq C\left( \frac{1}{n}\partial^2_2\bar F +  \E\big(\partial_2 F -  \partial_2\bar F\big)^2\right).        \label{e.var(x)}
\end{gather}
To complete the proof, it remains to verify that $\mathtt{II}_t = \mathtt{II}_{h_1}=0$ and prove the above assertions.

\subsubsection{Evaluating $\mathtt{II}_t$ and $\mathtt{II}_{h_1}$}\label{s.eval_II_t_II_h_1}
Recall the definition of $\mathtt{II}_t$ in~\eqref{e.def_II_t}.
By the Nishimori identity, we have that
\begin{align*}
    \mathtt {II}_{t}=\frac{1}{n}\E\la\partial_t H(x,w,a;y,y')|_{y=Y,\,y'=Y'}\ra=\frac{1}{n}\E\left[\partial_t H\left(X,W,A^{[1,L-1]};y,y'\right)\Big|_{y=Y,\,y'=Y'}\right].
\end{align*}
Using~\eqref{e.d_t_H} and the conditional law of $(Y,Y')$ in~\eqref{e.cond_law_Y_Y'} together with the notation $\d\tilde P_{y,y'}$ given in~\eqref{e.tilde_P}, we obtain that
\begin{align*}
    \mathtt {II}_{t}&=\frac{1}{2n}\E \left[\left(\frac{1}{\sqrt{tn_{L-1}}}\Phi^\pr{L} X^\pr{L-1}-\frac{\rho}{\sqrt{\rho(1-t)-h_1}}W\right)\cdot\nabla u_{Y}(S)\right]\\
    &=\frac{1}{2n}\E \left[\int \d \tilde P_{y,y'}e^{H(\tilde\dash;y,y')}\left(\frac{1}{\sqrt{tn_{L-1}}}\Phi^\pr{L} \tilde x^\pr{L-1}-\frac{\rho}{\sqrt{\rho(1-t)-h_1}}\tilde w\right)\cdot \nabla u_{y}(\tilde s)\right]\\
    &=\frac{1}{2n}\E \Bigg[\int \d \tilde P_{y,y'}e^{H(\tilde\dash;y,y')}\left(\frac{1}{n_{L-1}}\left|\tilde x^\pr{L-1}\right|^2-\rho\right)\big(\Delta u_{y}(\tilde s)+\left|\nabla u_{y}(\tilde s)\right|^2\big)\Bigg]
\end{align*}
where in the third equality we used the Gaussian integration by parts on $\Phi^\pr{L}$ and $\tilde w$ (recall that under $\d\tilde P_{y,y'}$, $\tilde w$ is a standard Gaussian vector).

Due to the definition of $u_y$ in~\eqref{e.u_y(x)}, we can compute that
\begin{align}\label{e.lap_u}
    \Delta u_{y}(\tilde s)+\left|\nabla u_{y}(\tilde s)\right|^2 = \frac{\Delta\Pout(y|\tilde s)}{\Pout(y|\tilde s)},
\end{align}
where we recall that all derivatives are carried out in the second argument.
Hence, we get that
\begin{align}\label{e.II_t}
    \mathtt{II}_t =\frac{1}{2n}\E \left[ \left(\frac{1}{n_{L-1}}\left| X^\pr{L-1}\right|^2-\rho\right)\E\left[\frac{\Delta\Pout(Y|S)}{\Pout(Y|S)}\bigg|X^\pr{L-1},S\right]\right].
\end{align}
In view of the definition of $Y$ in~\eqref{e.Y} and the formula for $\Pout$ in~\eqref{e.P_beta(y|x)}, we can see that, conditioned on $X^\pr{L-1},S$, the law of $Y$ has a Lebesgue density given by $(2\pi)^{-\frac{n_L}{2}}\Pout(y|S)$, namely, for any bounded measurable function $g$,
\begin{align}\label{e.cond_Y_on_S}
    \E \left[ g\left(Y,X^\pr{L-1},S\right)\Big|X^\pr{L-1},S\right] = \frac{1}{(2\pi)^\frac{n_L}{2}}\int g\left(y,X^\pr{L-1},S\right)\Pout(y|S)\d y.
\end{align}
Let us write
\begin{align}\label{e.d^2_Pout(y|S)}
    \Delta\Pout(y|S) =\sum_{j=1}^{n_L}\partial^2_j\Pout(y|S)
\end{align}
where again the derivatives are in the second argument. We can compute that
\begin{align}\label{e.d^2_Pout(y|S)_explicit}
    \partial^2_j\Pout(y|S) = \int \Gamma_j\left(y_j,S_j,a^\pr{L}_j\right) e^{-\frac{1}{2}|y-\sqrt{\beta}\varphi_L(S,a^\pr{L})|^2}\d P_{A^\pr{L}}\left(a^\pr{L}\right)
\end{align}
with
\begin{align}\label{e.Gamma_j(y_j,S_j,a_j)}
    \Gamma_j\left(y_j,S_j,a^\pr{L}_j\right) = \beta \left(\left(y_j-\sqrt{\beta}\varphi_j \right)^2-1\right)\left(\varphi'_j \right)^2+ \sqrt{\beta} \left(y_j-\sqrt{\beta}\varphi_j \right) \varphi''_j 
\end{align}
where we used the shorthand notation $\varphi_j = \varphi_L(S_j,a^\pr{L}_j)$, $\varphi'_j = \varphi'_L(S_j,a^\pr{L}_j)$, $\varphi''_j = \varphi''_L(S_j,a^\pr{L}_j)$. Recall that $\varphi_L$ acts component-wise on $(S,a^\pr{L})$, namely, $\varphi_L(S,a^\pr{L}) = (\varphi_L(S_j,a^\pr{L}_j))_{1\leq j\leq n_L}$. Using this and the assumption that $(A_j^\pr{L})_{1\leq j\leq n_L}$ are i.i.d.\ as in~\ref{i.assump_Phi_Gaussian}, we have that
\begin{align}\label{e.d_jP/P}
    \frac{\partial^2_j\Pout(y|S)}{\Pout(y|S)} = \frac{ \int \Gamma_j\left(y_j,S_j,a^\pr{L}_j\right) e^{-\frac{1}{2}|y_j-\sqrt{\beta}\varphi_L(S_j,a^\pr{L}_j)|^2}\d P_{A^\pr{L}_j}\left(a^\pr{L}_j\right)}{ \int  e^{-\frac{1}{2}|y_j-\sqrt{\beta}\varphi_L(S_j,a^\pr{L}_j)|^2}\d P_{A^\pr{L}_j}\left(a^\pr{L}_j\right)}.
\end{align}
Using this,~\eqref{e.d^2_Pout(y|S)_explicit} and~\eqref{e.Gamma_j(y_j,S_j,a_j)}, we can see that
\begin{align}
    &\frac{1}{(2\pi)^\frac{n_L}{2}}\int\frac{\partial^2_j\Pout(y|S)}{\Pout(y|S)} \Pout(y|S) \d y = 0,\label{e.d_jP/P_mean_0}
    \\
    &\frac{1}{(2\pi)^\frac{n_L}{2}}\int \frac{\partial^2_i\Pout(y|S)}{\Pout(y|S)}\frac{\partial^2_j\Pout(y|S)}{\Pout(y|S)} \Pout(y|S) \d y  = 0, \quad i\neq j. \label{e.d_jP/P_cov_0}
\end{align}
The second equation will be used later. Now, by~\eqref{e.cond_Y_on_S} and~\eqref{e.d_jP/P_mean_0}, we have that
\begin{align*}\E\left[\frac{\partial_j^2\Pout(Y|S)}{\Pout(Y|S)}\bigg|X^\pr{L-1},S\right]= 0,\quad\forall j\in\{1,\dots,n_L\},
\end{align*}
which together with~\eqref{e.II_t} implies that $\mathtt{II}_t=0$.

It remains to show $\mathtt{II}_{h_1}=0$. Recall the definition of $\mathtt{II}_{h_1}$ in~\eqref{e.II_h1}. The Nishimori identity gives that
\begin{align*}
    \mathtt{II}_{h_1} = \frac{1}{n}\E\left[\partial_1H\left(X,W,A^{[1,L-1]};y,y'\right)\Big|_{y= Y,\,y'=Y'}\right].
\end{align*}
Using~\eqref{eqn.partialh1H} and a similar argument used above, we have that
\begin{align*}
    \mathtt{II}_{h_1} &= \frac{1}{2n}\E\bigg[\int \d \tilde P_{y,y'}\bigg(\frac{1}{\sqrt{h_1}}V- \frac{1}{\sqrt{\rho(1-t)-h_1}}\tilde w\bigg)\cdot \nabla u_{y}(\tilde s)e^{H(\tilde\dash;y,y')}\bigg]
    \\
    &= \frac{1}{2n}\E\left[\int \d \tilde P_{y,y'}\left(1-1\right)\big(\Delta u_{y}(\tilde s)+\left|\nabla u_{y}(\tilde s)\right|^2\big) e^{H(\tilde\dash;y,y')}\right]=0
\end{align*}
where the second equality follows from the Gaussian integration by parts applied to $ V$ and $\tilde w$.

\subsubsection{Proof of~\eqref{e.a'_n_est}}

Using~\eqref{e.F_beta_L_n_(t,h)} and a computation similar to~\eqref{e.lap_u}, we rewrite $a'_n$ in~\eqref{e.a'} as
\begin{align*}
    a'_n = \frac{1}{2} \E\Bigg[\bigg(\frac{1}{n_{L-1}}\left|X^\pr{L-1}\right|^2 - \rho\big)\bigg)\frac{\Delta\Pout(Y|S)}{\Pout(Y|S)}F\Bigg].
\end{align*}
Since $\mathtt{II}_t =0 $ as shown above, using the formula~\eqref{e.II_t}, we then have
\begin{align*}
    a'_n = \frac{1}{2} \E\Bigg[\bigg(\frac{1}{n_{L-1}}\left|X^\pr{L-1}\right|^2 - \rho\big)\bigg)\frac{\Delta\Pout(Y|S)}{\Pout(Y|S)}\left(F-\bar F\right)\Bigg].
\end{align*}
By the Cauchy--Schwarz inequality, 
\begin{align}\label{e.|a'_n|_first_bdd}
    |a'_n|\leq \frac{1}{2} \left(\E\left[\left(\frac{1}{n_{L-1}}\left|X^\pr{L-1}\right|^2 - \rho\big)\right)^2\left(\frac{\Delta\Pout(Y|S)}{\Pout(Y|S)}\right)^2\right]\right)^\frac{1}{2} \left(\E\left(F-\bar F\right)^2\right)^\frac{1}{2}.
\end{align}
Now, to prove~\eqref{e.a'_n_est}, it suffices to bound the first expectation on the right. 

By~\eqref{e.cond_Y_on_S}, we have that
\begin{align*}
    \E\left[\left(\frac{\Delta\Pout(Y|S)}{\Pout(Y|S)}\right)^2\Bigg|X^\pr{L-1},S\right] = \frac{1}{(2\pi)^\frac{n_{L-1}}{2}}\int \left(\frac{\Delta\Pout(y|S)}{\Pout(y|S)}\right)^2 \Pout(y|S)\d y.
\end{align*}
Recall the notation~\eqref{e.d^2_Pout(y|S)}. Then, ~\eqref{e.d_jP/P_cov_0} implies that
\begin{align*}
    \E\left[\left(\frac{\Delta\Pout(Y|S)}{\Pout(Y|S)}\right)^2\Bigg|X^\pr{L-1},S\right] =  \frac{1}{(2\pi)^\frac{n_{L-1}}{2}}\sum_{j=1}^{n_L}\int \left(\frac{\partial_j^2\Pout(y|S)}{\Pout(y|S)}\right)^2 \Pout(y|S)\d y.
\end{align*}
Using Jensen's inequality to the integral in~\eqref{e.d_jP/P}, we have that
\begin{align*}
    \int \left(\frac{\partial_j^2\Pout(y|S)}{\Pout(y|S)}\right)^2 \Pout(y|S)\d y \leq \int \left(\Gamma_j\left(y_j,S_j,a^\pr{L}_j\right)\right)^2 e^{-\frac{1}{2}|y-\sqrt{\beta}\varphi_L(S,a)|^2}\d P_{A^\pr{L}}(a)\d y.
\end{align*}
By the boundedness assumption in~\ref{i.assump_w_varphi} and the formula for $\Gamma_j$ in~\eqref{e.Gamma_j(y_j,S_j,a_j)}, we obtain that
\begin{align*}
    \E\left[\left(\frac{\Delta\Pout(Y|S)}{\Pout(Y|S)}\right)^2\Bigg|X^\pr{L-1},S\right]\leq Cn_L,
\end{align*}
which implies that
\begin{align*}
    \E\left[\left(\frac{1}{n_{L-1}}\left|X^\pr{L-1}\right|^2 - \rho\right)^2\left(\frac{\Delta\Pout(Y|S)}{\Pout(Y|S)}\right)^2\right] \leq Cn_L \E \left(\frac{1}{n_{L-1}}\left|X^\pr{L-1}\right|^2 - \rho\right)^2.
\end{align*}
Inserting this to~\eqref{e.|a'_n|_first_bdd} yields~\eqref{e.a'_n_est}.

\subsubsection{Proof of~\eqref{e.var(u)}}
Recalling the definitions of $u$ in~\eqref{e.u_y(x)} and $\Pout$ in~\eqref{e.P_beta(y|x)}, we can see that
\begin{align*}
    \nabla u_{Y}(s)=\left(\frac{\int \left(Y_j - \varphi_L\left(s_j,a^\pr{L}_j\right)\right)\varphi_L'\left(s_j,a^\pr{L}_j\right) e^{-\frac{1}{2}|Y - \sqrt{\beta}\varphi_L(s,a^\pr{L})|^2 }\d P_{A^\pr{L}}\left(a^\pr{L}\right)}{\int e^{-\frac{1}{2}|Y - \sqrt{\beta}\varphi_L(s,a^\pr{L})|^2 }\d P_{A^\pr{L}}\left(a^\pr{L}\right)}\right)_{1\leq j\leq n_L},
\end{align*}
where $\varphi'$ is the derivative with respect to its first argument. Recall the definition of $Y$ in~\eqref{e.Y}. Using the boundedness of $\varphi_L$ and its derivatives ensured by~\ref{i.assump_w_varphi}, we can see that
\begin{align}\label{e.dot_u_bdd}
    |\nabla u_Y(s)|\leq C(\sqrt{n_L}+|Z|).
\end{align}
This computation also gives that
\begin{align*}
    |\nabla u_Y(S)|\leq C(\sqrt{n_L}+|Z|)
\end{align*}
which together with~\eqref{e.dot_u_bdd} verifies~\eqref{e.var(u)}.

\subsubsection{Proof of~\eqref{e.var(x)}}

For simplicity, we write
\begin{align}\label{e.bar_X_notation}
    \bar X = X^\pr{L-1},\qquad \bar x = x^\pr{L-1}.
\end{align}
Using the formula for $\partial_2 F$ in~\eqref{e.1st_der_F}, we can compute that
\begin{align}\label{e.2nd_der_F}
    n \partial^2_2 F & = \la \left(\partial_2 H(x,w,a)\right)^2\ra - \la \partial_2 H(x,w,a)\ra^2 -\frac{1}{4h_2^\frac{3}{2}}\la Z'\cdot \bar x\ra
\end{align}
Inserting~\eqref{e.1st_der_F} into the second term on the right and applying the Gaussian integration by parts to the last term, we obtain that
\begin{align}\label{e.2nd_der_bF}
    n \partial^2_2 \bar F= \E\la (\partial_2 H(x,w,a))^2\ra - n^2\E(\partial_2 F)^2 -\frac{1}{4h_2}\E\la|\bar x|^2\ra + \frac{1}{4h_2}\E|\la \bar x\ra|^2,
\end{align}
where, to get the last term, we also invoked the Nishimori identity.
We claim that
\begin{align}\label{e.key_lower_bd_approx_hj}
    \E\la (\partial_2 H(x,w,a))^2\ra \geq \frac{1}{4} \E\la (\bar x\cdot\bar x')^2\ra  + \frac{1}{4h_2}\E\la|\bar x|^2\ra,
\end{align}
and postpone its proof. Now, insert~\eqref{e.key_lower_bd_approx_hj} into~\eqref{e.2nd_der_bF} to see that
\begin{align*}
    n\partial^2_2 \bar F\geq \frac{1}{4}\E\la (\bar x\cdot\bar x')^2\ra -   n^2\E(\partial_2 F)^2. 
\end{align*}
By~\eqref{e.d_h_2F_n}, we have that
\begin{align*}
    \var_{\E\la\,\cdot\,\ra}\left[X^\pr{L-1}\cdot x^\pr{L-1}\right] = \E \la (\bar x\cdot \bar x')^2\ra - (\E\la \bar x\cdot\bar x'\ra)^2 = \E \la (\bar x\cdot\bar x')^2\ra - 4n^2 (\partial_2 \bF)^2.
\end{align*}
Then,~\eqref{e.var(x)} follows from the above two displays along with~\eqref{e.n_l/n}.

It remains to derive~\eqref{e.key_lower_bd_approx_hj}. Using the expression of $\partial_2 H$ in ~\eqref{e.1st_der_F}, we have that
\begin{align}\label{e.E(partial_H)^2}
\begin{split}
    &\quad\E\la (\partial_2 H(x,w,a))^2\ra = \E \la \bigg( \frac{1}{2\sqrt{h_2}}Z'\cdot \bar x + \bar x\cdot \bar X-\frac{1}{2
    }|\bar x|^2\bigg)^2\ra\\
    &=\E \la \frac{1}{4h_2 }(Z'\cdot \bar x)^2  + (\bar x\cdot \bar X)^2+\frac{1}{4}|\bar x|^4 +\frac{1}{\sqrt{h_2}}(Z'\cdot \bar x)(\bar x\cdot \bar X)-\frac{1}{2\sqrt{h_2}}(Z'\cdot \bar x)|\bar x|^2-(\bar x\cdot \bar X)|\bar x|^2 \ra
\end{split}
\end{align}
The first term on the last line can be rewritten as 
\begin{align*}
    \E\la\frac{1}{4 h_2}(Z'\cdot \bar x)^2 \ra =  \sum_{i,j=1}^{n_{L-1}}\frac{1}{4h_2}\E \la Z'_iZ'_j\bar x_i\bar x_j\ra.
\end{align*}
If $i\neq j$, the Gaussian integration by parts yields that
\begin{align*}
    \frac{1}{h_2}\E \la Z'_iZ'_j\bar x_i\bar x_j\ra = \E \la \bar x_i\bar x_j(\bar x_i-\bar x'_i)(\bar x_j+\bar x'_j-2\bar x''_j)\ra.
\end{align*}
If $i=j$, we have that
\begin{align*}
    \frac{1}{ h_2}\E \la Z'_iZ'_i\bar x_i\bar x_i\ra = \E \la \bar x_i\bar x_i(\bar x_i-\bar x'_i)(\bar x_i+\bar x'_i-2\bar x''_i)\ra + \frac{1}{h_2}\E \la \bar x^2_i\ra.
\end{align*}
The above three displays combined give that
\begin{align*}
    \E\la\frac{1}{4h_2}(Z'\cdot \bar x)^2 \ra  = \frac{1}{4}\E \la |\bar x|^4 -2 |\bar x|^2(\bar x\cdot \bar x')-(\bar x\cdot \bar x')^2 +2(\bar x\cdot \bar x')(\bar x\cdot\bar x'')\ra + \frac{1}{4h_2}\E \la |\bar x|^2 \ra.
\end{align*}
Other terms can be computed using the Nishimori identity and the Gaussian integration by parts. We shall omit the details but only list the results:
\begin{align*}
    &\E\la (\bar x\cdot\bar X)^2\ra = \E\la (\bar x\cdot\bar x')^2\ra,\\
    &\E\la \frac{1}{\sqrt{h_2}}(Z'\cdot \bar x)(\bar x\cdot\bar X) \ra = \E\la |\bar x|^2(\bar x\cdot\bar x')-(\bar x\cdot\bar  x')(\bar x\cdot\bar x'')\ra,\\
    &\E\la \frac{1}{\sqrt{h_2}}(Z'\cdot \bar x)|\bar x|^2\ra = \E\la |\bar x|^4 - |\bar x|^2(\bar x\cdot\bar x')\ra,\\
    &\E\la (\bar x\cdot\bar X)|\bar x|^2 \ra = \E\la |\bar x|^2(\bar x\cdot\bar x')\ra.
\end{align*}
Inserting these computations into~\eqref{e.E(partial_H)^2} yields that
\begin{align*}
    \E\la (\partial_2 H(x,w,a))^2\ra = \frac{1}{4}\E\la (\bar x\cdot\bar x')^2\ra + \frac{1}{2}\E\la (\bar x\cdot \bar x')^2 - (\bar x\cdot\bar x')(\bar x\cdot\bar x'')\ra + \frac{1}{4h_2}\E\la|\bar x|^2\ra.
\end{align*}
Apply the Cauchy--Schwarz inequality and the symmetry of replicas to see that
\begin{align*}
    \E\la (
    \bar x\cdot\bar x')(\bar x\cdot\bar x'')\ra \leq \frac{1}{2}\E\la (\bar x\cdot\bar x')^2\ra+ \frac{1}{2}\E\la (\bar x\cdot\bar x'')^2\ra = \E\la (\bar x\cdot\bar x')^2\ra .
\end{align*}
These two displays imply~\eqref{e.key_lower_bd_approx_hj}.

\subsection{Estimates of derivatives}

We collect useful properties of derivatives of $\bar F_{\beta,L,n}$ and $F_{\beta,L,n}$ in the following lemma.

\begin{lemma}\label{l.der_est}
Assume~\ref{i.assump_w_|X|<sqrt_n},~\ref{i.assump_w_varphi} and~\ref{i.assump_Phi_Gaussian} for some $L\in \N$. For every $\beta\geq 0$ and every $n\in\N$, there is a constant $C$ such that the following holds for all $n\in\N$ and all $(t,h)\in\Omega_{\rho_{L-1,n}}\setminus\{h_1 = \rho_{L-1,n}(1-t)\}$:
\begin{gather}
    \partial_1 \bar F_{\beta,L,n}\in [0,C], \label{e.partial_1_F_n}\\
    \partial_2\bar F_{\beta,L,n} \in \left[0, \frac{n_{L-1}\rho_{L-1,n}}{2n}\right]\subset [0,C]\label{e.partial_2_F_n},\\
    |\partial_2 F_{\beta,L,n}|\leq C\left(1+ n^{-\frac{1}{2}}h_2^{-\frac{1}{2}}|Z'|\right);\label{e.1st_der_F_n_est}\\
    \partial_{i}\partial_{j}\bar F_{\beta,L,n} \geq 0, \quad\forall i,j=1,2, \label{e.2nd_d_barF_lbd}\\
    \partial^2_2 F_{\beta,L,n} \geq -Cn^{-\frac{1}{2}}h_2^{-\frac{3}{2}}|Z'|,\label{e.2nd_d_F_lbd}
\end{gather}
where $Z'$ is \rvv{the} $n_{L-1}$-dimensional standard Gaussian vector appearing in~\eqref{e.Y'}.

\end{lemma}

Let us prove these assertions. Again, for simplicity, we write $F=F_{\beta,L,n}$ in the proofs below.
\subsubsection{Proof of~\eqref{e.partial_1_F_n}}
By~\eqref{e.d_h_1F_n} and the Nishimori identity, we can see that
\begin{align*}
    \partial_1\bar F = \frac{1}{2n}\E\left|\la \nabla u_Y(s)\ra\right|^2\geq 0.
\end{align*}
Due to~\eqref{e.dot_u_bdd}, it is also bounded.

\subsubsection{Proof of~\eqref{e.partial_2_F_n}}
The first range follows from the formula for $\partial_2 \bar F$ in~\eqref{e.d_h_2F_n}, the definition of $\rho_{L-1,n}$ in~\eqref{e.rho_l}, the Cauchy--Schwarz inequality and the Nishimori identity. The boundedness is clear from the observation that there is a constant $C$ such that, a.s.,
\begin{align}\label{e.|x|<sqrt(n)}
    \left|X^\pr{L-1}\right|,\ \left|x^\pr{L-1}\right| \leq C\sqrt{n}
\end{align}
which is ensured by
\eqref{e.n_l/n},~\ref{i.assump_w_|X|<sqrt_n} and~\ref{i.assump_w_varphi}.

\subsubsection{Proof of~\eqref{e.1st_der_F_n_est}}
In view of~\eqref{e.1st_der_F}, this is valid due to~\eqref{e.|x|<sqrt(n)}. 

\subsubsection{Proof of~\eqref{e.2nd_d_barF_lbd}}

We first show that $\partial_1\partial_2 \bar F\geq 0$. Recall the formula for $\partial_1 \bar F$ in~\eqref{e.d_h_1F_n}. Let use write $\bar u = \nabla u_Y(s)$ and $\bar U= \nabla u_Y(S)$. We also adopt the notation~\eqref{e.bar_X_notation}. Then, we compute that
\begin{align*}
    \partial_1\partial_2 \bF &=  (2n)^{-1}\partial_2\E \la \bar u\cdot\bar U\ra \\
    &= (4n)^{-1}\E\Big\langle (\bar u\cdot \bar U)\big((\htwo)^{-\frac{1}{2}}Z'\cdot \bar x+ 2 \bar x\cdot \bar X-\bar x\cdot \bar x \big) \\
    &\qquad -(\bar u\cdot\bar U) \big((\htwo)^{-\frac{1}{2}}Z'\cdot \bar x' + 2\bar x'\cdot \bar X -\bar x'\cdot \bar x'\big)\Big\rangle.
\end{align*}
Perform the Gaussian integration by parts on $Z'$ to get that
\begin{align*}
    \partial_1\partial_2 \bF &=(4n)^{-1}\E\Big\langle (\bar u\cdot \bar U)\big((\bar x-\bar x')\cdot \bar x+2\bar x\cdot \bar X- \bar x\cdot \bar x\big)\\
    & \qquad -(\bar u\cdot \bar U)\big((\bar x +\bar x'-2\bar x'')\cdot \bar x' + 2\bar x'\cdot \bar X-\bar x'\cdot \bar x'\big)\Big\rangle.
\end{align*}
Using the Nishimori identity to replace $\bar U$ and $\bar X$ with replicas and invoking the symmetry of replicas, we arrive at
\begin{align*}
    \partial_1\partial_2 \bF &= (2n)^{-1}\E\la (\bar u\cdot \bar u')(\bar x\cdot \bar x') - 2 (\bar u\cdot \bar u')(\bar x\cdot \bar x'') + (\bar u\cdot \bar u')(\bar x''\cdot \bar x''')\ra\\
    & = (2n)^{-1}\E\big|\la \bar u \,\bar x^\intercal \ra - \la \bar u\ra\la \bar x\ra^\intercal \big|^2\geq 0.
\end{align*}
The computation for $\partial_2^2\bar F\geq 0$ is the same with $\bar U,\bar u$ above replaced by $\bar X,\bar x$. The verification of $\partial^2_1\bar F\geq 0$ follows the same procedure but is computationally more involved. We refer to the proof of \cite[Proposition~18 in its supplementary material]{barbier2019optimal} for details.

\subsubsection{Proof of~\eqref{e.2nd_d_F_lbd}}

Notice that the first two terms on the right of formula~\eqref{e.2nd_der_F} for $\partial_2^2 F$ form a variance. Then, the desired lower bound follows from~\eqref{e.|x|<sqrt(n)}.

\section{Weak solutions}\label{s.hj}
We consider the equation~\eqref{e.HJ} over $\Omega_\rho$ defined in~\eqref{e.def_Omega_r} for some $\rho>0$. We define weak solutions and prove the uniqueness and existence of weak solutions. Uniqueness is ensured by Proposition~\ref{p.uniqueness}. Proposition~\ref{p.hopf} furnishes the existence part by providing a variational formula known as the Hopf formula. After stating these, we prove the two propositions in the ensuing subsections.

We endow measurable subsets of Euclidean spaces with the Lebesgue measure. In what follows, the phrase ``almost everywhere'' or ``almost every'' (a.e.) is understood with respect to the Lebesgue measure. We denote by $\Int \Omega_\rho$ the interior of $\Omega_\rho$. In this section, for convenience, we also denote the spacial variable by $x$ instead of $h$.

\begin{definition}\label{def.weak_sol}
For $L\in\N$ and $\rho>0$, a function $f:\Omega_\rho\to \R$ is a weak solution of~\eqref{e.HJ} if
\begin{enumerate}
    \item \label{item.1_der_weak_sol} $f$ is Lipschitz, and $\partial_1f\geq 0$, $\partial_2f\in [0,\frac{\alpha_{L-1}\rho}{2}]$ a.e.;
    \item \label{item:2}$f$ satisfies~\eqref{e.HJ} a.e.;
    \item \label{item:3} for all $(t,x)\in \Int \Omega_\rho$ and all sufficiently small $\lambda\geq 0$, it holds that
    \begin{align}\label{e.partial_convex_f}
        f(t,x+\lambda e_1+\lambda e_2) +f(t,x) - f(t,x+\lambda e_1) -f(t,x+\lambda e_2) & \geq 0.
    \end{align}
\end{enumerate}
\end{definition}

By Rademacher's theorem, condition~\eqref{item.1_der_weak_sol} implies that $f$ is differentiable a.e. Condition~\eqref{item:2} is understood in the sense that, outside a set with zero measure, $f$ is differentiable and its derivatives satisfy equation~\eqref{e.HJ}.
In~\eqref{item:3}, $\{e_1,e_2\}$ is the standard basis for $\R^2$. Condition~\eqref{item:3} can be interpreted as a type of partial convexity. For a smooth radial bump function $\xi:\R^2\to\R$ supported on the unit disk satisfying $\xi\in[0,1]$ and $\int\xi = 1$, introduce, for every $\eps\in(0,1)$,
\begin{align}\label{e.mollifier}
    \xi_\eps(x) = \eps^{-2}\xi\left(\eps^{-1}x\right),\quad\forall x \in \R^2.
\end{align}
If $f$ is a weak solution, then condition~\eqref{item:3} along with the continuity of $f$ implies that
\begin{align}\label{e.2nd_mixed_derivatives_f}
    \partial_1\partial_2(f(t,\cdot)*\xi_\eps)(x)\geq 0,
\end{align}
for every $(t,x)$ in
\begin{align}\label{e.Omega_rho_eps}
    \Omega_{\rho,\eps}= \left\{t\in\left[0,1-\frac{2}{\rho}\eps\right],\ x_1\in[\eps,\rho(1-t)-\eps],\ x_2\in[\eps,\infty)\right\},
\end{align}
where the convolution in~\eqref{e.2nd_mixed_derivatives_f} is taken in terms of the spacial variable.

\rv{The monotonicity implied by~\eqref{item.1_der_weak_sol} and the partial convexity in~\eqref{item:3} are crucial to the uniqueness of solutions in Proposition~\ref{p.uniqueness} below, and to the convergence of free energy in Proposition~\ref{p.cvg} of the next section. Both proofs share a similar argument where conditions~\eqref{item.1_der_weak_sol} and~\eqref{item:3} are needed to ensure that a certain quantity is nonpositive (see the paragraph above \eqref{e.dJ_leq_0} and the paragraph above \eqref{e.dt_J_n}).}

The main results of this section are stated below.

\begin{proposition}\label{p.uniqueness}
Given a Lipschitz function $\psi:[0,\rho]\times \R_+\rightarrow \R$, there is at most one weak solution $f$ of~\eqref{e.HJ} satisfying $f(0,\cdot) = \psi$.
\end{proposition}

\begin{proposition}\label{p.hopf}
Let $\psi_1:[0,\rho]\to \R$ and $ \psi_2:\R_+\to\R$ be Lipschitz, nondecreasing and convex. In addition, suppose that
\begin{align}\label{e.psi_2}
    \partial_2\psi_2\in \left[0,\frac{\alpha_{L-1}\rho}{2}\right],\quad\text{a.e.}
\end{align}
Define $\psi:[0,\rho]\times\R_+\to\R$ by
\begin{align}\label{e.def_general_psi}
    \psi(x)= \psi_1(x_1)+\psi_2(x_2),\quad\forall x \in [0,\rho)\times \R_+.
\end{align}
Then, the formula
\begin{align}\label{e.Hopf_formula_orig}
    f(t,x) = \sup_{z\in\R_+\times [0,\frac{\alpha_{L-1}\rho}{2}]}\inf_{y\in[0,\rho]\times\R_+}\{z\cdot(x-y)+\psi(y)+t\H_L(z)\}, \quad \forall (t,x)\in\Omega_\rho,
\end{align}
gives a weak solution of~\eqref{e.HJ} satisfying $f(0,\cdot)=\psi$.
\end{proposition}

The expression in~\eqref{e.Hopf_formula_orig} is known as the Hopf formula \cite{bardi1984hopf,lions1986hopf}. 

\subsection{Proof of Proposition~\ref{p.uniqueness}}

The idea of this proof can be seen in \cite[Section~3.3.3]{evans2010partial}.
Let $f$ and $g$ be weak solutions to~\eqref{e.HJ}. Setting $w = f-g$, we have that
\begin{align*}
    \partial_t w & = \H_L(\nabla f)-\H_L(\nabla g)  = b\cdot \nabla w
\end{align*}
where the vector $b$ is given by
\begin{align}\label{e.def_b_eps}
    b = \frac{2}{\alpha_{L-1}}\big(\partial_{2}g,\  \partial_{1}f\big).
\end{align}
For some smooth function $\phi:\R_+ \to \R_+$ to be chosen later, we set $v = \phi(w)$, which, by the chain rule, satisfies that
\begin{align}\label{e.v_eqn}
    \partial_t v = b \cdot \nabla v.
\end{align}
Then, we regularize $b$ by setting $b_\eps = b*\xi_\eps$ for the mollifier $\xi_\eps$ introduced in~\eqref{e.mollifier}, where we understand that the convolution is taken with respect to the spatial variable.
On $\Omega_{\rho,\eps}$ given in~\eqref{e.Omega_rho_eps},
the equation~\eqref{e.v_eqn} can be rewritten as
\begin{align}\label{e.partial_t_v}
    \partial_t v = \div( vb_\eps) - v\,\div\, b_\eps +(b-b_\eps)\cdot \nabla v.
\end{align}
Before proceeding further, we need to estimate some terms related to this display. 

\medskip

\rv{Let $f_\eps =f*\xi_\eps$ and $g_\eps=g*\xi_\eps$ for $\xi_\eps$ given in~\eqref{e.mollifier}.} \rv{Condition~\eqref{item:3} of }Definition~\ref{def.weak_sol} and~\eqref{e.2nd_mixed_derivatives_f} imply that, for all $(t,x)\in\Omega_{\rho,\eps}$,
\begin{align*}
     \partial_1\partial_2 f_\eps(t,x),\  \partial_1\partial_2 g_\eps(t,x)\geq 0,
\end{align*}
and thus
\begin{align}\label{e.div_b_eps_geq_0}
    \div\, b_\eps  \geq 0,\quad\forall (t,x)\in\Omega_{\rho,\eps}.
\end{align}
By the definitions of $f_\eps$ and $g_\eps$, we also have that
\begin{align}\label{e.grad_f_eps_bound}
    |\nabla f_\eps |\leq \|f\|_\mathrm{Lip},\quad |\nabla g_\eps |\leq \|g\|_\mathrm{Lip}.
\end{align}

\rv{By~\eqref{e.def_b_eps} and~\eqref{e.grad_f_eps_bound}, we can fix a finite constant $R$ to satisfy}
\begin{align}\label{e.def_R_in_HJ_section}
    R>\sup_{\eps}\|b_\eps\|_\infty.
\end{align}
Fix any $\eta>0$ and define, for $t\in[0,1-\frac{2}{\rho}\eta]$,
\begin{align}
    D_t& = [\eta, \rho(1-t)-\eta]\times [\eta, R(1-t)], \label{e.D_t_uniqueness}\\
    \Gamma_{1,t} &= [\eta, \rho(1-t)-\eta]\times\{R(1-t)\},\nonumber\\
     \Gamma_{2,t} &= \{\rho(1-t)-\eta\}\times[\eta, R(1-t)].\nonumber
\end{align}
Now, we introduce, for $t\in[0,1-\frac{2}{\rho}\eta]$,
\begin{align*}
    J(t) = \int_{D_t}v(t,x)\d x.
\end{align*}
We emphasize that $J$ depends on $\eta$.
Choose $\eps<\eta$ to ensure that $\bigcup_{t\in[0,1-\frac{2}{\rho}\eta]}(\{t\}\times D_t)\subset \Omega_{\rho,\eps}$. Using~\eqref{e.partial_t_v} and integration by parts on the integral of $\div( vb_\eps)$, we can compute that
\begin{align*}
    \frac{\d}{\d t}J(t)& = \int_{D_t}\partial_t v  - R \int_{\Gamma_{1,t}} v-  \rho \int_{\Gamma_{2,t}} v\\
    & = \int_{\Gamma_{1,t}}(\mathbf{n}\cdot b_\eps - R)v+  \int_{\Gamma_{2,t}}(\mathbf{n}\cdot b_\eps - \rho)v\\
    &\quad +\int_{\partial D_t \setminus \rvv{(\Gamma_{1,t}\cup\Gamma_{2,t})}} (\mathbf{n}\cdot b_\eps)v + \int_{D_t} v(-\div\, b_\eps) +\int_{D_t}(b-b_\eps)\cdot \nabla v,
\end{align*}
where $\mathbf{n}$ stands for the outer normal vector. Then, $\mathbf n = (0,1)$ on $\Gamma_{1,t}$ and $\mathbf n = (1,0)$ on $\Gamma_{2,t}$. We treat the integrals after the second equality individually. 
\rv{Due to~\eqref{e.def_R_in_HJ_section}}, the first integral is nonpositive. By \rv{condition~\eqref{item.1_der_weak_sol} of} Definition~\ref{def.weak_sol} and~\eqref{e.def_b_eps}, the second integral is nonpositive. 
Note that on $\partial D_t \setminus \rvv{(\Gamma_{1,t}\cup\Gamma_{2,t})}$, we have $-\mathbf{n}\in\R^2_+$. By \rv{condition~\eqref{item.1_der_weak_sol} of} Definition~\ref{def.weak_sol}, we can infer from the definition of $b_\eps$ that $b_\eps \in \R^2_+$ on $\partial D_t\setminus\rvv{(\Gamma_{1,t}\cup\Gamma_{2,t})}$, which implies that the third integral is nonpositive. 
In view of~\eqref{e.div_b_eps_geq_0}, the fourth integral is again nonpositive. The last one is $o_\eps(1)$. Therefore, sending $\eps\to 0$, we conclude that, for $t\in[0,1-\frac{2}{\rho}\eta]$,
\begin{align}\label{e.dJ_leq_0}
    \frac{\d }{\d t}J(t)\leq 0.
\end{align}

Since $w(0,x) = f(0,x)-g(0,x)=0$, we have $\|w(\delta, \cdot)\|_\infty \leq \delta(\|f\|_\mathrm{Lip}+\|g\|_\mathrm{Lip})$, for each $\delta>0$. Let us choose $\phi=\phi_\delta$ to satisfy 
\begin{align*}
    \begin{cases}\phi_\delta(z) = 0,&\quad \text{if }|z|\leq \delta(\|f\|_\mathrm{Lip}+\|g\|_\mathrm{Lip}),\\
    \phi_\delta(z)>0, &\quad \text{otherwise}.
    \end{cases}
\end{align*}
Therefore, due to $v=\phi_\delta(w)$, we have that
\begin{align*}
    J(\delta) = \int_{D_\delta}v(\delta,x)\d x =\int_{D_\delta}\phi_\delta(w(\delta ,x ))\d x = 0.
\end{align*}
Since $J(t)$ is nonnegative,~\eqref{e.dJ_leq_0} implies that $J_\delta (t)=0$ for all $t\in[\delta, 1-\frac{2}{\rho}\eta]$. This together with the definition of $\phi$ guarantees that
\begin{align*}
    |f(t,x)-g(t,x)|\leq \delta(\|f\|_\mathrm{Lip}+\|g\|_\mathrm{Lip}), \quad \forall x\in D_t,\ \forall t\in\left[\delta, 1-\frac{2}{\rho}\eta\right].
\end{align*}
Recall the definition of $D_t$ in~\eqref{e.D_t_uniqueness} which depends on $\eta$.
Taking $\delta\to0$, $\eta\to0$ and $R\to\infty$, we conclude that $f=g$.

\subsection{Proof of Proposition~\ref{p.hopf}}

Let us extend $\psi_1$ to be defined on $\R_+$ by setting
\begin{align}\label{e.psi_ext}
    \psi_1(x_1) = \infty,\quad\forall x_1 \in \R_+\setminus [0,\rho].
\end{align}
Then, $\psi_1$ is still convex and nondecreasing. For $u:\R^2_+\to \R\cup\{\infty\}$, the Fenchel transformation is defined by
\begin{align}\label{e.def_u*}
    u^*(x) = \sup_{y\in \R^2_+}\{y\cdot x - u(y)\},\quad \forall x\in \R^2_+.
\end{align}
Hence, we can rewrite the Hopf formula~\eqref{e.Hopf_formula_orig} as
\begin{align}\label{e.f_Hopf}
\begin{split}
    f(t,x) &= \sup_{z\in \R_+\times [0, \frac{\alpha_{L-1}\rho}{2}]}\inf_{y\in \R^2_+}\big\{z\cdot (x-y) +\psi(y) +t \H_L(z)\big\}\\
    & = \sup_{z\in \R_+\times [0, \frac{\alpha_{L-1}\rho}{2}]}\{z\cdot x - \psi^*(z)+t\H_L(z)\}.
\end{split}
\end{align}

We first show that $f$ is indeed finite on $\Omega_\rho$. From~\eqref{e.def_general_psi}, it follows that
\begin{align}\label{eq:psi*(z)=psi_1+psi_2}
    \psi^*(z) = \psi^*_1(z_1) + \psi^*_2(z_2),\quad\forall z\in\R_+^2,
\end{align}
where the Fenchel transforms on the right-hand side are for functions defined on $\R_+$ which are defined analogously to~\eqref{e.def_u*}. 
By the assumption that $\psi_1$ is Lipschitz and nondecreasing, there is some $R\geq 0$ such that
\begin{align*}
    0\leq \psi_1(r)- \rv{\psi_1}(r') \leq R(r-r'),\qquad \forall  r\geq r',\quad r,r'\in [0,\rho].
\end{align*}
Due to the extension in~\eqref{e.psi_ext}, we have that
\begin{align*}
    \psi_1^*(z_1) = \sup_{y_1\in[0,\rho]}\{y_1z_1 - \psi_1(y_1)\}.
\end{align*}
The above two displays imply that
\begin{align}\label{e.psi*_1(z_1)}
    \psi_1^*(z_1) = \rho z_1 - \psi_1(\rho),\quad\forall z_1\geq R.
\end{align}
On the other hand, due to~\eqref{e.psi_2}, 
\begin{align}\label{eq:psi_2=infty}
    \psi_2^*(z_2) =\infty,\quad\forall z_2>\frac{\alpha_{L-1}\rho}{2}.
\end{align}
Using this,~\eqref{eq:psi*(z)=psi_1+psi_2} and the expression of $\H_L$ in~\eqref{e.H}, we rewrite~\eqref{e.f_Hopf} as
\begin{align}\label{e.f_formula_2}
    f(t,x) = \sup_{z_2\in [0,\frac{\alpha_{L-1}\rho}{2}]}\left\{z_2x_2-\psi^*_2(z_2)+\sup_{z_1\in \R_+}\left\{z_1x_1-\psi^*_1(z_1)+\frac{2t}{\alpha_{L-1}}z_1z_2\right\}\right\}.
\end{align}
We show that the second $\sup$ can be restricted to $z_1\in[0,R]$. Given $(t,x)\in\Omega_\rho$, we have that $x_1\in[0,\rho(1-t)]$ due to the definition of $\Omega_\rho$ in~\eqref{e.def_Omega_r}. This implies that
\begin{align*}
    x_1+\frac{2t}{\alpha_{L-1}}z_2 -\rho \leq 0,\qquad\forall (t,x)\in\Omega_\rho, \ z_2\in\left[0,\frac{\alpha_{L-1}\rho}{2}\right],
\end{align*}
which together with~\eqref{e.psi*_1(z_1)} shows that
\begin{align*}
    \sup_{z_1\in[R,\infty)}\left\{z_1x_1-\psi^*_1(z_1)+\frac{2t}{\alpha_{L-1}}z_1z_2\right\} =
    \sup_{z_1\in[R,\infty)}\left\{\left(x_1+\frac{2t}{\alpha_{L-1}}z_2 -\rho \right)z_1 +\psi_1(\rho)\right\}
    \\
    \leq \left(x_1+\frac{2t}{\alpha_{L-1}}z_2 -\rho \right)R+\psi_1(\rho)= Rx_1-\psi^*_1(R)+\frac{2t}{\alpha_{L-1}}Rz_2,
\end{align*}
In other words, the above $\sup$ is achieved at $z_1 = R$.
Hence, the second $\sup$ in~\eqref{e.f_formula_2} can be taken over $z_1\in [0,R]$ and thus the
$\sup$ in~\eqref{e.f_Hopf} can be restricted to $z$ belonging to the compact set
\begin{align}\label{e.compact_K}
    K = [0,R]\times \left[0,\frac{\alpha_{L-1}\rho}{2}\right].
\end{align}
Therefore, due to the easy observation that $\psi^*$ is nondecreasing and lower semi-continuous, we can see that $f$ is finite on $\Omega_\rho$, and furthermore, for every $(t,x)\in\Omega_\rho$, 
\begin{align}\label{e.f_sup_achieve_at_z}
    \rv{\text{there is $z\in K$ such that }}f(t,x) = z\cdot x -\psi^*(z)+t\H_L(z).
\end{align}

In the following, we verify that~\eqref{e.f_Hopf} is a weak solution by checking the initial condition, and conditions~\eqref{item.1_der_weak_sol},~\eqref{item:2},~\eqref{item:3} of Definition~\ref{def.weak_sol}.

\subsubsection{Initial condition}
Using~\eqref{eq:psi*(z)=psi_1+psi_2} and~\eqref{eq:psi_2=infty}, the expression in~\eqref{e.f_Hopf} at $t=0$ becomes
\begin{align*}
    f(0,x) = \sup_{z\in\R_+^2}\{z\cdot x - \psi^*(z)\} = \psi^{**}(x).
\end{align*}
Since it is clear from the assumption that the extended $\psi$ is lower semi-continuous, nondecreasing and convex, the Fenchel--Moreau biconjugation identity (cf.\ \cite[Theorem~12.4]{rockafellar1970convex}, and \cite[Theorem~2.2]{chen2020fenchel} for more general cones) ensures that
\begin{align*}
    \psi(x) =\psi^{**}(x),\quad \forall x\in\R^2_+.
\end{align*}
In particular, we have $f(0,\cdot)=\psi$ on $\Omega_\rho$.

\subsubsection{Condition~\eqref{item.1_der_weak_sol}}

Let $(t,x)\in\Omega_\rho$ and $z\in K$ be given by~\eqref{e.f_sup_achieve_at_z}. Using this and~\eqref{e.f_Hopf}, we have for $(t',x')\in \Omega_\rho$ that
\begin{align}\label{e.f_Lip}
    f(t,x) -f(t',x')\leq z\cdot(x-x') + \H_L(z)(t-t').
\end{align}
A similar equality holds for some $z'\in K$ when interchanging $(t,x), (t',x')$.
By the compactness of $K$, we can see that $f$ is Lipschitz. 
Due to Rademacher's theorem, $f$ is differentiable a.e. 
Using~\eqref{e.f_Lip} and the definition of $K$ in~\eqref{e.compact_K}, we can also see that 
\begin{align*}
    \partial_1 f \in [0,R],\qquad \partial_2 f\in [0,\frac{\alpha_{L-1}\rho}{2}], \qquad \text{a.e.,}
\end{align*}
which completes the verification of \rv{condition~\eqref{item.1_der_weak_sol} of} Definition~\ref{def.weak_sol}.

\subsubsection{Condition~\eqref{item:2}}

We want to verify that~\eqref{e.f_Hopf} satisfies~\eqref{e.HJ} almost everywhere. Let $(t,x)$ be a point at which $f$ is differentiable. We can assume that $(t, x)\in\Int\Omega_\rho\subset (0,\infty)^3$, because otherwise, $(t,x)$ belongs to a set with Lebesgue measure zero. Let $z$ be given by~\eqref{e.f_sup_achieve_at_z}.
By this and~\eqref{e.f_Hopf}, for $s\in \R$ and $h\in \R^2$ sufficiently small,
\begin{align}\label{e.f(t,x)_upper_bound_verify_HJ}
     f(t+ s,x+h)- f(t,x)\geq z\cdot h+s\H_L( z).
\end{align}
Set $s=0$ and vary $h$ to see that
\begin{align*}
    z = \nabla f(t,x).
\end{align*}
Then, we set $h=0$ in~\eqref{e.f(t,x)_upper_bound_verify_HJ}, vary $s$ and use the above display to obtain
\begin{align*}
    \partial_t f(t,x)= \H_L(\nabla f(t,x)).
\end{align*}

\subsubsection{Condition~\eqref{item:3}}

Let $(t,x)\in\Int\Omega_\rho$ and $\lambda\in\R$ be sufficiently small. Due to~\eqref{e.f_sup_achieve_at_z}, there are $z,z'$ such that
\begin{align}\label{e.z_z'_property}
\begin{split}
    f(t,x+\lambda e_1) &= z\cdot (x+\lambda e_1) - \psi^*(z) + t\H_L(z),\\
    f(t,x+\lambda e_2) &= z'\cdot (x+\lambda e_2) - \psi^*(z') + t\H_L(z').
\end{split}
\end{align}

Case 1: $(z_1,z_2)\leq (z'_1,z'_2)$ or $(z_1,z_2)\geq (z'_1,z'_2)$. Let us only treat the latter case. The other case is similar. Using~\eqref{e.f_Hopf}, we have
\begin{align*}
    f(t,x+\lambda e_1+\lambda e_2) &\geq z\cdot (x+\lambda e_1+\lambda e_2)-\psi^*(z) + t\H_L(z),\\
    f(t, x) &\geq z'\cdot x-\psi^*(z') + t\H_L(z').
\end{align*}
This along with~\eqref{e.z_z'_property} implies that the left hand side of~\eqref{e.partial_convex_f} is bounded below by
\begin{align*}
   \lambda  z\cdot  e_2 -\lambda z'\cdot e_2=\lambda (z_2 -z'_2)\geq 0.
\end{align*}

Case 2: neither $(z_1,z_2)\leq (z'_1,z'_2)$ nor $(z_1,z_2)\geq (z'_1,z'_2)$. This condition implies that
\begin{align}\label{e.(z_1-z'_1)(z_2-z'_2)<0}
    (z_1-z'_1)(z_2-z'_2)< 0.
\end{align}
Let $\tz =(z_1,z'_2)$ and $\tz' = (z'_1, z_2)$. By \rv{the first line of}~\eqref{e.f_Hopf}, for each $\delta>0$, there are $y,y'\in\R^2_+$ such that
\begin{align}\label{e.partial_convex_lower_bound_1}
\begin{split}
    f(t,x+\lambda e_1+\lambda e_2) &\geq \tz \cdot (x+\lambda e_1+\lambda e_2-y)+\psi(y) +t\H_L(\tz)-\delta,\\
    f(t,x) &\geq \tz' \cdot (x-y')+\psi(y') +t\H_L(\tz')-\delta.
\end{split}
\end{align}
We set
\begin{align*}
    \ty =(y_1,y'_2),\quad \ty' = (y'_1, y_2).
\end{align*}
Note that
\begin{align}\label{e.tz_ty_z_y_relation}
    \tz\cdot y+\tz'\cdot y' - z\cdot \ty - z'\cdot \ty' =0.
\end{align}
\rv{Using~\eqref{e.def_u*} in~\eqref{e.z_z'_property}}, we also have
\begin{align}\label{e.partial_convex_lower_bound_2}
\begin{split}
     f(t,x+\lambda e_1) &\leq  z \cdot (x+\lambda e_1-\ty)+\psi(\ty) +t\H_L(z),\\
    f(t,x+\lambda e_2) &\leq z' \cdot (x+\lambda e_2-\ty')+\psi(\ty') +t\H_L(z').
\end{split}
\end{align}
To get a lower bound for the left hand side of~\eqref{e.partial_convex_f}, we start by observing that, due to~\eqref{e.tz_ty_z_y_relation}, 
\begin{align*}
    &\quad \tz \cdot (x+\lambda e_1+\lambda e_2-y) + \tz' \cdot (x-y') - z \cdot (x+\lambda e_1-\ty) - z'\cdot (x+\lambda e_2-\ty')\\
    &=(\tz+\tz'-z-z')\cdot x - (\tz\cdot y+\tz'\cdot y' - z\cdot \ty - z'\cdot \ty')+\lambda(z_1+z_2'-z_1-z'_2)\\
    &=0.
\end{align*}
This along with~\eqref{e.partial_convex_lower_bound_1} and~\eqref{e.partial_convex_lower_bound_2} implies that the left hand side of~\eqref{e.partial_convex_f} can be bounded from below by
\begin{align*}
    \psi(y)+\psi(y')- \psi(\ty)-\psi(\ty')+t\big(\H_L(\tz) + \H_L(\tz') - \H_L(z) -\H_L(z')\big)-2\delta.
\end{align*}
From~\eqref{e.def_general_psi}, we can see
\begin{align*}
    \psi(y)+\psi(y')- \psi(\ty)-\psi(\ty')=0.
\end{align*}
Lastly, due to~\eqref{e.(z_1-z'_1)(z_2-z'_2)<0} and the definition of $\H_L$ in~\eqref{e.H}, we can compute that
\begin{align*}
    \H_L(\tz)+\H_L(\tz') - \H_L(z)-\H_L(z')= -\frac{2}{\alpha_{L-1}}(z_1-z'_1)(z_2-z'_2)> 0.
\end{align*}
The above three displays imply that the left-hand side of~\eqref{e.partial_convex_f} is bounded from below by $-2\delta$. The desired result follows by sending $\delta \to 0$.

\section{Convergence of the free energy}\label{s.cvg}

The goal of this section is to prove Theorem~\ref{t}. The key tool is Proposition~\ref{p.cvg} stated below, which ensures the convergence of $\bar F_{\beta,L,n}$ given the convergence of $\bar F_{\beta,L,n}(0,\cdot)$ and some additional conditions. The object $\bar F_{\beta,L,n}(0,\cdot)$ is closely related to the free energy associated with the $(L-1)$-layer model. Hence, an iteration is employed in Section~\ref{s.iteration} to complete the proof of Theorem~\ref{t}.

Recall the definition of $\rho_{L-1,n}$ in~\eqref{e.rho_l} and of domain $\Omega_{\rho}$, for $\rho>0$, in~\eqref{e.def_Omega_r}.

\begin{proposition}\label{p.cvg}

Assume~\ref{i.assump_w_|X|<sqrt_n},~\ref{i.assump_w_varphi} and~\ref{i.assump_Phi_Gaussian} for some $L\in \N$.
Suppose that the following holds:
\begin{enumerate}
    \item \label{item:rho_cvg} the limit~\eqref{e.rho_cvg} for $l=L-1$ exists for some $\rho_{L-1}>0$;
    \item \label{item:initial_cvg} there is a continuous $\psi_{\beta,L}:[0,\rho_{L-1}]\times \R_+\to\R$ such that
    \begin{align*}
        \lim_{n\to\infty}\bF_{\beta,L,n}(0,h)=\psi_{\beta,L}(h),\quad \forall h\in [0,\rho_{L-1})\times\R_+ ,
    \end{align*}
    and there is a weak solution $f_{\beta,L}$ to~\eqref{e.HJ} on $\Omega_{\rho_{L-1}}$ satisfying $f_{\beta,L}(0,\cdot)=\psi_{\beta,L}$;
    \item
    \label{i.concent_norm}

    there is $C>0$ such that
    \begin{align*}
        \E\left[\left(\frac{\left|X^\pr{L-1}\right|^2}{n_{L-1}}-\rho_{L-1,n} \right)^2\right] \leq \frac{C}{n},\quad\forall n\in\N;
    \end{align*}

    \item \label{item:concentration} for every $M\geq 1$,
    \begin{align*}
        \lim_{n\to\infty} \sup_{\substack{t\in[0,1],\\ h_1\in[0,\ \rho_{L-1,n}(1-t)]}}\E\left[\left\|F_{\beta,L,n} - \bar F_{\beta,L,n}\right\|^2_{L^\infty_{h_2}([0,M])}(t,h_1)\right]=0.
    \end{align*}
\end{enumerate}
Then, for every $\rho'\in(0,\rho_{L-1})$,
\begin{align*}
    \lim_{n\to\infty}\bar F_{\beta,L,n}(t,h) = f_{\beta,L}(t,h),\quad\forall (t,h)\in\Omega_{\rho'}.
\end{align*}

\end{proposition}

We restrict to the domain $\Omega_{\rho'}$ because the pointwise limit of $\bar F_{\beta,L,n}$ may not be well-defined on boundary points of $\Omega_{\rho_{L-1}}$ (recall that $\bar F_{\beta,L,n}$ is defined on $\Omega_{\rho_{L-1,n}}$).
The proof of this proposition is postponed to Section~\ref{s.proof_cvg}.

\subsection{Iteration}\label{s.iteration}
Let us prove Theorem~\ref{t} using Proposition~\ref{p.cvg} together with some technical results postponed to Section~\ref{s.aux}.

Assuming that~\ref{i.assump_iid_bdd}--\ref{i.assump_Phi_Gaussian} hold for the model with $L_0$ layers, then these assumptions automatically hold for all $L\in\{1,\dots, L_0\}$.
Hence, for all $L\in\{1,\dots, L_0\}$, conditions~\eqref{item:rho_cvg},~\eqref{i.concent_norm}, ~\eqref{item:concentration} in Proposition~\ref{p.cvg} are guaranteed to hold by Lemmas~\ref{l.cvg_rho},~\ref{l.cvg_norm} and~\ref{l.concent}, respectively.
We will apply Proposition~\ref{p.cvg} iteratively to prove Theorem~\ref{t}. 
Recall the definitions of $F^\circ_{\beta,L,n}$, $\Psi_0$, $\Psi_l$, $F_{\beta,L,n}$ in~\eqref{e.F^circ},~\eqref{e.Psi_0},~\eqref{e.Psi_l},~\eqref{e.F_beta_L_n_(t,h)}, respectively, and also the important relation~\eqref{e.EF^circ=bar_F}, which implies that
\begin{align}\label{e.lim_EF^circ=lim_bar_F}
    \lim_{n\to\infty}\E F^\circ_{\beta,L,n} = \lim_{n\to\infty} \bar F_{\beta,L,n}(1,0),
\end{align}
whenever one of the limits exists.
Also, recall the definition of $\alpha_l$ in~\eqref{e.n_l/n}.

Before proceeding, let us record the following result. Comparing the definitions of~\eqref{e.P_beta(y|x)} and~\eqref{e.tilde_Pout}, and using the fact that $A^\pr{L}$ has i.i.d.\ components due to~\ref{i.assump_Phi_Gaussian}, we can see that, for every $\beta,L,n$,
\begin{align}\label{e.P_to_tilde_P}
    \Pout_{\beta,L,n}(y|z) = \prod_{j=1}^{n_L}\tilde\Pout_{\beta,L}(y_j|z_j),\quad \forall y,z\in \R^{n_L}.
\end{align}

We start with $L=1$. 
Using~\eqref{e.cap_S_mu}--\eqref{e.F_beta_L_n_(t,h)} with $L$ replaced by $1$, we can compute that
\begin{align*}
    \bar F_{\beta,1,n}(0,h)
    &= \frac{1}{n}\E\log\int\Pout_{\beta,1,n}\left(\mathcal Y^\pr{1}\Big| \sqrt{h_1}V+\sqrt{\rho_{0,n}-h_1}w\right)\d P_W(w) 
    \\
    &\quad\qquad + \frac{1}{n}\E\log\int e^{h_2 X\cdot x+ \sqrt{h_2}Z'\cdot x-\frac{h_2}{2}|x|^2}\d P_X(x)
\end{align*}
where
\begin{align*}
    \mathcal Y^\pr{1} = \sqrt{\beta}\varphi_1\left(\sqrt{h_1}V+\sqrt{\rho_{0,n}-h_1}W,A^\pr{1}\right)+Z,
\end{align*}
$V,W,Z$ are independent $n_1$-dimensional standard Gaussian vectors, $Z'$ is an $n$-dimensional Gaussian vector, \rv{and $A^{(1)}$ is the random vector with i.i.d.\ entries introduced above~\eqref{e.A^[l,l'],Phi^[l,l']} (in particular, $A^{(1)}_1$ appears in~\eqref{e.Psi_l} for $l=1$)}.
By~\eqref{e.P_to_tilde_P}, the definitions of $\Psi_0,\Psi_1$ in~\eqref{e.Psi_0},~\eqref{e.Psi_l}, and the fact that $X,V,W$ have i.i.d.\ entries (see~\ref{i.assump_iid_bdd} for the claim about $X$), the above can be rewritten as
\begin{align*}
    \bar F_{\beta,1,n}(0,h)= \frac{n_1}{n}\Psi_1(h_1,\beta;\rho_{0,n}) + \Psi_0(h_2),
\end{align*}
which, by~\eqref{e.n_l/n} and~\eqref{e.rho_cvg}, converges pointwise to
\begin{align*}
    \psi_{\beta, 1}(h) = \alpha_1\Psi_1(h_1,\beta;\rho_0) + \Psi_0(h_2).
\end{align*}
The results collected in Lemma~\ref{l.der_est} allow us to verify that $\psi_{\beta, 1}$ satisfies all the conditions imposed in Proposition~\ref{p.hopf}. Indeed, the above display shows that the decomposition as in~\eqref{e.def_general_psi} exists, and both components are Lipschitz, nondecreasing and convex due to~\eqref{e.partial_1_F_n},~\eqref{e.partial_2_F_n} and~\eqref{e.2nd_d_barF_lbd}. Moreover,~\eqref{e.psi_2} is ensured by~\eqref{e.partial_2_F_n},~\eqref{e.n_l/n},~\eqref{e.rho_cvg}. Hence, Proposition~\ref{p.hopf} yields the existence of a weak solution $f_{\beta,1}$ satisfying $f_{\beta,1}(0,\cdot)=\psi_{\beta, 1}$ given by the formula~\eqref{e.Hopf_formula_orig} with $L, \rho,\psi$ there replaced by $1,\rho_0,\psi_{\beta, 1}$, namely,
\begin{align*}
    f_{\beta,1}(t,h) = \sup_{z^\pr{1}\in\R_+\times [0,\frac{\alpha_{0}\rho_0}{2}]}\inf_{y^\pr{1}\in[0,\rho_0]\times\R_+}\left\{z^\pr{1}\cdot\left(h-y^\pr{1}\right)+\psi_{\beta, 1}\left(y^\pr{1}\right)+t\H_1\left(z^\pr{1}\right)\right\}
\end{align*}
for every $(t,h)\in\Omega_{\rho_0}$. Inserting the previous display and the formula for $\H_1$ in~\eqref{e.H} into the above and evaluating at $(t,h)=(1,0)$ \rv{yields} that
\begin{align*}
    f_{\beta,1}(1,0) = \sup_{z^\pr{1}}\inf_{y^\pr{1}}\left\{ \alpha_1\Psi_1\left(y^\pr{1}_1,\beta;\rho_0\right) + \Psi_0\left(y^\pr{1}_2\right) -y^\pr{1}\cdot z^\pr{1} +\frac{2}{\alpha_0}z^\pr{1}_1z^\pr{1}_2\right\}
\end{align*}
which exactly matches the right-hand side of~\eqref{e.limit} for $L=1$.

The above discussion also validates condition~\eqref{item:initial_cvg} in Proposition~\ref{p.cvg}.
Therefore, applying this proposition yields that
\begin{align*}
    \lim_{n\to\infty}\bar F_{\beta,1,n}(1,0) = f_{\beta,1}(1,0).
\end{align*}
Using~\eqref{e.lim_EF^circ=lim_bar_F}, this proves~\eqref{e.limit} for $L=1$.

Now, we assume that~\eqref{e.limit} is verified for $L-1$.
Using~\eqref{e.cap_S_mu}--\eqref{e.F_beta_L_n_(t,h)}, we can compute that
\begin{align*}
    \bar F_{\beta,L,n}(0,h)
    & = \frac{1}{n}\E\log\int\Pout_{\beta,L,n}\left(\mathcal Y^\pr{L}\Big| \sqrt{h_1}V+\sqrt{\rho_{L-1,n}-h_1}w\right)\d P_W(w) 
    \\
    &\quad\qquad + \frac{1}{n}\E\log\int e^{\sqrt{h_2}Y'\cdot x^\pr{L-1}-\frac{h_2}{2}|x^\pr{L-1}|^2}\d P_X(x)\d P_{A^{[1,L-1]}}(a)\\
    & = \mathtt{I}_1 + \mathtt{I}_2
\end{align*}
where
\begin{align*}
    \mathcal Y^\pr{L} =\sqrt{\beta}\varphi_L\left(\sqrt{h_1}V+\sqrt{\rho_{L-1,n}-h_1}W,A^\pr{L}\right)+Z.
\end{align*}
By~\eqref{e.P_to_tilde_P} and the definition of $\Psi_L$ given in~\eqref{e.Psi_l}, $\mathtt{I}_1=\frac{n_L}{n}\Psi_L(h_1,\beta;\rho_{L-1,n})$. Completing the square, we can rewrite $\mathtt{I}_2$ as
\begin{align}\label{e.I_2_split}
    \mathtt{I}_2 = \frac{1}{n}\E\log\int e^{-\frac{1}{2}|Y'-\sqrt{h_2}x^\pr{L-1}|^2}\d P_X(x)\d P_{A^{[1,L-1]}}(a) + \frac{1}{n}\E\log e^{\frac{1}{2}|Y'|^2}.
\end{align}
Recall the definition of $x^\pr{L-1}$ in~\eqref{e.x^L-1}. We can define $x^\pr{L-2}$ in the same fashion and it is related to $x^\pr{L-1}$ via
\begin{align*}
    x^\pr{L-1}=\varphi_{L-1}\left(\frac{1}{\sqrt{n_{L-2}}}\Phi^\pr{L-1}x^\pr{L-2},a^\pr{L-1}\right).
\end{align*}
Inserting this and $\d P_{A^{[1,L-1]}} = \d P_{A^\pr{L-1}}\d P_{A^{[1,L-2]}}$ to~\eqref{e.I_2_split}, and using~\eqref{e.P_beta(y|x)} with $\beta,L$ replaced by $h_2,L-1$, we can see that the first term in~\eqref{e.I_2_split} is given by
\begin{align*}
    \frac{1}{n}\E\log\int \Pout_{h_2,L-1,n}\left(Y'\bigg|\frac{1}{\sqrt{n_{L-2}}}\Phi^\pr{L-1}x^\pr{L-2}\right)\d P_X(x)\d P_{A^{[1,L-2]}}(a).
\end{align*}
Recall the definition of $Y'$ in~\eqref{e.Y'}. Comparing it with~\eqref{e.Y^circ}, we can see that $Y'$ is exactly the observable for the $(L-1)$-layer model with $\beta=h_2$. In view of the definition of the original free energy in~\eqref{e.F^circ}, we can see that the above display is exactly $\E F^\circ_{h_2,L-1,n}$. Now, we turn to the second term in~\eqref{e.I_2_split}. Using the definitions of $Y'$ in~\eqref{e.Y'} and $\rho_{L-1,n}$ in~\eqref{e.rho_l}, we can compute that this term is equal to
\begin{align*}
    \frac{1}{2n}\E|Y'|^2 = \frac{n_{L-1}}{2n}(1 + \rho_{L-1,n}h_2).
\end{align*}
We conclude that
\begin{align*}
    \bar F_{\beta,L,n}(0,h)=  \frac{n_L}{n}\Psi_L(h_1,\beta;\rho_{L-1,n}) + \E F^\circ_{h_2,L-1,n}+ \frac{n_{L-1}}{2n}(1 + \rho_{L-1,n}h_2),
\end{align*}
which, by the induction \rv{hypothesis that~\eqref{e.limit} is true}, converges pointwise on $[0,\rho_{L-1})\times\R_+$ to
\begin{align*}\psi_{\beta,L}(h) = \alpha_L\Psi_L(h_1,\beta;\rho_{L-1}) + f^\circ_{h_2,L-1}+ \frac{\alpha_{L-1}}{2}(1 + \rho_{L-1}h_2)
\end{align*}
where $f^\circ_{h_2,L-1}$ is the right-hand side of~\eqref{e.limit} with $\beta,L$ replaced by $h_2,L-1$, namely,
\begin{align*}
    f^\circ_{h_2,L-1} = \sup_{z^\pr{L-1}}\inf_{y^\pr{L-1}}\sup_{z^\pr{L-2}}\inf_{y^\pr{L-2}}\cdots \sup_{z^\pr{1}}\inf_{y^\pr{1}} \phi_{L-1}\left(h_2; y^\pr{1},\cdots,y^\pr{L-1};z^\pr{1},\cdots,z^\pr{L-1}\right)
\end{align*}
with $\phi_{L-1}$ defined analogously as in~\eqref{e.phi_L}.

Again, as argued in the base case, Lemma~\ref{l.der_est} enables us to verify all conditions in Proposition~\ref{p.hopf}, which gives a weak solution $f_{\beta,L}$ satisfying $f_{\beta,L}(0,\cdot) = \psi_{\beta,L}$. Moreover, $f_{\beta,L}$ is given by the formula~\eqref{e.Hopf_formula_orig} with $\rho,\psi$ there replaced by $\rho_{L-1},\psi_{\beta,L}$, namely,
\begin{align*}
    f_{\beta,L}(t,h) = \sup_{z^\pr{L}}\inf_{y^\pr{L}}\left\{z^\pr{L}\cdot\left(h-y^\pr{L}\right)+\psi_{\beta,L}\left(y^\pr{L}\right)+t\H_L\left(z^\pr{L}\right)\right\}
\end{align*}
for every $(t,h)\in\Omega_{\rho_{L-1}}$, where $\sup$ is taken over $z^\pr{L}\in\R_+\times [0,\frac{\alpha_{L-1}\rho_{L-1}}{2}]$ and $\inf$ is taken over $y^\pr{L}\in[0,\rho_{L-1}]\times\R_+$.
Inserting the previous two displays and the expression of $\H_L$ in~\eqref{e.H}  into the above, and evaluating at $(t,h)=(1,0)$, we obtain that
\begin{align*}
    &f_{\beta,L}(1,0) = \sup_{z^\pr{L}}\inf_{y^\pr{L}}\sup_{z^\pr{L-1}}\inf_{y^\pr{L-1}}\cdots \sup_{z^\pr{1}}\inf_{y^\pr{1}}
    \bigg\{-y^\pr{L}\cdot z^\pr{L} + \alpha_L\Psi_L\left(y^\pr{L}_1,\beta;\rho_{L-1}\right) 
    \\
    & \quad + \phi_{L-1}\left(y^\pr{L}_2; y^\pr{1},\cdots,y^\pr{L-1};z^\pr{1},\cdots,z^\pr{L-1}\right)+ \frac{\alpha_{L-1}}{2}\left(1 + \rho_{L-1}y^\pr{L}_2\right) + \frac{2}{\alpha_{L-1}}z^\pr{L}_1z^\pr{L}_2 \bigg\}.
\end{align*}
We can verify that the expression inside the curly brackets is given by~\eqref{e.phi_L}, and thus $f_{\beta,L}(1,0)$ is exactly the right-hand side of~\eqref{e.limit}.

Again, the above verifies condition~\eqref{item:initial_cvg} and allows us to apply Proposition~\ref{p.cvg} to obtain that
\begin{align*}\lim_{n\to\infty}\bar F_{\beta,L,n}(1,0) = f_{\beta,L}(1,0)
\end{align*}
which along with~\eqref{e.lim_EF^circ=lim_bar_F} gives~\eqref{e.limit} and completes the proof of Theorem~\ref{t}.

\subsection{Proof of Proposition~\ref{p.cvg}} \label{s.proof_cvg}

For lighter notation, we suppress some of the subscripts and simply write
\begin{align*}
    F_n = F_{\beta,L,n},\quad \psi=\psi_{\beta,L},\quad f=f_{\beta,L}, \quad \rho=\rho_{L-1},\quad \rho_n = \rho_{L-1,n}.
\end{align*}
\rv{Throughout this subsection, we write $\d h =\d h_1 \d h_2$ and $a\wedge b = \min\{a,b\}$ for $a,b\in\R$.}

We remark that it suffices to show
\begin{align}\label{e.cvg_L^1}
    \lim_{n\to\infty}\sup_{t\in[0,1]}\int_{E_{n,t}(R)}\big|\bar F_n(t,h)-f(t,h)\big|\d h=0
\end{align}
for every $R>0$, where 
\begin{align}\label{e.E_n,t}
    E_{n,t}(R) = [0,(\rho\wedge\rho_n)(1-t)]\times [0,R(1-t)].
\end{align}
Indeed, for every $\rho'<\rho$, we have $\rho'<\rho\wedge\rho_n$ for sufficiently large $n$ due to assumption~\eqref{item:rho_cvg}. Then,~\eqref{e.cvg_L^1} together with Fubini's theorem implies that the integral of $|\bar F_n-f|$ over $\Omega_{\rho'}\cap \{h_2\leq R(1-t)\}$ decays to $0$ as $n\to\infty$, which further implies that $\bar F_n$ converges to $f$ pointwise a.e.\ on $\Omega_{\rho'}\cap \{h_2\leq R(1-t)\}$. By enlarging $R$, we conclude that this convergence holds pointwise everywhere on $\Omega_{\rho'}$.

\rv{To show~\eqref{e.cvg_L^1}, we will use arguments similar to those in the proof of Proposition~\ref{p.uniqueness}.}
Henceforth, we denote by $C$ a positive constant independent of $n,t,h$, which may change from instance to instance. We also absorb $R$ and $\rho$ into $C$. Define $w_n = \bF_n -f$ and
\begin{align}\label{e.def_r_n}
    r_n = \partial_t\bF_n - \H_L(\nabla \bar  F_n).
\end{align}
Then, by the definition of $\H_L$ in~\eqref{e.H}, we have that
\begin{align}\label{e.w_n}
    \partial_t w_n = b_n\cdot \nabla w_n+r_n
\end{align}
where
\begin{align}\label{e.def_b_n}
    b_n =(b_{n,1},\ b_{n,2})= \frac{2}{\alpha_{L-1}}(\partial_2 f,\ \partial_1 \bF_n).
\end{align}
For $\delta\in(0,1)$, let $\phi_\delta:\R\to \R_+$ be given by
\begin{align}\label{e.def_phi_delta}
    \phi_\delta(x)=\left(\delta+x^2\right)^\frac{1}{2},\quad\forall x\in \R,
\end{align}
which serves as a smooth approximation of the absolute value. 
Take $v_n = \phi_\delta(w_n)$ and multiply both sides of~\eqref{e.w_n} by $\phi'_\delta(w_n)$ to see
\begin{align}\label{e.d_tv_n}
    \partial_t v_n = b_n\cdot \nabla v_n +\phi'_\delta(w_n)r_n
\end{align}
The Lipschitzness of $f$ and that of $\bar F_n$ uniform in $n$ due to~\eqref{e.partial_1_F_n} and~\eqref{e.partial_2_F_n} imply that,  uniformly in $n$, $\delta$,
\begin{align}\label{e.|grad_v|<C}
    |\nabla v_n|\leq C.
\end{align}
\rv{Using the Lipschitzness of the two functions, the boundedness of $\Omega_{\rho\wedge\rho_n}\cap\{h_2\leq R\}$, and $\lim_{n\to\infty}\bar F_n(0,0)=\psi(0)=f(0,0)$ due to assumption~\eqref{item:initial_cvg}, we have that}
\begin{align}\label{e.F_n-f_bdd}
    \sup_{\Omega_{\rho\wedge\rho_n}\cap\{h_2\leq R\}}|\bar F_n - f|\leq C 
\end{align}
uniformly in $n$, which implies that, uniformly in $n$, $\delta$,
\begin{align}\label{e.v_n_bdd}
    \sup_{\Omega_{\rho\wedge\rho_n}\cap\{h_2\leq R\}}|v_n|\leq C .
\end{align}
Recall the mollifier $\xi_\eps$ given in~\eqref{e.mollifier} and that the mollification is well-defined on domain $\Omega_{\rho\wedge\rho_n,\eps}$ described in~\eqref{e.Omega_rho_eps}. Let us regularize $b_n$ by setting $b^\eps_{n,i} = b_{n,i}*\xi_\eps  $, with the convolution taken in $h$. For $(t,h)\in\Omega_{\rho\wedge\rho_n,\eps}$, we can rewrite~\eqref{e.d_tv_n} as 
\begin{align}\label{e.v_n_eq}
    \partial_t v_n = \div(v_nb^\eps_n) - v_n\div b^\eps_n + (b_n - b^\eps_n)\cdot \nabla v_n +\phi'_\delta(w_n)r_n.
\end{align}
\rv{Note that~\eqref{e.v_n_eq} is derived in the same way as that for~\eqref{e.partial_t_v}.}

By~\eqref{e.def_b_n},~\eqref{e.partial_1_F_n},~\eqref{e.partial_2_F_n} and \rv{condition~\eqref{item.1_der_weak_sol} of} Definition~\ref{def.weak_sol}, there is $C>0$ such that the following hold for all $n$, all $\eps \in (0,1)$ and all $ (t,h)\in\Omega_{\rho\wedge\rho_n,\eps}$,
\begin{gather}
    \|b_n -b_n^\eps\|_\infty=  o_\eps(1);\label{e.b_n-b_n^eps}\\
    \|b_n^\eps\|_\infty \leq \|b_n\|_\infty \leq C;\notag\\
    b_{n,1}^\eps\in [ 0,\rho],\quad b_{n,2}^\eps\in [ 0,C] .\label{e.b_n_range}
\end{gather}
Using~\eqref{e.2nd_d_barF_lbd} and \rv{condition~\eqref{item:3} of} Definition~\ref{def.weak_sol}, we also have that, for $(t,h)\in\Omega_{\rho\wedge\rho_n,\eps}$,
\begin{align}\label{e.div_b_n^eps_lower_bound}
    \div b_n^\eps =\frac{2}{\alpha_{L-1}}\left( \partial_1 \partial_2 \big( f *\xi_\eps\big)+ \partial_1 \partial_2 \big(\bF_n*\xi_\eps\big)\right) \geq 0. 
\end{align}

Fix $R > \sup_{n,\eps}\|b_n^\eps\|_\infty$. In the following, we absorb $R$ into $C$. Let $\eta>0$ be specified later. Consider the following sets, indexed by $t\in[0,1-\frac{2}{\rho\wedge\rho_n}\eta]$,
\begin{align}
    D_t & = [\eta,(\rho\wedge\rho_n)(1-t)-\eta]\times [\eta,R(1-t)],\label{e.D_t_cvg}\\
    \Gamma_{1,t}&= [\eta,(\rho\wedge\rho_n)(1-t)-\eta]\times \{R(1-t)\},\nonumber\\
    \Gamma_{2,t}&= \{(\rho\wedge\rho_n)(1-t)-\eta\}\times [\eta, R(1-t)],\nonumber
\end{align}
where, for simplicity, we suppressed the dependence on $n,\eta$ in the notation.

Let us consider the object
\begin{align}\label{e.def_J(t)}
    J_\delta(t) = \int_{D_t} v_n(t,h)\d h = \int_{D_t} \phi_\delta\big(w_n(t,h)\big)\d h.
\end{align}
Choose $\eps<\eta$ to ensure that $\bigcup_{t\in[0,1-\frac{2}{\rho\wedge\rho_n}\eta]}(\{t\}\times D_t)\subset \Omega_{\rho\wedge\rho_n,\eps}$. Differentiate $J_\delta(t)$ in $t$ and use~\eqref{e.v_n_eq} to see
\begin{align*}
    \frac{\d}{\d t}J_\delta(t) &= \int_{D_t} \partial_t  v_n - R\int_{\Gamma_{1,t}}  v_n - \rho\wedge \rho_n\int_{\Gamma_{2,t}}  v_n\\
    &= \int_{\Gamma_{1,t}}( \mathbf{n}\cdot b^\eps_n -R)  v_n + \int_{\Gamma_{2,t}}( \mathbf{n}\cdot b^\eps_n -\rho\wedge \rho_n)  v_n \\ &  \quad +\int_{\partial D_t \setminus (\Gamma_{1,t}\cup\Gamma_{2,t})}( \mathbf{n}\cdot b^\eps_n) v_n  + \int_{D_t}\Big(-v_n \div  b_n^\eps + (b_n-b_n^\eps)\cdot \nabla v_n + \phi'_\delta(w_n)r_n\Big)
\end{align*}
\rv{where $\mathbf{n}$ is the outer normal vector.}
Here in the second identity, we used integration by parts on the integral of $\div(v_nb^\eps_n)$. The first integral on the second line is nonpositive due to the choice of $R$. 
\rv{The} second integral on that line is bounded from above by $C|\rho_n-\rho|$ due to~\eqref{e.v_n_bdd},~\eqref{e.b_n_range} and the fact that on $\Gamma_{2,t}$ the outer normal $\mathbf n=(1,0)$. 
On the last line of the display, the first integral is nonpositive due to \rv{the fact} that \rv{$-\mathbf{n}\in \R^2_+$} on $\partial D_t \setminus(\Gamma_{1,t}\cup\Gamma_{2,t})$, and~\eqref{e.b_n_range}. It is clear from~\eqref{e.def_phi_delta} that $\|\phi'_\delta\|_\infty\leq 1$. By this,~\eqref{e.|grad_v|<C},~\eqref{e.b_n-b_n^eps},~\eqref{e.div_b_n^eps_lower_bound} \rv{and the positivity of $v_n$}, the integrand in the last integral is bounded from above by $C(o_\eps(1)+|r_n|)$. Therefore, sending $\eps\to0$, we conclude that, for $t\in[0,1-\frac{2}{\rho\wedge\rho_n}\eta]$,
\begin{align}\label{e.dt_J_n}
    \frac{\d}{\d t}J_\delta(t)\leq C|\rho_n-\rho| + \int_{D_t}|r_n|.
\end{align}
Recall the definition of $r_n$ in~\eqref{e.def_r_n}. Proposition~\ref{p.approx_hj} gives an upper bound for $|r_n|$, which along with Jensen's inequality gives that
\begin{align}\label{e.int_rn_up_bdd_1}
    \int_{D_t}|r_n| \leq C\bigg(\int_{D_t}\frac{1}{n}\partial^2_2\bar F_n +  \E\int_{D_t}\big(\partial_2 F_n -  \partial_2\bar F_n\big)^2
    \bigg)^\frac{1}{2}+\rv{\int_{D_t}}a_n
\end{align}
for $a_n$ bounded as in~\eqref{e.a_n_bd}.
In view of~\eqref{e.partial_2_F_n}, the first integral on the right-hand side of~\eqref{e.int_rn_up_bdd_1} can be bounded by $Cn^{-1}$.
For the \rv{second} integral in~\eqref{e.int_rn_up_bdd_1}, we will show that
\begin{align}\label{e.derivative_concentration}
    \sup_{t\in[0,1-\frac{2}{\rho\wedge\rho_n}\eta]}\E\int_{D_t}\big|\partial_2( F_n -  \bF_n)\big|^2\leq \Delta_{1,n}^2\eta^{-\frac{1}{2}},
\end{align}
for some $\Delta_{1,n}$ converging to $0$ as $n\to\infty$. 
These estimates imply that
\begin{align*}\int_{D_t}|r_n| \leq C(n^{-\frac{1}{2}}+\Delta_{1,n}\eta^{-\frac{1}{4}}+\rv{\int_{D_t}}a_n).
\end{align*}

This along with~\eqref{e.dt_J_n} implies that
\begin{align}\label{e.J_delta(t)_<}
    J_\delta(t) \leq J_\delta(0)+ C\left(|\rho_n-\rho|+n^{-\frac{1}{2}}+\Delta_{1,n}\eta^{-\frac{1}{4}}+\rv{\int_{D_t}}a_n\right),\quad t\in\left[0,1-\frac{2}{\rho\wedge\rho_n}\eta\right].
\end{align}
\rv{Setting
\begin{align*}
    \Delta_{2,n} = \limsup_{\delta\to0}J_\delta(0) + |\rho_n-\rho|+n^{-\frac{1}{2}}+\rv{\int_{D_t}}a_n,
\end{align*}
we want to show $\lim_{n\to0}\Delta_{2,n} =0$.
}
By~\eqref{e.def_phi_delta} and~\eqref{e.def_J(t)}, we have that
\begin{align*}
   \lim_{\delta\to0} J_\delta(0)= \int_{D_0}\big|\bF_n(0,h)-f(0,h)\big|\d h
\end{align*}
which converges to $0$ as $n\to\infty$ by assumption~\eqref{item:initial_cvg},~\eqref{e.F_n-f_bdd} and the bounded convergence theorem.
Note that $\lim_{n\to\infty}|\rho_n-\rho|=0$ by assumption~\eqref{item:rho_cvg} and $\lim_{n\to\infty}\rv{\int_{D_t}}a_n=0$ due to~\eqref{e.a_n_bd}, assumptions~\eqref{i.concent_norm} and~\eqref{item:concentration}, and~\eqref{e.n_l/n}.
Hence, \rv{$\Delta_{2,n}$ decays to $0$ as $n\to\infty$.} 

Now, sending $\delta\to0$ in~\eqref{e.J_delta(t)_<}, we arrive at
\begin{align*}
     \sup_{t\in[0,1-\frac{2}{\rho\wedge\rho_n}\eta]}\int_{D_t}\left|\bF_n(t,h)-f(t,h)\right|\d h \leq C \big(\Delta_{1,n}\eta^{-\frac{1}{4}} + \Delta_{2,n}\big), 
\end{align*}
We want to extend the above result from integrating over $D_t$ to $E_{n,t}(R)$ for $t\in[0,1]$. The definitions of $E_{n,t}(R)$ in~\eqref{e.E_n,t} and $D_t$ in~\eqref{e.D_t_cvg} give that
\begin{align*}
    |E_{n,t}(R)\setminus D_t|&\leq C\eta,\quad\forall t\in \left[0,\ 1-\frac{2}{\rho\wedge\rho_n}\eta\right],
    \\
    |E_{n,t}(R)|&\leq C\eta,\quad\forall t\in \left[1-\frac{2}{\rho\wedge\rho_n}\eta,\ 1 \right].
\end{align*}
These along with~\eqref{e.F_n-f_bdd} yield that
\begin{align*}
    \sup_{t\in[0,1-\frac{2}{\rho\wedge\rho_n}\eta]}\int_{E_{n,t}(R)\setminus D_t}\big|\bF_n(t,h) - f(t,h)\big|\d h  &\leq C\eta,
    \\
    \sup_{t\in [1-\frac{2}{\rho\wedge\rho_n}\eta,1]}\int_{E_{n,t}(R)}\big|\bF_n(t,h) - f(t,h)\big|\d h  &\leq C\eta.
\end{align*}
Therefore, we obtain that
\begin{align*}
    \sup_{t\in[0,1]}\int_{E_{n,t}(R)}\big|\bF_n(t,h)-f(t,h)\big|\d h\leq C \big(\eta+\Delta_{1,n}\eta^{-\frac{1}{4}} + \Delta_{2,n}\big).
\end{align*}
Insert $\eta = \Delta_{1,n}^\frac{4}{5}$ into the above display to see that the right-hand side of the above is bounded by $C(\Delta_{1,n}^\frac{4}{5}+ \Delta_{2,n})$, which gives the desired result~\eqref{e.cvg_L^1}.

It remains to verify~\eqref{e.derivative_concentration}.

\subsubsection{Proof of~\eqref{e.derivative_concentration}}
By writing
\begin{align}\label{e.Eint||}
    \E\int_{D_t}\big|\partial_2(F_n - \bF_n)\big|^2 = \int_{\eta}^{(\rho\wedge\rho_n)(1-t)}\left(\E\int_{\eta}^{R(1-t)} \big|\partial_2(F_n - \bF_n)\big|^2 \d h_2\right) \d h_1,
\end{align}
it suffices to show that the term inside the parentheses is $o(1)\eta^{-\frac{1}{2}}$ uniformly in $t, h_1$. Now, let us fix any $(t,h_1)$ and investigate the integration with respect to $h_2$. 
Integration by parts yields that
\begin{align}
    \int_{\eta}^{R(1-t)}\big|\partial_2(F_n - \bF_n)\big|^2= (F_n - \bF_n)\partial_2(F_n - \bF_n)\big|_{h_2 = R(1-t)} - (F_n - \bF_n)\partial_2(F_n - \bF_n)\big|_{h_2 = \eta}\notag 
    \\
    -  \int_{\eta}^{R(1-t)}(F_n - \bF_n)\partial^2_2(F_n - \bF_n)\notag
    \\
    \leq \|F_n-\bar F_n\|_{L_{h_2}^\infty([0,R])}\bigg(\Big|\partial_2(F_n - \bF_n)\big|_{h_2 = R(1-t)}\Big|+\Big|\partial_2(F_n - \bF_n)\big|_{h_2 = \eta}\Big| \label{e.IBP_step}
    \\
    +\int_{\eta}^{R(1-t)}\big|\partial^2_2(F_n - \bF_n)\big|\bigg). \notag
\end{align}
Let us estimate the last integral. \rv{Recall from~\eqref{e.2nd_d_barF_lbd} and~\eqref{e.2nd_d_F_lbd} that}
\begin{align*}
    \partial^2_2 \bF_n\geq 0,\qquad \partial^2_2 F_n + Cn^{-\frac{1}{2}}\htwo^{-\frac{3}{2}}|Z'|\geq 0
\end{align*}
\rv{for an $n_{L-1}$-dimensional standard Gaussian vector $Z'$.}
This implies that
\begin{align*}
    \int_\eta^{R(1-t)}\big|\partial_2^2\rv{(F_n -\bF_n)}\big|&\leq \int_\eta^{R(1-t)} \big|\partial_2^2 F_n\big|+\big|\partial_2^2\bF_n\big|\\
    &\leq \int_\eta^{R(1-t)} \big(\partial_2^2 F_n + \partial_2^2\bF_n\big) + \int_\eta^{R(1-t)} 2Cn^{-\frac{1}{2}}\htwo^{-\frac{3}{2}}|Z'|.
\end{align*}
Applying integration by parts to the first integral after the second inequality gives that
\begin{align*}
     \int_\eta^{R(1-t)}&\big|\partial_2^2\rv{(F_n -\bF_n)}\big| 
     \\
     &\leq \big(\big|\partial_2 F_n\big|+\big|\partial_2\bF_n\big|\big)\Big|_{h_2 =R(1-t)}- \big(\big|\partial_2 F_n\big|+\big|\partial_2\bF_n\big|\big)\Big|_{h_2 =\eta} + Cn^{-\frac{1}{2}}\eta^{-\frac{1}{2}}|Z'|\\
     &\leq C(1+ n^{-\frac{1}{2}}\eta^{-\frac{1}{2}}|Z'|)
\end{align*}
where the last inequality follows from the estimates of $\partial_2 \bar F_n$ in~\eqref{e.partial_2_F_n} and $\partial_2 F_n$ in~\eqref{e.1st_der_F_n_est}.
Insert estimates~\eqref{e.partial_2_F_n} and ~\eqref{e.1st_der_F_n_est}, and the above display into~\eqref{e.IBP_step} to get that
\begin{align*}
    \int_\eta^{R(1-t)}\big|\partial_2(F_n - \bF_n)\big|^2\leq C\|F_n-\bF_n\|_{L_{h_2}^\infty([0,R])}\Big(1+n^{-\frac{1}{2}}\eta^{-\frac{1}{2}}|Z'|\Big).
\end{align*}
\rv{
We set
\begin{align*}
    \Delta_{1,n} = \sup_{\substack{t\in[0,1]\\ h_1\in[0,\rho_n(1-t)]}}\E\left[\|F_n-\bF_n\|_{L_{h_2}^\infty([0,R])}^2\right]^\frac{1}{4}
\end{align*}
which decays to zero by assumption~\eqref{item:concentration}.
Invoking the Cauchy--Schwarz inequality, we get
\begin{align*}
    \E\int_\eta^{R(1-t)}\big|\partial_2(F_n - \bF_n)\big|^2
    &\leq C\E\left[\|F_n-\bF_n\|_{L_{h_2}^\infty([0,R])}^2\right]^\frac{1}{2}\E\left[\left(1+n^{-\frac{1}{2}}\eta^{-\frac{1}{2}}|Z'|\right)^2\right]^\frac{1}{2}
    \\
    &\leq C \Delta^2_{1,n}\eta^{-\frac{1}{2}}.
\end{align*}
Inserting this to~\eqref{e.Eint||} and absorbing $C$ into $\Delta^{1,n}$, we obtain~\eqref{e.derivative_concentration}.}

\section{Auxiliary results}\label{s.aux}

We collect proofs of Lemma~\ref{l.cvg_rho} which verifies~\eqref{e.rho_cvg}, Lemma~\ref{l.cvg_norm} which gives the concentration of $\frac{1}{n_l}\left|X^\pr{l}\right|^2$, and Lemma~\ref{l.concent} which shows that the concentration condition~\eqref{item:concentration} in Proposition~\ref{p.cvg} always holds.

\subsection{Convergence of the averaged norm}

Recall $\rho_{l,n}$ from~\eqref{e.rho_l}.

\begin{lemma}\label{l.cvg_rho}
Assume~\ref{i.assump_iid_bdd}--\ref{i.assump_Phi_Gaussian} for some $L\in \N$. For each $l\in\{ 0,1,\dots,L\}$,~\eqref{e.rho_cvg} holds for
$\rho_l$ defined iteratively by
\begin{align}
    \rho_0 &= \E |X_1|^2\notag\\
    \rho_l &= \E \left|\varphi_l\left(\sqrt{\rho_{l-1}}\rv{N},A^\pr{l}_1\right)\right|^2 \label{e.formula_rho_l}
\end{align}
\rv{where $N$ is an independent standard Gaussian random variable}.
\end{lemma}

Examining the proof below, we can see that the lemma is still valid (with $\rho_0$ defined as a limit) if we replace~\ref{i.assump_iid_bdd} and~\ref{i.assump_varphi_l_2^l_diff} by weaker assumptions that $\frac{1}{n}|X|^2$ converges in probability together with~\ref{i.assump_w_|X|<sqrt_n}, and that $\varphi_l$ is Lipschitz for all $l$.

\begin{proof}
\rv{In view of~\eqref{e.rho_l}, to show~\eqref{e.rho_cvg}, it} suffices to show that
\begin{align}\label{e.var_avg_L^2}
    \lim_{n\to\infty}\E \left|\frac{\left|X^\pr{l}\right|^2}{n_l}-\rho_l\right|^2=0.
\end{align}

\rv{Let us first verify~\eqref{e.var_avg_L^2} for $l=0$. Recall that we have assumed that $n_0=n$ (see the line above~\eqref{e.n_l/n}), and that $X^\pr{0}=X$ consists of $n$ bounded i.i.d.\ entries (see the beginning of Section~\ref{s.setting} above~\eqref{e.X^l}). Hence, we have $\rho_0 = \frac{1}{n_0}\E|X^{(0)}|^2$ which yields~\eqref{e.var_avg_L^2} for $l=0$.} 

We proceed by induction.
Now, we assume that~\eqref{e.var_avg_L^2} holds for $l-1$. Let us denote by $\E^\pr{l}$ the expectation with respect to $\Phi^\pr{l}$ and $A^\pr{l}$. We start by writing
\begin{align}\label{eq:split_rho_cvg}
    \E\left|\frac{\left|X^\pr{l}\right|^2}{n_l}- \rho_l\right|^2\leq 2 \E\left|\frac{\left|X^\pr{l}\right|^2}{n_l}- \E^\pr{l}\frac{\left|X^\pr{l}\right|^2}{n_l}\right|^2 + 2 \E\left| \E^\pr{l}\frac{\left|X^\pr{l}\right|^2}{n_l}- \rho_l\right|^2 \rv{= \mathtt{I}+\mathtt{II}}.
\end{align}
Let us estimate $\rv{\mathtt{I}}$. It is clear from~\eqref{e.X^l} that, conditioned on $X^\pr{l-1}$, $(|X^\pr{l}_j|^2)_{j=1}^{n_l}$ is a sequence of i.i.d.\ random \rv{variables}. Hence,
\begin{align*}
    \rv{\mathtt{I}} =\frac{2}{n^2_l}\sum_{j=1}^{n_l}\rv{\E}\left|\left|X^\pr{l}_j\right|^2-\E^\pr{l}\left|X^\pr{l}_j\right|^2\right|^2 \rv{\leq Cn_l^{-1}}
\end{align*}
\rv{where the last inequality follows from the boundedness of $X^\pr{l}_j$.}

Now, we turn to $\rv{\mathtt{II}}$. Using~\eqref{e.X^l} \rv{(also see the display above~\eqref{e.X^l} for its expression in coordinates)}, we can compute that
\begin{align*}
    \E^\pr{l}\frac{\left|X^\pr{l}\right|^2}{n_l} 
    &= \rv{\frac{1}{n_l}\sum_{j=1}^{n_l}\E^\pr{l}\left|\varphi_l\left(\frac{1}{\sqrt{n_{l-1}}}\Phi^\pr{l}_{j,\,\cdot}X^\pr{l-1},\ A^\pr{l}_j\right)\right|^2 }
    \\
    &= \rv{\E^\pr{l}\left|\varphi_l\left(\frac{1}{\sqrt{n_{l-1}}}\Phi^\pr{l}_{1,\,\cdot}X^\pr{l-1},\ A^\pr{l}_1\right)\right|^2.}
\end{align*}
\rv{Note that, conditioned on $X^\pr{l-1}$, we have that $\frac{1}{\sqrt{n_{l-1}}}\Phi^\pr{l}_{1,\,\cdot}X^\pr{l-1}$ is a centered Gaussian variable with variance $\frac{1}{n_{l-1}}|X^\pr{l-1}|^2$. Setting
\begin{align*}
    g(\sigma) = \E\left |\varphi_l\left(\sqrt{\sigma}N,A_1^\pr{l}\right)\right|^2,
\end{align*}
for $\sigma\geq 0$ where $N$ is a standard Gaussian random variable, we obtain that}
\begin{align*}
    \E^\pr{l}\frac{\left|X^\pr{l}\right|^2}{n_l}  = g\left(\frac{\left|X^\pr{l-1}\right|^2}{n_{l-1}}\right) .
\end{align*}
\rv{In this notation, due to~\eqref{e.formula_rho_l}, we also have $\rho_l = g(\rho_{l-1})$. Hence, we get that}
\begin{align*}
    \rv{\mathtt{II} = 2\E\left|g\left(\frac{\left|X^\pr{l-1}\right|^2}{n_{l-1}}\right) - g(\rho_{l-1})\right|^2.}
\end{align*}
Since $\varphi_l$ is assumed to have bounded derivatives, we can see that $g$ is $\frac{1}{2}$-H\"older continuous.
\rv{Using this, we get that}
\begin{align*}
    \rv{\mathtt{II}} \leq C\E\left|\frac{\left|X^\pr{l-1}\right|^2}{n_{l-1}} - \rho_{l-1}\right|
\end{align*}
which converges to $0$ due to the induction assumption~\eqref{e.var_avg_L^2} for $l-1$.

\rv{Now, since both $\mathtt{I}$ and $\mathtt{II}$ in~\eqref{eq:split_rho_cvg} converge to $0$, we have completed the induction step showing that~\eqref{e.var_avg_L^2} holds for $l$ and thus completed the proof.}
\end{proof}

\subsection{Concentration of the norm}

The goal is to show the following lemma.

\begin{lemma}\label{l.cvg_norm}
Assume~\ref{i.assump_iid_bdd}--\ref{i.assump_Phi_Gaussian} for some $L\in \N$. There is a constant $C>0$ such that, for every $n\in\N$,
\begin{align*}
    \var \left[\frac{1}{n_{L}}\left|X^\pr{L}\right|^2\right]\leq \frac{C}{n}.
\end{align*}
\end{lemma}

To prove this, we need a classic result on concentration.

\begin{lemma}\label{l.fin_diff}
Let $A_1, A_2,\dots,A_n$ be independent random variables with values in some space $\mathcal X$. Suppose that a function $f:\mathcal X^n\to \R$ satisfies
\begin{align*}
    \sup_{1\leq i\leq n}\sup_{\substack{a_1,\dots, a_n,\\a'_i\in\mathcal X}}|f(a_1,\dots,a_n)- f(a_1,\dots,a_{i-1},a'_i,a_{i+1},\dots, a_n)|\leq c
\end{align*}
for some $c>0$. Then, $\var[f(A)]\leq \frac{1}{4}nc^2$.
\end{lemma}

This is a corollary of the Efron--Stein inequality. We refer to \cite[Corollary~3.2]{boucheron2013concentration} for a proof.

\begin{proof}[Proof of Lemma~\ref{l.cvg_norm}]
Setting
\begin{align*}
    g_L(x) = \frac{1}{n_L}\left|x\right|^2,\quad\forall x \in \R^{n_L},
\end{align*}
we have that
\begin{align}\label{e.g_L(X^L)}
    g_L\left(X^\pr{L}\right) = \frac{1}{n_L}\left|X^\pr{L}\right|^2.
\end{align}
For $l\in\{0,1,\dots, L-1\}$, we can iteratively define
\begin{align}\label{e.def_g_l}
    g_{l}(x) = \E \left[g_{l+1}\left(\varphi_{l+1}\left(\frac{1}{\sqrt{n_l}}\Phi^\pr{l+1}x , A^\pr{l+1}\right)\right)\right], \quad \forall x \in \R^{n_l}.
\end{align}
Due to~\eqref{e.X^l}, this implies that
\begin{align*}
    g_l\left(X^\pr{l}\right) = \E\left[g_{l+1}\left(X^\pr{l+1}\right)\Big|X^\pr{l}\right].
\end{align*}
For convenience, we also set
\begin{align*}
    X^\pr{-1} = 0,\qquad g_{-1}(0) = \E \left[g_{0}\left(X^\pr{0}\right)\right].
\end{align*}
Iterating these yields that
\begin{align*}
    g_{-1}\left(X^\pr{-1}\right) = \E\left[g_L\left(X^\pr{L}\right)\right] = \E\left[\frac{1}{n_L}\left|X^\pr{L}\right|^2\right],
\end{align*}
which along with~\eqref{e.g_L(X^L)} gives that
\begin{align*}
    \var \left[\frac{1}{n_{L}}\left|X^\pr{L}\right|^2\right] &= \E\left[\left(g_L\left(X^\pr{L}\right)\right)^2 - \left(g_{-1}\left(X^\pr{-1}\right)\right)^2\right]\\
    & = \sum_{l=0}^L \E \left[\left(g_l\left(X^\pr{l}\right)\right)^2 - \left(g_{l-1}\left(X^\pr{l-1}\right)\right)^2\right]\\
    & = \sum_{l=0}^L \E\left[\left(g_l\left(X^\pr{l}\right) -  \E\left[g_l\left(X^\pr{l}\right) \Big| X^\pr{l-1}\right] \right)^2\right].
\end{align*}
Then, the desired result follows if we can show that, for all $l\in\{0,1,\dots,L\}$,
\begin{align}\label{e.var_g_l}
    \E\left[\left(g_l\left(X^\pr{l}\right) -  \E\left[g_l\left(X^\pr{l}\right) \Big| X^\pr{l-1}\right] \right)^2\right] \leq \frac{C}{n}.
\end{align}

\rv{We first verify~\eqref{e.var_g_l} for $l=L$. 
For $j\in\{1,\dots,n_L\}$, we set
\begin{align*}
    V_j = \frac{1}{n_L}\left|X^\pr{L}_j\right|^2 -  \E^\pr{L}\frac{1}{n_L}\left|X^\pr{L}_j\right|^2.
\end{align*}
Recall $g_L\left(X^\pr{L}\right) $ from~\eqref{e.g_L(X^L)}. Let $\E^\pr{L}$ be the expectation with respect to $\Phi^\pr{L}$ and $A^\pr{L}$. Due to the relation between $X^\pr{L}$ and $X^\pr{L-1}$ given in~\eqref{e.X^l}, we have that
\begin{align*}
    g_L\left(X^\pr{L}\right)  -\E\left[g_L\left(X^\pr{L}\right) \Big| X^\pr{L-1}\right] = \sum_{j=1}^{n_L}V_j.
\end{align*}
By~\eqref{e.X^l}, conditioned on $X^\pr{L-1}$ (or equivalently, under $\E^\pr{L}$), $X^\pr{L}$ has i.i.d.\ entries. Hence, $V_j$'s are i.i.d.\ with zero mean under $\E^\pr{L}$. In particular, this implies that $\E^\pr{L}V_iV_j =0$ if $i\neq j$. Also, since $\left|X^\pr{L}_j\right|$ is bounded by $\|\varphi_L\|_\infty$, we have $|V_j|\leq \frac{C}{n_L}$. Using these, we can rewrite the left-hand side of~\eqref{e.var_g_l} at $l=L$ as
\begin{align*}
    \E\left[\left(\sum_{j=1}^{n_L}V_j\right)^2\right] = \sum_{j=1}^{n_L}\E\left[V_j^2\right] \leq \frac{C}{n_L}.
\end{align*}
This verifies~\eqref{e.var_g_l} for $l=L$.}

Now, let $l\leq L-1$. Due to~\eqref{e.X^l}, $X^\pr{l}$ has i.i.d.\ entries when conditioned on $X^\pr{l-1}$. Recall the notation~\eqref{e.A^[l,l'],Phi^[l,l']}.
Due to~\eqref{e.X^l}, viewing $X^\pr{L}$ as a deterministic function of $\Phi^{[l+1,m]}$, $A^{[l+1,m]}$ and $X^\pr{l}$, and using~\eqref{e.def_g_l}, we can check inductively that
\begin{align*}
    g_l\left(X^\pr{l}\right) = \E\left[\frac{1}{n_{L}}\left|X^\pr{L}\right|^2\bigg|X^\pr{l}\right].
\end{align*}
Then, using~\eqref{e.X^l} and the chain rule, we can compute that
for $i_l\in\{1,2,\dots,n_l\}$,
\begin{align}\label{e.partial_g_l}
    \frac{\partial g_l\left(X^\pr{l}\right)}{\partial X^\pr{l}_{i_l}} = \frac{2}{n_L}\sum_{\mathbf{i}} \E \left[ \dot\varphi^\pr{l+1}_{i_{l+1}} \dot \varphi^\pr{l+2}_{i_{l+2}}\cdots\dot \varphi^\pr{L}_{i_{L}}\frac{\Phi^\pr{l+1}_{i_{l+1},i_l}}{\sqrt{n_{l}}}\ \frac{\Phi^\pr{l+2}_{i_{l+2},i_{l+1}}}{\sqrt{n_{l+1}}}\ \cdots \ \frac{\Phi^\pr{L-1}_{i_{L-1},i_{L-2}}}{\sqrt{n_{L-2}}}\ \frac{\Phi^\pr{L}_{i_{L},i_{L-1}}}{\sqrt{n_{L-1}}}\Bigg|X^\pr{l}\right],
\end{align}
where the summation is over
\begin{align}\label{e.bf_i_concent_norm}
    \mathbf{i}&=(i_{l+1},i_{l+2},\dots, i_L)\in \prod_{m=l+1}^L \{1,\dots,n_m\}
\end{align}
and
\begin{align*}
    \dot \varphi^\pr{m}_{i_m} =\varphi'_{m}\left(\frac{1}{\sqrt{n_{m-1}}}\left(\Phi^\pr{m} X^\pr{m-1}\right)_{i_m},A^\pr{m}_{i_m}\right), \quad \forall i_m\in \{1,\dots, n_m\}.
\end{align*}
The derivative on $\varphi_m$ is with respect to its first argument. 

To proceed, we want to perform the Gaussian integration by parts \rv{on} every $\Phi^\pr{m}_{i_m,i_{m-1}}$ in every summand on the right-hand side of~\eqref{e.partial_g_l}. The \rv{heuristic} is that since $\Phi^\pr{m}_{i_m,i_{m-1}}$ always appears in the form of $\frac{1}{\sqrt{n_{m-1}}} \Phi^\pr{m}X^\pr{m-1}$, we expect to obtain an extra factor of order $n^{-\frac{1}{2}}$ after performing one instance of integration by parts. However, due to the layered structure given in~\eqref{e.X^l} and the chain rule, the differentiation involved in the process of integration by parts may produce new terms, the number of which grows as $n$ increases. To counter this effect, we need to perform more instances of integration by parts on Gaussian variables introduced by the chain rule. 

The above \rv{heuristic} is made rigorous by Corollary~\ref{c.GIBP} which follows from a more general result Lemma~\ref{l.multip_GIBP}. Applying Corollary~\ref{c.GIBP} \rv{(with $L,l+1$ substituted for $k,l$ therein) and noticing the square roots of $n_k$'s in the denominator}, we obtain that each summand in~\eqref{e.partial_g_l} has its absolute value bounded by $C n^{-(L-l)}$ where $C$ is absolute. Due to~\eqref{e.bf_i_concent_norm}, the summation in~\eqref{e.partial_g_l} is over $O(n^{L-l})$ many terms. Therefore, we conclude that, for each $i_l\in\{1,2,\dots, n_l\}$,
\begin{align*}
    \left| \frac{\partial g_l\left(X^\pr{l}\right)}{\partial X^\pr{l}_{i_l}} \right|\leq \frac{C}{n}.
\end{align*}
Invoking Lemma~\ref{l.fin_diff} \rv{(with $n_L$, $\R$, $X^\pr{l}$ (conditioned on $X^\pr{l-1}$), $g_l$, $\frac{C}{n}$ substituted for $n$, $\mathcal{X}$, $A$, $f$, $c$ therein)}, we obtain that there is a constant $C$ such that, for almost every realization of $X^\pr{l-1}$,
\begin{align*}
    \E\left[\left(g_l\left(X^\pr{l}\right) -  \E\left[g_l\left(X^\pr{l}\right) \Big| X^\pr{l-1}\right] \right)^2\Bigg| X^\pr{l-1} \right] \leq \frac{C}{n},
\end{align*}
which then gives
\eqref{e.var_g_l} and completes the proof.
\end{proof}

\subsection{Concentration of the free energy}
Recall the definitions of $\Pout_{\beta,L,n}$, $H_{\beta,L,n}$ and $F_{\beta,L,n}$ given in~\eqref{e.P_beta(y|x)},~\eqref{e.H_beta_L_n_(t,h)}, and~\eqref{e.F_beta_L_n_(t,h)}. The goal is to show the lemma below.

\begin{lemma}\label{l.concent}
Assume~\ref{i.assump_iid_bdd}--\ref{i.assump_Phi_Gaussian} for some $L\in \N$. For every $\beta\geq 0$, and $M\geq1$, there is a constant $C>0$ such that
\begin{align*}
    \sup_{t\in[0,1],\  h_1\in[0,\rho_n(1-t)]}\E\left[\left\|F_{\beta,L,n} - \bar F_{\beta,L,n}\right\|^2_{L^\infty_{h_2}([0,M])}(t,h_1)\right]\leq \frac{C}{\sqrt{n}}.
\end{align*}
\end{lemma}

The remaining part of this subsection is devoted to the proof of this lemma. In addition to Lemma~\ref{l.fin_diff}, we recall one more classic result on concentration.

\begin{lemma}\label{l.GPI}
Let $Z=(Z_1,Z_2,\dots,Z_n)$ be a standard Gaussian vector and $f:\R^n\to\R$ be a continuously differentiable function. Then $\var[f(Z)]\leq \E |\nabla f(Z)|^2$.
\end{lemma}

This result is often called the Gaussian Poincar\'e inequality, whose proof \rv{can be found in} \cite[Theorem~3.20]{boucheron2013concentration}.

Let $h_2 \in [0,M]$. In the following, $C>0$ denotes a deterministic constant independent of $n$, which may differ from line to line. We also absorb $M$ and $\beta$ into $C$. For simplicity, we write $H=H_{\beta,L,n}$ and $F= F_{\beta,L,n}$. In addition, we set
\begin{gather}\label{e.Gamma}
    \Gamma\left(s,a^\pr{L}\right)=\varphi_L\left(S,A^\pr{L}\right)-\varphi_L\left(s,a^\pr{L}\right),
\end{gather}
where $S$ and $s$ are defined in~\eqref{e.cap_S_mu} and~\eqref{e.s_mu}, respectively, and
\begin{align}\label{e.a^(L)}
    a^\pr{L} \in \R^{n_L\times k_L}
\end{align}
is of the same size as $A^\pr{L}$.
In view of~\eqref{e.P_beta(y|x)} and~\eqref{e.H_beta_L_n_(t,h)}, note that $H$ can be rewritten as
\begin{align*}
    H(x,w,a) =\log\bigg(\int e^{-\frac{1}{2}\big|\sqrt{\beta}\Gamma\left(s,a^\pr{L}\right) +Z \big|^2}\d P_{A^\pr{L}}(a^\pr{L})\bigg) +\sqrt{h_2}Y'\cdot x^\pr{L-1}-\frac{h_2}{2}\left|\rv{x}^\pr{L-1}\right|^2,
\end{align*}
where $Y'$ is given in~\eqref{e.Y'} and $a=(a^\pr{1},\cdots, a^\pr{L-1})$ appearing in $x^\pr{L-1}$ is defined in~\eqref{e.x^L-1}.
We introduce the Hamiltonian
\begin{align}
    &\hat H\left(x,w,a,a^\pr{L}\right)\notag\\
    &\qquad = -\frac{1}{2}\left(2\sqrt{\beta}Z\cdot\Gamma\left(s,a^\pr{L}\right) +\beta\left|\Gamma\left(s,a^\pr{L}\right)\right|^2\right)+ \sqrt{h_2}Y'\cdot x^\pr{L-1}-\frac{h_2}{2}\left|x^\pr{L-1}\right|^2,\label{e.hat_H}
\end{align}
and the associated free energy
\begin{gather*}
    \hat F=\frac{1}{n}\log\int e^{\hat H(x,w,a,a^\pr{L})}\d P_X(x)\d P_W(w) \d P_{A^{[1,L-1]}}(a)\d P_{A^\pr{L}}\left(a^\pr{L}\right).
\end{gather*}
Then, using these and the definition of $F$ in terms in of $H$ in~\eqref{e.F_beta_L_n_(t,h)}, we can see that
\begin{align*}
    F = \hat F - \frac{1}{2n}|Z|^2,
\end{align*}
which implies that
\begin{align}\label{e.Var(F)-Var(hat_F)}
    \var(F)\leq 2 \var (\hat F) + 2\var\bigg(\frac{1}{2n}|Z|^2\bigg)\leq 2 \var (\hat F)+ \frac{C}{n},
\end{align}
where we used the fact that $Z$ is a standard Gaussian vector in $\R^{n_L}$. Therefore, it suffices to study $\var(\hat F)$. In the \rv{following}, we denote by $\la\ \cdot\ \ra$ the Gibbs measure with Hamiltonian~$\hat H$.

Recall the notation~\eqref{e.A^[l,l'],Phi^[l,l']}. Note that $\hat F$ is a function of $Z,Z',V,W,A^\pr{L},\Phi^{[1,L]},X^\pr{L-1}$, where the dependence on $Z$ is in~\eqref{e.hat_H}; $\Phi^\pr{L},X^\pr{L-1},V,W$ appear in $S$ defined in~\eqref{e.cap_S_mu}; $X^\pr{L-1},Z'$ appear in $Y'$ defined in~\eqref{e.Y'}; $A^\pr{L}$ appears in~\eqref{e.Gamma}; $\Phi^{[1,L-1]}$ appears in $x^\pr{L-1}$ defined in~\eqref{e.x^L-1}; and finally $\Phi^\pr{L},V,x^\pr{L-1}$ appear in $s$ defined in~\eqref{e.s_mu}.

The plan is to prove the concentration of $\hat F$ conditioned on subsets of these random variables, and then combine them. The order of conditioning matters and we proceed as in \cite{gabrie2019entropy}. Lastly, to get concentration uniformly in $h_2\in[0,M]$, we will apply an $\eps$-net argument.

\subsubsection{Concentration conditioned on $V,W,A^\pr{L},\rv{\Phi^{[1,L-1]}},X^\pr{L-1}$}
Denote by $\E_{Z,Z'}$ the expectation with respect to only $Z$ and $Z'$. We want to show that
\begin{align}\label{e:concent_Z,Z'}
    \E_{Z,Z'}\left(\hat F - \E_{Z,Z'}\hat F\right)^2 \leq \frac{C}{n}, 
\end{align}
for almost every realization of other randomness.

For simplicity, we write $\Gamma=\Gamma\left(s,a^\pr{L}\right)$ from now on.
We fix any realization of other randomness. Note that $Z$ appears only in~\eqref{e.hat_H} and $Z'$ appears only in $Y'$ (defined in~\eqref{e.Y'}). Then, we can compute that
\begin{align*}
    \left|\frac{\partial \hat F}{\partial Z_j}\right|&=\frac{1}{n}\left|\la\sqrt{\beta}\Gamma_j\ra\right|\leq \frac{C}{n},\quad \forall j\in \{1,2,\dots,n_L\}\\
    \left|\frac{\partial \hat F}{\partial Z_i'}\right|&=\frac{1}{n}\left|\sqrt{h_2}\la x^\pr{L-1}_i\ra\right|\leq \frac{C}{n},\quad \forall i\in \{1,2,\dots,n_{L-1}\},
\end{align*}
where we used the boundedness of $\varphi_L$, and the boundedness of $x^\pr{L-1}$ to get the inequalities.
Hence, we have that $|\nabla_{Z,Z'}\hat F|\leq Cn^{-\frac{1}{2}}$ and thus, by Lemma~\ref{l.GPI}, obtain~\eqref{e:concent_Z,Z'}.

\subsubsection{Concentration conditioned on $A^\pr{L},\Phi^{[1,L]},X^\pr{L-1}$}

Set $\gau = (Z,Z', V,W,\Phi^\pr{L})$, and let  $\E_\gau$ be the expectation with respect to these Gaussian random variables. We want to show that, a.s.,
\begin{align}\label{e.concent_gau}
    \E_\gau \left(\E_{Z,Z'}\hat F-\E_\gau\hat F\right)^2\leq \frac{C}{n}.
\end{align}
Note that $V$ appears in both $S$ (defined in~\eqref{e.cap_S_mu}) and $s$ (defined in~\eqref{e.s_mu}) in $\Gamma$ and $W$ appears only in $S$.
Hence, in view of~\eqref{e.Gamma}, using the boundedness for the derivatives of $\varphi_L$, we can verify that
\begin{align*}
    \left|\frac{\partial \E_{Z,Z'}\hat F}{\partial V_j}\right|&=\frac{1}{n}\left|\E_{Z,Z'}\la\left(\sqrt{\beta}Z_j+\beta\Gamma_j\right)\frac{\partial  \Gamma_j}{\partial V_j}\ra\right|\leq \frac{C}{n}, \quad\forall j \in \{1,2,\dots, n_L\},\\
    \left|\frac{\partial \E_{Z,Z'}\hat F}{\partial W_j}\right|&=\frac{1}{n}\left|\E_{Z,Z'}\la\left(\sqrt{\beta}Z_j+\beta\Gamma_j\right)\frac{\partial  \Gamma_j}{\partial W_j}\ra\right|\leq \frac{C}{n},\quad\forall j \in \{1,2,\dots, n_L\}.
\end{align*}
On the other hand, $\Phi^\pr{L}$ only \rv{appears} in both $S$ and $s$. Due to the computation that
\begin{align*}
    \frac{\partial  \Gamma_j}{\partial \Phi^\pr{L}_{jk}} = \sqrt{\frac{ t}{n_{L-1}}}\left(\varphi'_L\left(S,A^\pr{L}\right)X^\pr{L-1}_k -\varphi'_L\left(s,a^\pr{L}\right)x^\pr{L-1}_k  \right),
\end{align*}
where $\varphi'_L$ is the derivative with respect to its first argument, and the boundedness of the derivatives of $\varphi_L$,
we also can show that
\begin{align*}
    \left|\frac{\partial \E_{Z,Z'}\hat F}{\partial \Phi^\pr{L}_{jk}}\right|&=\frac{1}{n}\left|\E_{Z,Z'}\la\left(\sqrt{\beta}Z_j+\beta\Gamma_j\right)\frac{\partial  \Gamma_j}{\partial \Phi^\pr{L}_{jk}}\ra\right|
    \leq \frac{C}{n^\frac{3}{2}}.
\end{align*}
for all $j\in \{1,\dots,n_L\}$ and $k\in \{1,\dots,n_{L-1}\}$.
Therefore,
\begin{align*}
    \left|\nabla_{V,W,\Phi^\pr{L}}\E_{Z,Z'}\hat F\right|^2 = \sum_{j=1}^{n_L}\left|\frac{\partial \E_{Z,Z'}\hat F}{\partial V_j}\right|^2 + \sum_{j=1}^{n_L}\left|\frac{\partial \E_{Z,Z'}\hat F}{\partial W_j}\right|^2 +\sum_{j=1}^{n_L}\sum_{k=1}^{n_{L-1}}\left|\frac{\partial \E_{Z,Z'}\hat F}{\partial \Phi^\pr{L}_{jk}}\right|^2\leq \frac{C}{n},
\end{align*}
which together with Lemma~\ref{l.GPI} implies~\eqref{e.concent_gau}.

\subsubsection{Concentration conditioned on $\Phi^{[1,L]},X^\pr{L-1}$}
Fixing any realization of other randomness, we express $\E_\gau \hat F=g(A^\pr{L})$ as a function of $A^\pr{L}$. Then, we fix a realization of $A^\pr{L}$ and let ${A'}^\pr{L}$ be another realization such that $A^\pr{L}_j = {A'}^\pr{L}_j$ for all $j$ except for some $j=i$. We want to show that there is an absolute constant $C$ such that
\begin{align}\label{e.diff_A}
    \left|g\left(A^\pr{L}\right) - g\left({A'}^\pr{L}\right)\right|\leq \frac{C}{n},
\end{align}
which by Lemma~\ref{l.fin_diff} implies that, a.s.,
\begin{align}\label{e.concent_g_A^L}
    \E_{\gau, A^\pr{L}}\left(\E_\gau\hat F - \E_{\gau, A^\pr{L}}\hat F \right)^2\leq \frac{C}{n}.
\end{align}
We denote by $\la \ \cdot \ \ra_{\hat \H}$ the Gibbs measure with $A^\pr{L}$ and $\la \ \cdot \ \ra_{\hat \H'}$ the Gibbs measure with ${A'}^\pr{L}$. Using the definition of $g$, we can verify that
\begin{align*}
    g\left(A^\pr{L}\right) - g\left({A'}^\pr{L}\right) = \frac{1}{n}\E_\gau\log\la e^{\hat H - \hat H'}\ra_{\hat H'}.
\end{align*}
By Jensen's inequality, we have that
\begin{align*}
    g\left(A^\pr{L}\right) - g\left({A'}^\pr{L}\right)\geq \frac{1}{n}\E_\gau \la \hat H - \hat H'\ra_{\hat H'}.
\end{align*}
Symmetrically,
\begin{align*}
    g\left({A'}^\pr{L}\right) - g\left({A}^\pr{L}\right)\geq \frac{1}{n}\E_\gau \la \hat H' - \hat H\ra_{\hat H}.
\end{align*}
Using~\eqref{e.Gamma},~\eqref{e.hat_H} and the definitions of $A^\pr{L}$ and ${A'}^\pr{L}$, we have that
\begin{align*}
    \hat H - \hat H' = \frac{1}{2}\left(\Gamma'_i-\Gamma_i\right)\left(2\rv{\sqrt{\beta}}Z_i +\rv{\beta}\Gamma_i+\rv{\beta}\Gamma'_i\right)
\end{align*}
where $\Gamma_i$ and $\Gamma'_i$ correspond to $A^\pr{L}$ and ${A'}^\pr{L}$, respectively. Together with the boundedness of $\Gamma,\Gamma'$, the above three displays yield~\eqref{e.diff_A} and thus imply the desired result~\eqref{e.concent_g_A^L}.

\subsubsection{Iteration}

Note that in~\eqref{e.concent_g_A^L}, we can rewrite that
\begin{align*}\E_{\gau, A^\pr{L}}\hat F  = \E\left[\hat F\Big|X^\pr{L-1},\Phi^{[1,L]}\right].
\end{align*}
To proceed, we claim that
\begin{align}
    \E\left(\E\left[\hat F\Big|X^\pr{l},\Phi^{[1,l]}\right]-\E\left[\hat F\Big|X^\pr{l-1},\Phi^{[1,l]}\right]\right)^2
    &\leq \frac{C}{n}, \qquad \forall l\in\{0,1,\dots, L-1\},\label{e.concent_itr_1}\\
    \E\left(\E\left[\hat F\Big|X^\pr{l-1},\Phi^{[1,l]}\right]-\E\left[\hat F\Big|X^\pr{l-1},\Phi^{[1,l-1]}\right]\right)^2
    &\leq \frac{C}{n},
    \qquad\forall l\in\{1,\dots, L-1\},\label{e.concent_itr_2}
\end{align}
where $X^\pr{-1}$ and $\Phi^{[1,0]}$ are understood to be constantly $0$ (or any constant).
Given the above, we can iterate these to see that
\begin{align}\label{e.concent_X_Phi}
     \E\left(\E\left[\hat F\Big|X^\pr{L-1},\Phi^{[1,L-1]}\right]-\E\left[\hat F\right]\right)^2\leq \frac{C}{n}.
\end{align}
Combining~\eqref{e.Var(F)-Var(hat_F)},~\eqref{e:concent_Z,Z'},~\eqref{e.concent_gau},~\eqref{e.concent_g_A^L} and~\eqref{e.concent_X_Phi} yields the pointwise concentration
\begin{align}\label{e.ptw_concent}
    \E \left[\left(F- \bar F\right)^2(t,h)\right]\leq \frac{C}{n},\quad\forall (t,h)\in\Omega_{\rho_n}\cap\{|h_2|\leq M\}.
\end{align}
\rv{Thus}, let us prove the assertions~\eqref{e.concent_itr_1} and~\eqref{e.concent_itr_2}.

\subsubsection{Proof of~\eqref{e.concent_itr_1}}

Due to the expression~\eqref{e.X^l} and the fact that $\hat F$ depends on $X^\pr{l-1}$ only through $X^\pr{l}$, we can see that
\begin{align*}
    \E\left[\hat F\Big|X^\pr{l},\Phi^{[1,l]}\right] = \E\left[\hat F\Big|X^\pr{l},X^\pr{l-1},\Phi^{[1,l]}\right].
\end{align*}
Also, note that $X^\pr{l}$ consists of i.i.d.\ entries when conditioned on $X^\pr{l-1}$. Hence, we want to apply Lemma~\ref{l.fin_diff}. Since each entry of $X^\pr{l}$ is bounded uniformly in $n$, to verify the condition in Lemma~\ref{l.fin_diff}, it suffices to obtain bounds for derivatives of $\tilde \E\hat F$ with respect to $X^\pr{l}$, where $\tilde \E = \E\left[\ \cdot \ |X^\pr{l},X^\pr{l-1},\Phi^{[1,l]}\right]$.

We introduce the following notation:
\begin{align}\label{e.dot_phi_notation}
    \begin{cases}
    \varphi^\pr{L}_*  =  \varphi_{L}\left(S,A^\pr{L}\right),\\
    \tilde\varphi^\pr{L}_*  =  \varphi_{L}\left(s,a^\pr{L}\right),\\
    \dot \varphi^\pr{m}  = \varphi'_{m}\left(\frac{1}{\sqrt{n_{m-1}}}\Phi^\pr{m}X^\pr{m-1},A^\pr{m}\right),\quad\forall m\in\{1,\dots,L\},\\
    \dot {\tilde\varphi}^\pr{m} = \varphi'_{m}\left(\frac{1}{\sqrt{n_{m-1}}}\Phi^\pr{m}x^\pr{m-1},a^\pr{m}\right),\quad\forall m\in\{1,\dots,L\},\\
    \dot \varphi^\pr{L}_* = \varphi'_L\left(S,A^\pr{L}\right),\\
    \dot {\tilde\varphi}^\pr{L}_* = \varphi'_L\left(s,a^\pr{L}\right),
    \end{cases}
\end{align}
where $\varphi'_m$ is the derivative with respect to its first argument.
For $i_l\in\{1,\dots,n_l\}$, we can compute that
\begin{align}
    \frac{\partial \tilde \E \hat F}{\partial X^\pr{l}_{i_l}}&= -\frac{1}{n}\tilde\E \la\left(\sqrt{\beta}Z+\beta\Gamma\right)\cdot \partial_{X^\pr{l}_{i_l}} \Gamma\ra +\frac{1}{n}\tilde \E\la h_2  x^\pr{L-1}\cdot \partial_{X^\pr{l}_{i_l}} X^\pr{L-1}\ra \notag
    \\
    & = -\frac{\sqrt{t}}{n}\sum_{\mathbf{i}} \tilde\E \left[ \la \sqrt{\beta} Z_{i_L}+\beta \left(\varphi^\pr{L}_{*,i_L}-\tilde\varphi^\pr{L}_{*,i_L}\right)\ra  \dot \varphi^\pr{l+1}_{i_{l+1}}\cdots \dot \varphi^\pr{L-1}_{i_{L-1}}\dot\varphi^\pr{L}_{*,i_{L}}\ \frac{\Phi^\pr{l+1}_{i_{l+1},i_l}}{\sqrt{n_{l}}}\ \cdots \   \frac{\Phi^\pr{L}_{i_{L},i_{L-1}}}{\sqrt{n_{L-1}}}\right]\label{e.iterate_term_1}
    \\
    &\  +\frac{h_2}{n}\sum_{\mathbf{i}'} \tilde\E \left[ \la x^\pr{L-1}_{i_{L-1}}\ra  \dot \varphi^\pr{l+1}_{i_{l+1}}\cdots\dot \varphi^\pr{L-1}_{i_{L-1}}\ \frac{\Phi^\pr{l+1}_{i_{l+1},i_l}}{\sqrt{n_{l}}}\ \frac{\Phi^\pr{l+2}_{i_{l+2},i_{l+1}}}{\sqrt{n_{l+1}}}\ \cdots \ \frac{\Phi^\pr{L-1}_{i_{L-1},i_{L-2}}}{\sqrt{n_{L-2}}}\right]\label{e.iterate_term_2}
\end{align}
where $\sum_\mathbf{i}$ is over~\eqref{e.bf_i_concent_norm} and $\sum_{\mathbf{i}'}$ is over
\begin{align}
\begin{split}\label{e.bf_i'}
    \mathbf{i}'&=(i_{l+1},i_{l+2},\dots, i_{L-1})\in \prod_{m=l+1}^{L-1} \{1,\dots,n_m\},
    \end{split}
\end{align}
respectively. The treatments for~\eqref{e.iterate_term_1} and~\eqref{e.iterate_term_2} are similar to that for~\eqref{e.partial_g_l}, where the main tool is the Gaussian integration by parts summarized in Corollary~\ref{c.GIBP}. Recall that \rv{heuristic} was given below~\eqref{e.partial_g_l}. Now, applying Corollary~\ref{c.GIBP} \rv{(with $L,l+1$ substituted for $k,l$ therein)} to each summand in~\eqref{e.iterate_term_1}, we obtain that, for every $\mathbf{i}$, the summand in~\eqref{e.iterate_term_1} has its absolute value bounded by $C n^{-(L-l)}$.
Since $\sum_{\mathbf{i}}$ is over $O(n^{L-l})$ many terms, we conclude that the part in~\eqref{e.iterate_term_1} is bounded from both sides by $Cn^{-1}$. Analogous arguments can be applied to~\eqref{e.iterate_term_2} to derive a similar bound.
Hence, 
\begin{align*}
    \left|\frac{\partial \tilde \E \hat F}{\partial X^\pr{l}_{i_l}}\right|\leq \frac{C}{n},\quad\forall i_l \in\{1,\dots,n_l\}.
\end{align*}
and thus Lemma~\ref{l.fin_diff} yields~\eqref{e.concent_itr_1}.

\subsubsection{Proof of~\eqref{e.concent_itr_2}}\label{s.pf.e.concent_itr_2}

Let us redefine $\tilde \E = \E\left[\ \cdot \ |X^\pr{l-1},\Phi^{[1,l]}\right]$.
For $i_l\in\{1,\dots,n_l\}$, $i_{l-1}\in\{1,\dots,n_{l-1}\}$, we can compute
\begin{align*}
    \frac{\partial \tilde \E \hat F}{\partial \Phi^\pr{l}_{i_l,i_{l-1}}} &= -\frac{1}{n}\tilde\E \la\left(\sqrt{\beta}Z+\beta\Gamma\right)\cdot \partial_{\Phi^\pr{l}_{i_l,i_{l-1}}} \Gamma\ra  +\frac{1}{n}\tilde \E\la h_2   x^\pr{L-1}\cdot \partial_{\Phi^\pr{l}_{i_l,i_{l-1}}} X^\pr{L-1}\ra\\
    &\qquad +\frac{1}{n}\tilde \E\la \left(h_2  X^\pr{L-1}+\sqrt{h_2}Z'-h_2 x^\pr{L-1}\right)\cdot \partial_{\Phi^\pr{l}_{i_l,i_{l-1}}} x^\pr{L-1}\ra\\
    & = \mathtt{I}_1 + \mathtt{I}_2 +\mathtt{I}_3.
\end{align*}
Here, 
\begin{align*}
    \mathtt{I}_1 & = -\frac{\sqrt{t}}{n}\sum_{\mathbf{i}} \tilde\E \left[ \la \sqrt{\beta} Z_{i_L}+\beta \left(\varphi^\pr{L}_{*,i_L}-\tilde\varphi^\pr{L}_{*,i_L}\right)\ra \frac{X^\pr{l-1}_{i_{l-1}}}{\sqrt{n_{l-1}}} \dot\varphi^\pr{l}_{i_l} \dot \varphi^\pr{l+1}_{i_{l+1}}\cdots\dot \varphi^\pr{L}_{*,i_{L}}\frac{\Phi^\pr{l+1}_{i_{l+1},i_l}}{\sqrt{n_{l}}}\ \cdots \ \frac{\Phi^\pr{L}_{i_{L},i_{L-1}}}{\sqrt{n_{L-1}}}\right]
    \\
    &+\frac{\sqrt{t}}{n}\sum_{\mathbf{i}} \tilde\E \left[ \la \left( \sqrt{\beta} Z_{i_L}+\beta \left(\varphi^\pr{L}_{*,i_L}-\tilde\varphi^\pr{L}_{*,i_L}\right)\right) \frac{x^\pr{l-1}_{i_{l-1}}}{\sqrt{n_{l-1}}} \dot{\tilde\varphi}^\pr{l}_{i_l} \dot {\tilde\varphi}^\pr{l+1}_{i_{l+1}}\cdots\dot {\tilde\varphi}^\pr{L}_{*,i_{L}}\ra\frac{\Phi^\pr{l+1}_{i_{l+1},i_l}}{\sqrt{n_{l}}}\ \cdots \ \frac{\Phi^\pr{L}_{i_{L},i_{L-1}}}{\sqrt{n_{L-1}}}\right]
    \\
    \mathtt{I}_2 & =  \frac{h_2}{n}\sum_{\mathbf{i}'} \tilde\E \left[ \la x^\pr{L-1}_{i_{L-1}}\ra \frac{X^\pr{l-1}_{i_l}}{\sqrt{n_{l-1}}} \dot\varphi^\pr{l}_{i_l} \dot \varphi^\pr{l+1}_{i_{l+1}}\cdots\dot \varphi^\pr{L-1}_{i_{L-1}}\frac{\Phi^\pr{l+1}_{i_{l+1},i_l}}{\sqrt{n_{l}}}\ \frac{\Phi^\pr{l+2}_{i_{l+2},i_{l+1}}}{\sqrt{n_{l+1}}}\ \cdots \ \frac{\Phi^\pr{L-1}_{i_{L-1},i_{L-2}}}{\sqrt{n_{L-2}}}\right]
    \\
    \mathtt{I}_3 & = 
    \frac{1}{n}\sum_{\mathbf{i}'} \tilde\E \Bigg[ \la \left(h_2  X^\pr{L-1}_{i_{L-1}}+\sqrt{h_2}Z'_{i_{L-1}}-h_2 x^\pr{L-1}_{i_{L-1}}\right) \frac{x^\pr{l-1}_{i_l}}{\sqrt{n_{l-1}}} \dot{\tilde\varphi}^\pr{l}_{i_l} \dot {\tilde\varphi}^\pr{l+1}_{i_{l+1}}\cdots\dot {\tilde\varphi}^\pr{L-1}_{i_{L-1}}\ra
    \\
    &\qquad\qquad \qquad \qquad\qquad \qquad\qquad\qquad  \qquad
    \times\frac{\Phi^\pr{l+1}_{i_{l+1},i_l}}{\sqrt{n_{l}}}\ \frac{\Phi^\pr{l+2}_{i_{l+2},i_{l+1}}}{\sqrt{n_{l+1}}}\ \cdots \ \frac{\Phi^\pr{L-1}_{i_{L-1},i_{L-2}}}{\sqrt{n_{L-2}}}\Bigg]
\end{align*}
where $\mathbf{i}$ and $\mathbf{i}'$ are given in~\eqref{e.bf_i_concent_norm} and~\eqref{e.bf_i'}, respectively.

Similar to the the treatments for~\eqref{e.iterate_term_1} and~\eqref{e.iterate_term_2}, applying Corollary~\ref{c.GIBP}, we can see that $|\mathtt{I}_1|,|\mathtt{I}_2|,|\mathtt{I}_3|\leq Cn^{-\frac{3}{2}}$, which implies that
\begin{align*}
    \left|\frac{\partial \tilde \E \hat F}{\partial \Phi^\pr{l}_{i_l,i_{l-1}}}\right|\leq \frac{C}{n^\frac{3}{2}},\quad\forall i_l\in\{1,\dots,n_l\},\ i_{l-1}\in\{1,\dots,n_{l-1}\}.
\end{align*}
Now, we can conclude that $|\nabla_{\Phi^\pr{l}}\tilde \E\hat F|\leq Cn^{-\frac{1}{2}}$ and thus~\eqref{e.concent_itr_2} by Lemma~\ref{l.GPI}.

\subsubsection{An $\eps$-net argument}

By~\eqref{e.partial_2_F_n} and~\eqref{e.1st_der_F_n_est}, there is $C$ such that, for all $t, h_1$ and all $h_2,h'_2\in\R_+$ satisfying $|h_2-h'_2|\leq 1$,
\begin{align*}
    \left|\left(F-\bar F\right)(t,h_1,h_2) -\left(F-\bar F\right)(t,h_1,h'_2) \right|\leq C\left(1+n^{-\frac{1}{2}}\left|Z'\right|\right)\left|h_2-h'_2\right|^\frac{1}{2}.
\end{align*}
Setting $E_\eps = [0,M]\cap \{\eps,2\eps,3\eps,\dots\}$ for $\eps \in (0,1)$, we have that, for all $t,h_1$,
\begin{align*}
    &\E\left[\left\|F - \bar F\right\|^2_{L^\infty_{h_2}([0,M])}(t,h_1)\right]
    \\
    &\leq \E\left[\sup_{h_2\in E_\eps}\left(F - \bar F\right)^2(t,h_1)\right] + \E \left[C\left(1+n^{-\frac{1}{2}}\left|Z'\right|\right)^2\eps\right]
    \\
    &\leq \sum_{h_2\in E_\eps }\E\left[\left(F - \bar F\right)^2(t,h_1)\right]+C\eps\leq C\left(\eps^{-1}n^{-1} + \eps\right),
\end{align*}
where the last inequality follows from~\eqref{e.ptw_concent}.
Optimizing this by taking $\eps = n^{-\frac{1}{2}}$ completes the proof of Lemma~\ref{l.concent}.

\subsection{Multiple Gaussian integrations by parts}

\rv{The goal here is to prove Lemma~\ref{l.multip_GIBP} and Corollary~\ref{c.GIBP}. In estimates of~\eqref{e.partial_g_l},~\eqref{e.iterate_term_1}, 
\eqref{e.iterate_term_2}, and those in Section~\ref{s.pf.e.concent_itr_2}, we have used this corollary to control sums of a growing number of expectations of a complicated expression $\Theta$ multiplied by Gaussians of the form $\Phi^{(m)}_{i_m,i_{m-1}}$. As explained in the heuristic below~\eqref{e.partial_g_l}, the key observation is that the variable $\Phi^{(m)}_{i_m,i_{m-1}}$ inside $\Theta$ always comes with a multiplicative factor $\frac{1}{\sqrt{n_{m-1}}}=O(n^{-\frac{1}{2}}) $. Hence, we want to perform the Gaussian integration by parts on all of these Gaussians outside $\Theta$ starting from the top layer in order to release these factors. Corollary~\ref{c.GIBP} records the corresponding result.} 

\rv{Since we integrate by parts many times, it is natural to use induction. As it happens with induction, we need a more general result, Lemma~\ref{l.multip_GIBP}, in order to proceed in the inductive step. More precisely, we need Lemma~\ref{l.multip_GIBP} to account for performing a certain number of integrations by parts, and Corollary~\ref{c.GIBP} is the end result after performing all instances of integration by parts.}

\rv{
Both the statement and the proof of Lemma~\ref{l.multip_GIBP} may look slightly technical at least on the level of notation. But the proof is no other than performing the Gaussian integration by parts starting from the top layer. In other words, we integrate by parts first with respect to  $\Phi^{L}_{i_L,i_{L-1}}$, then $\Phi^{L-1}_{i_{L-1},i_{L-2}}$, and so forth until some layer $l<L$. The only hard part is to keep track of things.}

\rv{We start by introducing our notation.}
Denote by $\la \,\cdot\,\ra$ the Gibbs measure with Hamiltonian $\hat H$ given in~\eqref{e.hat_H}. \rv{Then, abusing the notation, we set $\la \,\cdot\,\ra = \la \,\cdot\,\ra^\infty$ so that $\la \,\cdot\,\ra$ can average over finitely many replicas.} Recall the variables $x,w,a,a^\pr{L}$ in $\hat H$, and also the definition of $s$ in~\eqref{e.s_mu}. For $\gamma\in\N\cup\{0\}$, we enumerate the replicas, i.e., i.i.d.\ copies of $x, w, a, a^\pr{L}, s$ under $\la \,\cdot\,\ra$, as
\begin{align*}
    x^\rp{\gamma},\  w^\rp{\gamma}, \ a^\rp{\gamma},\ a^\pr{L|\gamma},\ s^\rp{\gamma}.
\end{align*}
Recall the definition of $x^\pr{L-1}$ in~\eqref{e.x^L-1}\rv{;} we want to extend this. Using~\eqref{e.X^l} iteratively, for every $l\in\{0,\dots,L\}$, we can find a deterministic function $\zeta_l$ satisfying
\begin{align*}
    X^\pr{l} = \zeta_l\left(X^\pr{0}, A^{[1,l]}, \Phi^{[1,l]}\right)
\end{align*}
where we understand that $A^{[1,0]}=0$ and $\Phi^{[1,0]}=0$. Replacing $X^\pr{0}, A^{[1,l]}$ above by $x$ and projections of $(a,a^\pr{L})$, we can define $x^\pr{l}$ in a way analogous to~\eqref{e.x^L-1}:
\begin{align*}
    x^\pr{l}=\zeta_l\left(x, \pi_{[1,l]}\left(a,a^\pr{L}\right),\Phi^{[1,l]}\right)
\end{align*}
where $\pi_{[1,l]}$ is the projection of the first $\sum_{m=1}^l n_m k_m$ coordinates into $\prod_{m=1}^l \R^{n_m\times k_m}$ (recall $a=(a^\pr{1},\dots, a^\pr{L-1})$ given in~\eqref{e.x^L-1} and $a^\pr{L}$ in~\eqref{e.a^(L)}).
For $\gamma\in\N\cup\{0\}$, we denote by $x^\pr{l|\gamma}$ the $\gamma$-th replica of $x^\pr{l}$. We also set $\hat {H}^\rp{\gamma}$ to be $\hat H$ with variables replaced by their $\gamma$-th replicas.

Recall $S$ in~\eqref{e.cap_S_mu}. For $\nu\in \N\cup\{0\}$, $\gamma \in \N\cup\{0\}$, $j\in\{1,\dots, n_m\}$, we introduce 
\begin{align*}
    \varphi^\pr{m|\nu}_j & = \frac{\partial^\nu }{\partial r^\nu}\varphi_m\left(r, A^\pr{m}_j\right)\Big|_{r = \frac{1}{\sqrt{n_{m-1}}}(\Phi^\pr{m}X^\pr{m-1})_j},\quad \forall m\in\{1,\dots,L\},
    \\
    \varphi^\pr{L|\nu}_{*,j} &= \frac{\partial^\nu }{\partial r^\nu}\varphi_L\left(r, A^\pr{L}_j\right)\Big|_{r = \frac{1}{\sqrt{n_{L-1}}}(\Phi^\pr{L}S)_j},
    \\
    \tilde\varphi^\pr{m|\nu|\gamma}_j & = \frac{\partial^\nu }{\partial r^\nu}\varphi_m\left(r, a^\pr{m}_j\right)\Big|_{r = \frac{1}{\sqrt{n_{m-1}}}(\Phi^\pr{m}x^\pr{m-1|\gamma})_j},\quad \forall m\in\{1,\dots,L\},
    \\
    \tilde\varphi^\pr{L|\nu|\gamma}_{*,j} & = \frac{\partial^\nu }{\partial r^\nu}\varphi_L\left(r, a^\pr{L}_j\right)\Big|_{r = \frac{1}{\sqrt{n_{L-1}}}(\Phi^\pr{L}s^{|\gamma)})_j}.
\end{align*}
In particular, $\varphi^\pr{m|0}_j = X^\pr{m}_j$ and $\tilde\varphi^\pr{m|0|\gamma}_j = x^\pr{m|\gamma}_j$ and note that these two identities can be extended to $m=0$.
Recall that $Z$ and $Z'$ are standard Gaussian vectors given in~\eqref{e.Y^circ} and~\eqref{e.Y'}, respectively. We introduce the following collections of random variables
\begin{align*}
    \mathcal Z & = \{Z_j\}_{1\leq j\leq n_L}\cup \{Z'_j\}_{1\leq j\leq n_{L-1}},
    \\
    \mathcal M^\pr{l|\nu|\gamma}_j & = \bigcup_{\tilde\nu\in\{0,\dots,\nu\}}\bigcup_{\tilde\gamma\in\{0,\dots,\gamma\}} \left\{\varphi^\pr{l|\tilde\nu}_j,\ \tilde\varphi^\pr{l|\tilde\nu|\tilde\gamma}_j\right\},\quad\forall l \leq L-1, 
    \\
    \mathcal M^\pr{L|\nu|\gamma}_j & = \bigcup_{\tilde\nu\in\{0,\dots,\nu\}} \bigcup_{\tilde\gamma\in\{0,\dots,\gamma\}} \left\{\varphi^\pr{L|\tilde\nu}_j,\ \varphi^\pr{L|\tilde\nu}_{*,j},\ \tilde\varphi^\pr{L|\tilde\nu|\tilde\gamma}_j,\ \tilde\varphi^\pr{L|\tilde\nu|\tilde\gamma}_{*,j}\right\},
    \\
    \bar {\mathcal M} & =  \bigcup_{l\in\{1,\dots,L\}}\bigcup_{\nu\in\N\cup\{0\}}\bigcup_{\gamma\in\N\cup\{0\}}\bigcup_{i_l\in\{1,\dots,n_l\}}\mathcal M^\pr{l|\nu|\gamma}_{i_l}.
\end{align*}
For $\nu_1,\nu_2,\dots,\nu_L,\gamma\in\N\cup\{0\}$, we set
\begin{align*}
    \mathcal N^\pr{\nu_1,\dots,\nu_L|\gamma} = \mathcal Z\cup\left( \bigcup_{m=1}^L\bigcup_{j_m=1}^{n_m}\mathcal M^\pr{m|\nu_m|\gamma}_{j_m}\right).
\end{align*}

For $d,r\in\N$, let $\mathbf{P}_{d,r}$ be the collection of polynomials of degree up to $d$ over $\R^r$ with real coefficients. For every $P\in \mathbf{P}_{d,r}$ expressed as
\begin{align*}
    P(x) = \sum a_{p_1,p_2,\dots ,p_r} x^{p_1}_1x^{p_2}_2\cdots x^{p_r}_r
\end{align*}
where the summation is over
\begin{align*}
    \left\{(p_1,p_2,\dots ,p_r) \in \left(\N\cup\{0\}\right)^r: \sum_{i=1}^r p_i \leq d\right\},
\end{align*}
we define
\begin{align*}
    \|P\| = \sum |a_{p_1,p_2,\dots ,p_r}|.
\end{align*}
Slightly abusing the notation, we view any finite subcollection $\mathcal E \subset \bar{\mathcal M}$ as an ordered tuple of random variables. In this notation, for any $P\in \mathbf{P}_{d,|\mathcal E|}$ for some $d$, we view $P(\mathcal E)$ as $P$ evaluated at $\mathcal E$.
Lastly, for $a,b\in\R$, we write $a\vee b =\max\{a,b\}$.

\begin{lemma}\label{l.multip_GIBP}
Let $l\in \{1,\cdots,L\}$, $\nu_1,\dots,\nu_L \in \N\cup\{0\}$, $\gamma \in \N\cup\{0\}$, $\beta\geq 0$, $M\geq 1$. In addition to~\ref{i.assump_w_|X|<sqrt_n}, assume that $\Phi^\pr{m}$ consists of i.i.d.\ standard Gaussian entries and that $\varphi_m$ is bounded and continuously differentiable with bounded derivatives up to $\nu'_m$-th order for every $m\in\{1,\dots,L\}$, where
\begin{align}\label{e.nu'_m}
    \nu'_m = \nu_m + (2^{m-l+1}-1)\vee0,\quad\forall m\in\{1,\dots,L\}.
\end{align}
Then, there are constants $C,\gamma',d'$ such that the following holds.

For every $k\in\{l,\dots, L\}$, every $n\in\N$, every $(t,h)\in \Omega_{\rho_{L-1,n}}\cap\{|h_2|\leq M\}$, every $i_m \in \{1,\dots,n_m\}$ with $m\in\{l-1,\dots,k\}$, every 
\begin{align*}
    \mathcal E =  \mathcal N^\pr{\nu_1,\dots,\nu_L|\gamma}
\end{align*}
and every $P\in \mathbf{P}_{d,|\mathcal E|}$, there is $P' \in \mathbf{P}_{d',|\mathcal E'|}$ for some
\begin{align*}
    \mathcal E' = \mathcal N^\pr{\nu'_1,\dots,\nu'_L|\gamma'}
\end{align*}
such that
\begin{align}\label{e.induction_result_multi_GIBP}
    \E^{[l,k]} \left[\la P(\mathcal E)\ra\prod_{m=l}^k\Phi^\pr{m}_{i_m, i_{m-1}}\right] = n^{-\frac{1}{2}(k-l+1)} \E^{[l,k]} \la P'(\mathcal{E}')\ra 
\end{align}
and
\begin{align}\label{e.|P'|<C|P|}
    \|P'\|\leq C\|P\|,
\end{align}
where $\E^{[l,k]}$ is the expectation with respect to $\Phi^{[l,k]}$.
\end{lemma}

Recall the notation introduced in~\eqref{e.dot_phi_notation}. We state an immediate corollary of Lemma~\ref{l.multip_GIBP}.

\begin{corollary}\label{c.GIBP}
Assume~\ref{i.assump_iid_bdd}--\ref{i.assump_Phi_Gaussian} for some $L\in \N$. For $l\in \{1,\dots,L\}$, $d\in\N$, $\beta\geq 0$, $M\geq1$, there is $C$ such that the following holds. Suppose that $P\in \mathbf{P}_{d,\,2L-2l+10}$ is a monomial with coefficient $1$ and independent of $n$. Then, for every $k\in\{l,\dots,L\}$, every $n\in\N$, every $(t,h)\in \Omega_{\rho_{L-1,n}}\cap\{|h_2|\leq M\}$, every $i_m \in \{1,\dots,n_m\}$ with $m\in\{l-1,\dots,k\}$, it holds that
\begin{align*}
    \tilde \E \left[\la P(\mathcal E)\ra\prod_{m=l}^k\Phi^\pr{m}_{i_m, i_{m-1}}\right] \leq Cn^{-\frac{1}{2}(k-l+1)}
\end{align*}
where
\begin{align*}
    \mathcal E = \bigg(Z_{i_L},\ Z'_{i_{L-1}},\ \varphi^\pr{L}_{*,i_L},\ \tilde\varphi^\pr{L}_{*,i_{L}},\ \dot\varphi^\pr{L}_{*,i_L},\ \dot{\tilde\varphi}^\pr{L}_{*,i_{L}},\ X^\pr{L-1}_{i_{L-1}},\ x^\pr{L-1}_{i_{L-1}},\ 
    \\
    \left(\dot\varphi^\pr{m}_{i_m}\right)_{m=l}^{L-1},\ \left(\dot{\tilde\varphi}^\pr{m}_{i_m}\right)_{m=l}^{L-1},\ X^\pr{l-1}_{i_{l-1}},\ x^\pr{l-1}_{i_{l-1}}\bigg)
\end{align*}
and $\tilde \E$ integrates over $Z_{i_L}$, $Z'_{i_{L-1}}$ and $(\Phi^\pr{m}_{i_m,i_{m-1}})_{m=l}^k$.
\end{corollary}

\begin{proof}[Proof of Corollary~\ref{c.GIBP}]
Comparing~\eqref{e.dot_phi_notation} with the notation here, we can rewrite
\begin{align*}
    \mathcal E = \bigg(Z_{i_L},\ Z'_{i_{L-1}},\ \varphi^\pr{L|0}_{*,i_L},\ \tilde\varphi^\pr{L|0|0}_{*,i_{L}},\ \varphi^\pr{L|1}_{*,i_L},\ \tilde\varphi^\pr{L|1|0}_{*,i_{L}},\ \varphi^\pr{L-1|0}_{i_{L-1}},\ \tilde\varphi^\pr{L-1|0|0}_{i_{L-1}},\ 
    \\
    \left(\varphi^\pr{m|1}_{i_m}\right)_{m=l}^{L-1},\ \left(\tilde\varphi^\pr{m|1|0}_{i_m}\right)_{m=l}^{L-1},\ \varphi^\pr{l-1|0}_{i_{l-1}},\ \tilde\varphi^\pr{l-1|0|0}_{i_{l-1}}\bigg).
\end{align*}
Hence, we have that $\mathcal E \subset \mathcal N^\pr{1,1,\dots,1|0}$. This corollary follows from Lemma~\ref{l.multip_GIBP} by setting $\gamma=0$ and $\nu_m=1$ for all $m$ and noticing that the differentiability condition~\eqref{e.nu'_m} is fulfilled by assumption~\ref{i.assump_varphi_l_2^l_diff}.
\end{proof}

\begin{proof}[Proof of Lemma~\ref{l.multip_GIBP}]
We use induction on $l$ and start with the base case $l=L$. 
The Gaussian integration by parts yields that
\begin{align}\label{e.E^L_GIBP_L}
    \E^\pr{L} \la P(\mathcal E)\Phi^\pr{L}_{i_L, i_{L-1}}\ra = \E^\pr{L} \left[\partial_{\Phi^\pr{L}_{i_L, i_{L-1}}}\la P(\mathcal E)\ra\right] = \sum_{\phi\in\bar{\mathcal M}}\E^\pr{L}\la  \mathscr{P}_\phi\partial_{\Phi^\pr{L}_{i_L, i_{L-1}}}\phi \ra
\end{align}
where $\E^\pr{L} = \E^{[L,L]}$, and, by the chain rule, viewing $\phi\in\bar{\mathcal M}$ as labels for the arguments in $P$ and \rvv{$\hat H^\rp{\tilde\gamma}$} \rv{for $\tilde\gamma \in \{0,1,\dots,\gamma+1\}$}, we have that
\begin{gather}\label{e.G_phi_L}
    \mathscr{P}_\phi = \partial_\phi P(\mathcal E) + P(\mathcal E)\left( \sum_{\tilde\gamma=0}^\gamma \partial_\phi \hat{H}^\rp{\tilde\gamma} - \gamma \partial_\phi\hat{H}^\rp{\gamma+1}\right),
\end{gather}
with
\begin{align}
    \partial_\phi \hat{H}^\rp{\tilde\gamma} & = -\sum_{j=1}^{n_L}\left(\mathbf{1}_{\phi = \varphi^\pr{L|0}_{*,j}}-\mathbf{1}_{\rv{\phi = \tilde\varphi^\pr{L|0|\tilde\gamma}_{*,j}}}\right)\left(\sqrt{\beta}Z_j+ \beta \left( \varphi^\pr{L|0}_{*,j} - \tilde  \varphi^\pr{L|0|\tilde\gamma}_{*,j}\right)\right)
    \label{e.d_phi_hat_H}\\
    &\quad + \sum_{j=1}^{n_{L-1}}\mathbf{1}_{\phi = \varphi^\pr{L-1|0}_{j}}h_2 \tilde \varphi^\pr{L-1|0|\tilde\gamma}_j\notag\\
    &\quad +\sum_{j=1}^{n_{L-1}} \mathbf{1}_{\rv{\phi = \tilde\varphi^\pr{L-1|0|\tilde \gamma}_{j}}}\left(h_2\varphi^\pr{L-1|0}_j+\sqrt{h_2}Z'_j - h_2 \tilde \varphi^\pr{L-1|0|\tilde\gamma}_j \right).\notag
\end{align}
Let us clarify~\eqref{e.d_phi_hat_H}. Due to the definition of $\hat H$ in~\eqref{e.hat_H}, fixing $Z,Z'$, we can view $\hat H^\rp{\tilde\gamma}$ as a function of $\varphi_L\left(S,A^\pr{L}\right)$, $ \varphi_L\left(s^\rp{\tilde\gamma},a^\pr{L|\tilde \gamma}\right)$, $ X^\pr{L-1}$, $x^\pr{L-1|\tilde\gamma}$, or equivalently, $\varphi^\pr{L|0}_*$, $ \tilde\varphi^\pr{L|0|\tilde\gamma}_*$, $ \varphi^\pr{L-1|0}$, $\tilde\varphi^\pr{L-1|0|\tilde\gamma}$. Therefore, when viewing these as labels for the variables inside $\hat H^\rp{\tilde\gamma}$, we have~\eqref{e.d_phi_hat_H} and the left-hand side of it is nonzero only if $\phi$ is an entry of those vectors.

Next, let us show that
\begin{align}\label{e.P_phi_d_phi_neq_0_L}
    \mathscr{P}_\phi \partial_{\Phi^\pr{L}_{i_L, i_{L-1}}}\phi \neq 0
    \quad\text{only if}\quad
    \phi \in \mathcal M^\pr{L|\nu_L|\gamma+1}_{i_L}.
\end{align}
From~\eqref{e.G_phi_L} and~\eqref{e.d_phi_hat_H}, we can see that
\begin{align}\label{e.G_phi_neq_0_L}
    \mathscr{P}_\phi \neq 0\quad\text{only if}\quad
    \phi \in \mathcal E\cup \left(\bigcup_{j_{L-1}=1}^{n_{L-1}}\mathcal M^\pr{L-1|0|\gamma+1}_{j_{L-1}}\right) \cup \left(\bigcup_{j_{L}=1}^{n_L}\mathcal M^\pr{L|0|\gamma+1}_{j_{L}}\right).
\end{align}
On the other hand, due to~\eqref{e.X^l}, note that
\begin{align}\label{e.dphi_neq_0_L}
    \partial_{\Phi^\pr{L}_{i_L, i_{L-1}}}\phi \neq 0 
    \quad\text{only if}\quad
    \phi \in \bigcup_{\tilde \nu \in \N}\bigcup_{\tilde \gamma\in\N }\mathcal M^\pr{L|\tilde\nu|\tilde\gamma}_{i_L}.
\end{align}
The intersection of sets in~\eqref{e.G_phi_neq_0_L} and~\eqref{e.dphi_neq_0_L} is a subset of the set in~\eqref{e.P_phi_d_phi_neq_0_L}. Hence,~\eqref{e.P_phi_d_phi_neq_0_L} is valid.

Due to~\eqref{e.dphi_neq_0_L},
\rv{whenever $\partial_{\Phi^\pr{L}_{i_L, i_{L-1}}}\phi$ in~\eqref{e.P_phi_d_phi_neq_0_L} is nonzero, it} is of one of the four forms below, for some $\tilde \nu\leq \nu_L$ and $\tilde\gamma\leq \gamma+1$,
\begin{align}
    \begin{cases}\label{e.d_tilde_phi_L}
    \partial_{\Phi^\pr{L}_{i_L, i_{L-1}}} \varphi^\pr{L|\tilde\nu}_{i_L}  = \varphi^\pr{L|\tilde\nu+1}_{i_L}\frac{1}{\sqrt{n_{L-1}}}X^\pr{L-1}_{i_{L-1}}
    \\
    \partial_{\Phi^\pr{L}_{i_L, i_{L-1}}} \tilde\varphi^\pr{L|\tilde\nu|\tilde\gamma}_{i_L}  = \varphi^\pr{L|\tilde\nu+1|\tilde\gamma}_{i_L}\frac{1}{\sqrt{n_{L-1}}}x^\pr{L-1|\tilde\gamma}_{i_{L-1}}
    \\
    \partial_{\Phi^\pr{L}_{i_L, i_{L-1}}} \varphi^\pr{L|\tilde\nu}_{*,i_L}  = \varphi^\pr{L|\tilde\nu+1}_{*,i_L}\frac{1}{\sqrt{n_{L-1}}}S_{i_{L-1}}
    \\
    \partial_{\Phi^\pr{L}_{i_L, i_{L-1}}} \tilde\varphi^\pr{L|\tilde\nu|\tilde\gamma}_{*,i_L}  = \tilde\varphi^\pr{L|\tilde\nu+1|\tilde\gamma}_{*,i_L}\frac{1}{\sqrt{n_{L-1}}}s^\rp{\tilde\gamma}_{i_{L-1}}
    \end{cases}
\end{align}
Using this,~\eqref{e.G_phi_L} and~\eqref{e.d_phi_hat_H}, we can see that for
\begin{align*}
    \mathcal E'  =\mathcal N^\pr{\nu_1,\dots,\nu_{L-1},\nu_{L}+1|\gamma+1}
\end{align*}
there is a polynomial $P'_\phi\in \mathbf{P}_{d',|\mathcal E'|}$ for some $d'$ such that
\begin{align}\label{e.G'_phi_L}
    P'_\phi(\mathcal E')= n^\frac{1}{2}\mathscr{P}_\phi \partial_{\Phi^\pr{L}_{i_L, i_{L-1}}}\phi.
\end{align}
Here, the scalar $n^\frac{1}{2}$ is to make $n^\frac{1}{2} \partial_{\Phi^\pr{L}_{i_L, i_{L-1}}}\phi$ to be of order $1$.
By~\eqref{e.E^L_GIBP_L} and ~\eqref{e.P_phi_d_phi_neq_0_L}, setting
\begin{align*}
    P'(\mathcal E') = \sum_{\phi\in \mathcal M^\pr{L|\nu_L|\gamma+1}_{i_L}} P'_\phi (\mathcal E')
\end{align*}
we have
\begin{align*}
    \E^\pr{L} \la P(\mathcal E)\Phi^\pr{L}_{i_L, i_{L-1}}\ra = n^{-\frac{1}{2}} \E^\pr{L} \la P'(\mathcal E '
    )\ra.
\end{align*}
Using~\eqref{e.G_phi_L},~\eqref{e.d_phi_hat_H},~\eqref{e.d_tilde_phi_L} and~\eqref{e.G'_phi_L}, we can see that
\begin{align*}
    \|P'\|\leq C\|P\|
\end{align*}
for some constant $C$ that depends only on $L,\nu_1,\dots,\nu_L, \gamma,\beta,M$.

Now, we consider the induction step and assume that the lemma holds for $l+1\leq L$. In the following, we denote by $C$ a constant that depends only on $l,\nu_1,\dots,\nu_L, \gamma,\beta,M$ and may vary from line to line. Setting $\E^\pr{l}= \E^{[l,l]}$ and using the induction assumption for $l+1$, we get that for
\begin{align*}
    \mathcal{F}=\mathcal N^\pr{\nu'_1,\dots,\nu'_L|\gamma'}
\end{align*}
with some $\gamma'>0$ and
\begin{align}\label{e.a'_induction}
    \nu'_m = \nu_m + (2^{m-l}-1)\vee0,\quad\forall m\in\{1,\dots,L\},
\end{align}
there is $Q\in \mathbf{P}_{d',|\mathcal F|}$ for some $d'$  such that
\begin{align}
    \E^{[l,k]} \left[\la P(\mathcal E)\ra\prod_{m=l}^k\Phi^\pr{m}_{i_m, i_{m-1}}\right] & = \E^\pr{l}\left[\E^{[l+1,k]} \left[\la P(\mathcal E)\ra\prod_{m=l+1}^k\Phi^\pr{m}_{i_m, i_{m-1}}\right]\Phi^\pr{l}_{i_l, i_{l-1}}\right]\notag\\
    & = n^{-\frac{1}{2}(k-l)}\E^{[l,k]} \left[\la Q(\mathcal F)\ra \Phi^\pr{l}_{i_l, i_{l-1}}\right] \label{e.induction_GIBP}
\end{align}
and
\begin{gather}
    \|Q\|\leq C\|P\|.\label{e.|G|<C|G'|}
\end{gather}
Applying the Gaussian integration by parts to the last expectation in~\eqref{e.induction_GIBP} yields
\begin{align}\label{e.EGPhi_GIBP}
    \E^{[l,k]}\left[\la Q(\mathcal F)\ra \Phi^\pr{l}_{i_l, i_{l-1}}\right] = \E^{[l,k]}\left[ \partial_{\Phi^\pr{l}_{i_l, i_{l-1}}}\la Q(\mathcal F)\ra \right] = \sum_{\phi\in\bar{\mathcal M}} \E^{[l,k]}\la  \mathscr{Q}_\phi\partial_{\Phi^\pr{l}_{i_l, i_{l-1}}}\phi \ra
\end{align}
where
\begin{align}\label{e.Q_phi}
     \mathscr{Q}_\phi = \partial_\phi Q(\mathcal F) + Q(\mathcal F)\left( \sum_{\tilde\gamma=0}^{\gamma'} \partial_\phi \hat{H}^\rp{\tilde\gamma} - \gamma'\partial_\phi \hat{H}^\rp{\gamma'+1}\right).
\end{align}

Next, we show that
\begin{align}\label{e.G'_phi_d_phi_neq_0_induction}
    \mathscr{Q}_\phi \partial_{\Phi^\pr{l}_{i_l, i_{l-1}}}\phi \neq 0
    \quad\text{only if}\quad
    \phi \in\mathcal M^\pr{l|\nu'_l|\gamma'+1}_{i_l}\cup \left(\bigcup_{m=l+1}^L\bigcup_{j_m=1}^{n_m}\mathcal M^\pr{m|\nu'_m|\gamma'+1}_{j_m}\right).
\end{align}
Similar to the derivation of~\eqref{e.G_phi_neq_0_L}, using~\eqref{e.Q_phi} and~\eqref{e.d_phi_hat_H},  we can see that
\begin{align*}\mathscr{Q}_\phi \neq 0
    \quad\text{only if}\quad
    \phi \in \mathcal F \cup \left(\bigcup_{j_{L-1}=1}^{n_{L-1}}\mathcal M^\pr{L-1|0|\gamma'+1}_{j_{L-1}}\right) \cup \left(\bigcup_{j_{L}=1}^{n_L}\mathcal M^\pr{L|0|\gamma'+1}_{j_{L}}\right)
\end{align*}
Due to~\eqref{e.X^l}, note that
\begin{align*}
    \partial_{\Phi^\pr{l}_{i_l, i_{l-1}}}\phi \neq 0
    \quad\text{only if}\quad
    \phi \in \bigcup_{\tilde \nu\in\N}\bigcup_{\tilde \gamma\in \N}\left( \mathcal M^\pr{l|\tilde\nu|\tilde\gamma}_{i_l}\cup \left(\bigcup_{m=l+1}^L\bigcup_{j_m=1}^{n_m}\mathcal M^\pr{m|\tilde\nu|\tilde\gamma}_{j_m}\right)\right).
\end{align*}
The intersection of the sets in the above two displays is contained in the set in~\eqref{e.G'_phi_d_phi_neq_0_induction} and thus~\eqref{e.G'_phi_d_phi_neq_0_induction} is valid.

Then, we compute the summands
in~\eqref{e.EGPhi_GIBP}. Due to~\eqref{e.G'_phi_d_phi_neq_0_induction}, we distinguish two cases:
\begin{align}\label{e.two_cases_GIBP}
    \phi \in\mathcal M^\pr{l|\nu'_l|\gamma'+1}_{i_l}\quad\text{or}\quad \phi \in \bigcup_{m=l+1}^L\bigcup_{j_m=1}^{n_m}\mathcal M^\pr{m|\nu'_m|\gamma'+1}_{j_m}.
\end{align}

Let us consider the first case in~\eqref{e.two_cases_GIBP}. Since $l+1\leq L $, $\partial_{\Phi^\pr{l}_{i_l, i_{l-1}}}\phi$ has one of the two forms below, for $\tilde\nu\leq \nu'_L$ and $\tilde \gamma\leq \gamma'+1$,
\begin{align*}
    \partial_{\Phi^\pr{l}_{i_l, i_{l-1}}}\varphi^\pr{l|\tilde\nu}_{i_l} & = \varphi^\pr{l|\tilde\nu+1}_{i_l} \frac{1}{\sqrt{n_{l-1}}}X^\pr{l-1}_{i_{l-1}},\\
    \partial_{\Phi^\pr{l}_{i_l, i_{l-1}}}\tilde\varphi^\pr{l|\tilde\nu|\tilde\gamma}_{i_l} & = \varphi^\pr{l|\tilde\nu+1|\tilde\gamma}_{i_l} \frac{1}{\sqrt{n_{l-1}}}x^\pr{l-1|\tilde\gamma}_{i_{l-1}}
\end{align*}
From this,~\eqref{e.Q_phi} and~\eqref{e.d_phi_hat_H}, we can see that, for every $\phi$ belonging to the first set in~\eqref{e.two_cases_GIBP}, there is a polynomial $Q'_\phi\in\mathbf{P}_{d_\phi,|\mathcal F'|}$ for some $d_\phi$ \rv{and}
\begin{gather}
    \mathcal F' = \mathcal N^\pr{\bar\nu_1,\dots,\bar\nu_L|\gamma'+1}\label{e.mathcal_F'}
\end{gather}
with
\begin{align}\label{e.bar_alpha_m}
    \bar\nu_m = 
    \begin{cases}
        \nu'_m+1 &\quad m\geq l\\
        \nu'_m &\quad m\leq l-1
    \end{cases}
\end{align}
such that
\begin{gather}
    Q'_\phi(\mathcal F') = n^\frac{1}{2}\mathscr{Q}_\phi \partial_{\Phi^\pr{l}_{i_l, i_{l-1}}}\phi,\quad \forall \phi \in\mathcal M^\pr{l|\nu'_l|\gamma'+1}_{i_l},\notag \\
    \|Q'_\phi\|\leq C\|Q\|, \quad \forall \phi \in\mathcal M^\pr{l|\nu'_l|\gamma'+1}_{i_l}.\label{e.|bar_G|<C|G'|_m=l}
\end{gather}
Therefore
\begin{align}\label{e.EG'_phi_d_phi_m=l}
    \E^{[l,k]}\la  \mathscr{Q}_\phi\partial_{\Phi^\pr{l}_{i_l, i_{l-1}}}\phi \ra = n^{-\frac{1}{2}} \E \la Q'_\phi(\mathcal F')\ra, \quad\forall \phi \in\mathcal M^\pr{l|\nu'_l|\gamma'+1}_{i_l}.
\end{align}

Now, we turn to the second case in~\eqref{e.two_cases_GIBP}. Let us assume that
\begin{align}\label{e.second_case}
    \phi \in \mathcal M_{j_m}^\pr{m|\nu'_m|\gamma'+1},\quad m\in\{l+1,\dots,L\},\ {j_m}\in\{1,\dots,n_m\}.
\end{align}
Then, due to~\eqref{e.X^l} and the chain rule, $\partial_{\Phi^\pr{l}_{i_l, i_{l-1}}}\phi$ is one of the \rvv{following for} $\tilde \nu \leq\nu'_m$ and $ \tilde\gamma\leq\gamma'+1$:
\begin{align*}
    &\partial_{\Phi^\pr{l}_{i_l, i_{l-1}}}\varphi^\pr{m|\tilde\nu}_{j_m}
    \\
    &\qquad=  \varphi^\pr{m|\tilde\nu+1}_{j_m}\sum_{\mathbf{j}}\left(\prod_{\tilde m=l+1}^{m}\frac{1}{\sqrt{n_{\tilde m-1}}}\Phi^\pr{\tilde m}_{j_{\tilde m}, j_{\tilde m-1}}\varphi^\pr{\tilde m-1|1}_{j_{\tilde m-1}}\right)\bigg|_{j_l=i_l}\frac{1}{\sqrt{n_{l-1}}}X^\pr{l-1}_{i_{l-1}},
    \\
    &\partial_{\Phi^\pr{l}_{i_l, i_{l-1}}}\tilde\varphi^\pr{m|\tilde\nu|\tilde\gamma}_{j_m} 
    \\
    &\qquad=  \tilde\varphi^\pr{m|\tilde\nu+1|\tilde\gamma}_{j_m}\sum_{\mathbf{j}}\left(\prod_{\tilde m=l+1}^{m}\frac{1}{\sqrt{n_{\tilde m-1}}}\Phi^\pr{\tilde m}_{j_{\tilde m}, j_{\tilde m-1}}\tilde\varphi^\pr{\tilde m-1|1|\tilde\gamma}_{j_{\tilde m-1}}\right)\bigg|_{ j_l=i_l}\frac{1}{\sqrt{n_{l-1}}}x^\pr{l-1|\tilde\gamma}_{i_{l-1}},
\end{align*}
where the summation is over 
\begin{align}\label{e.mathbf_j}
    \mathbf{j}=(j_{l+1},j_{l+2},\dots, j_{m-2},j_{m-1})\in \prod_{\tilde m=l+1}^{m-1} \{1,\dots,n_{\tilde m}\}.
\end{align}
When $m=L$, there are two more possibilities $\partial_{\Phi^\pr{l}_{i_l, i_{l-1}}}\varphi^\pr{L|\tilde\nu}_{*,j_L}$ and $\partial_{\Phi^\pr{l}_{i_l, i_{l-1}}}\tilde\varphi^\pr{L|\tilde\nu|\tilde\gamma}_{*,j_L}$, which are similar to the above and omitted for brevity.
These computations allow us to write that
\begin{align}\label{e.EQ_phi_d_phi_m>l}
    \E^{[l,k]}\la  \mathscr{Q}_\phi\partial_{\Phi^\pr{l}_{i_l, i_{l-1}}}\phi \ra =n^{-\frac{1}{2}(m-l+1)}\sum_{\mathbf{j}}\E^{[l,k]}\left[\la  \mathscr{Q}_\phi g_{\phi,\,\mathbf{j}}\ra \prod_{\tilde m=l+1}^{m} \Phi^\pr{\tilde m}_{j_{\tilde m}, j_{\tilde m-1}}\right]\,\bigg|_{j_l=i_l}
\end{align}
where
\begin{align}\label{e.g_phi,j}
    g_{\phi,\,\mathbf{j}} = 
    \begin{cases}
    \varphi^\pr{m|\tilde\nu+1}_{j}\left(\prod_{\tilde m=l+1}^{m}\sqrt{\frac{n}{n_{\tilde m-1}}}\varphi^\pr{\tilde m-1|1}_{j_{\tilde m-1}}\right)\sqrt{\frac{n}{n_{l-1}}}X^\pr{l-1}_{i_{l-1}}, &\quad \phi =\varphi^\pr{m|\tilde\nu}_{j_m},
    \\
    \tilde\varphi^\pr{m|\tilde\nu+1|\tilde\gamma}_{j}\left(\prod_{\tilde m=l+1}^{m}\sqrt{\frac{n}{n_{\tilde m-1}}}\tilde\varphi^\pr{\tilde m-1|1|\tilde\gamma}_{j_{\tilde m-1}}\right)\sqrt{\frac{n}{n_{l-1}}}x^\pr{l-1|\tilde\gamma}_{i_{l-1}}, &\quad \phi =\tilde\varphi^\pr{m|\tilde\nu|\tilde\gamma}_{j_m}.
    \end{cases}
\end{align}
By these,~\eqref{e.Q_phi} and~\eqref{e.d_phi_hat_H}, there is a polynomial $Q_{\phi,\,\mathbf{j}}\in\mathbf{P}_{d',|\mathcal F'|}$ for some larger $d'$ independent of $\phi$, $\mathbf{j}$ and for $\mathcal F'$ in~\eqref{e.mathcal_F'} such that
\begin{align}\label{e.Q_phi_g_phi_j}
    Q_{\phi,\,\mathbf{j}}(\mathcal F')=\mathscr{Q}_\phi g_{\phi,\,\mathbf{j}} 
\end{align}
which, due to~\eqref{e.g_phi,j}, also satisfies that
\begin{align}\label{e.|Q_phi_j|<}
    \|Q_{\phi,\,\mathbf{j}}\|\leq C\|Q\|.
\end{align}
Recall that we are considering the case~\eqref{e.second_case}.
Insert~\eqref{e.Q_phi_g_phi_j} into the right-hand side of~\eqref{e.EQ_phi_d_phi_m>l} and applying the induction assumption for $l+1$ to every summand there yields that
\begin{align}
    \E^{[l,k]}\la  \mathscr{Q}_\phi\partial_{\Phi^\pr{l}_{i_l, i_{l-1}}}\phi \ra 
    & = n^{-\frac{1}{2}(m-l+1)}\sum_{\mathbf{j}}\E^{[l,k]}\left[\la  Q_{\phi,\,\mathbf{j}}(\mathcal F')\ra \prod_{\tilde m=l+1}^{m} \Phi^\pr{\tilde m}_{j_{\tilde m}, j_{\tilde m-1}}\right]\,\bigg|_{j_l=i_l}\notag
    \\
    &=n^{-(m-l+\frac{1}{2})}\sum_{\mathbf{j}}\E^{[l,k]}\la Q'_{\phi,\,\mathbf{j}}(\mathcal E') \ra,\qquad\forall \phi \in \mathcal M_{j_m}^\pr{m|\nu'_m|\gamma'+1}\label{e.G'_phi_d_phi_m>l}
\end{align}
for some polynomials $Q'_{\phi,\,\mathbf{j}}\in \mathbf{P}_{d',|\mathcal E'|}$ for some larger $d'$, and
\begin{align}\label{e.E'_induction}
    \mathcal E' = \mathcal N^\pr{\nu''_1,\dots,\nu''_L|\gamma''}
\end{align}
with some larger $\gamma''$ and
\begin{align}\label{e.alpha''_m}
    \nu''_m = \bar\nu_m + (2^{m-l}-1)\vee0,\quad\forall m\in\{1,\dots,L\},
\end{align}
where $\bar\nu_m$ is given in~\eqref{e.bar_alpha_m}.
In addition, each of these polynomials satisfies that
\begin{align}\label{eq:|Q'_phi,j|<}
    \|Q'_{\phi,\,\mathbf{j}}\|\leq C\|Q_{\phi,\,\mathbf{j}}\|.
\end{align}
Since $\sum_{\mathbf{j}}$ is a summation of $O(n^{m-l-1})$ many terms due to~\eqref{e.mathbf_j}, setting
\begin{align}\label{e.P'_phi(E')}
    P'_\phi(\mathcal E') = n^{-(m-l-1)}\sum_{\mathbf{j}} Q'_{\phi,\,\mathbf{j}}(\mathcal E'),
\end{align}
and using~\eqref{e.|Q_phi_j|<} and~\eqref{eq:|Q'_phi,j|<},
we obtain that
\begin{align}
    \|P'_\phi\|\leq C\|Q\|,\quad\forall \phi \in \mathcal M_{j_m}^\pr{m|\nu'_m|\gamma'+1}.\label{e.|bar_G|<C|G'|_m>l}
\end{align}
Inserting~\eqref{e.P'_phi(E')} into~\eqref{e.G'_phi_d_phi_m>l} gives that
\begin{align}\label{e.EG'_phi_d_phi_m>l}
    \E^{[l,k]}\la  \mathscr{Q}_\phi\partial_{\Phi^\pr{l}_{i_l, i_{l-1}}}\phi \ra =n^{-\frac{3}{2}}\E^{[l,k]}\la  P'_{\phi}(\mathcal E')\ra,\quad\forall \phi \in \mathcal M_{j_m}^\pr{m|\nu'_m|\gamma'+1}
\end{align}
for $m\in\{l+1,\dots,L\}$, $ {j_m}\in\{1,\dots,n_m\}$.

Now, we are ready to conclude.
Due to~\eqref{e.G'_phi_d_phi_neq_0_induction}, the summation in~\eqref{e.EGPhi_GIBP} can to restricted to be over the set in~\eqref{e.G'_phi_d_phi_neq_0_induction}. Also note that $\mathcal F'\subset \mathcal E'$ due to their definitions in~\eqref{e.mathcal_F'} and~\eqref{e.E'_induction}.
Using these,~\eqref{e.EG'_phi_d_phi_m=l} and~\eqref{e.EG'_phi_d_phi_m>l}, we can rewrite the left-hand side of~\eqref{e.EGPhi_GIBP} as
\begin{align}
    \E^{[l,k]}\left[\la Q(\mathcal F)\ra \Phi^\pr{l}_{i_l, i_{l-1}}\right] &=\left(\sum_{\phi\in\mathcal M^\pr{l|\nu'_l|\gamma'+1}_{i_l}} + \sum_{m=l+1}^L\sum_{j_m=1}^{n_m}\sum_{\phi\in\mathcal M_{j_m}^\pr{m|\nu'_m|\gamma'+1}}\right)\E^{[l,k]}\la  \mathscr{Q}_\phi\partial_{\Phi^\pr{l}_{i_l, i_{l-1}}}\phi \ra \notag
    \\
    & = n^{-\frac{1}{2}}\E^{[l,k]}\la P'(\mathcal E') \ra \label{e.n^-1/2EP'E'}
\end{align}
where
\begin{align}\label{e.P'(E')}
    P'(\mathcal E') = \sum_{\phi\in\mathcal M^\pr{l|\nu'_l|\gamma'+1}_{i_l}} Q'_\phi(\mathcal F') + \sum_{m=l+1}^L\sum_{j_m=1}^{n_m}\sum_{\phi\in\mathcal M_{j_m}^\pr{m|\nu'_m|\gamma'+1}}n^{-1} P'_{\phi}(\mathcal E').
\end{align}
Inserting~\eqref{e.n^-1/2EP'E'} to~\eqref{e.induction_GIBP} gives the desired result~\eqref{e.induction_result_multi_GIBP}.
Then, we verify~\eqref{e.|P'|<C|P|}.
Note that $\sum_{j=1}^{n_m}$ in~\eqref{e.P'(E')} is a summation of $O(n)$ many terms. Using this,~\eqref{e.|bar_G|<C|G'|_m=l}, and~\eqref{e.|bar_G|<C|G'|_m>l}, we obtain that
\begin{align*}
    \|P'\|\leq C\|Q\|,
\end{align*}
which along with~\eqref{e.|G|<C|G'|} implies~\eqref{e.|P'|<C|P|}. Lastly, by~\eqref{e.a'_induction},~\eqref{e.bar_alpha_m} and~\eqref{e.alpha''_m}, we can see that $\nu''_m$ in the definition of $\mathcal E'$ in~\eqref{e.E'_induction} satisfies
\begin{align*}
    \nu''_m = \nu_m + (2^{m-l+1}-1)\vee0,\quad\forall m\in\{1,\dots,L\},
\end{align*}
completing the proof.
\end{proof}

\small
\bibliographystyle{abbrv}

\begin{thebibliography}{10}

\bibitem{barbier2019optimal}
J.~Barbier, F.~Krzakala, N.~Macris, L.~Miolane, and L.~Zdeborov{\'a}.
\newblock Optimal errors and phase transitions in high-dimensional generalized
  linear models.
\newblock {\em Proceedings of the National Academy of Sciences},
  116(12):5451--5460, 2019.

\bibitem{barbier2019adaptive}
J.~Barbier and N.~Macris.
\newblock The adaptive interpolation method: a simple scheme to prove replica
  formulas in bayesian inference.
\newblock {\em Probability Theory and Related Fields}, 174(3-4):1133--1185,
  2019.

\bibitem{barbier2019adaptive2}
J.~Barbier and N.~Macris.
\newblock The adaptive interpolation method for proving replica formulas.
  applications to the curie--weiss and wigner spike models.
\newblock {\em Journal of Physics A: Mathematical and Theoretical},
  52(29):294002, 2019.

\bibitem{bmm}
J.~Barbier, N.~Macris, and L.~Miolane.
\newblock The layered structure of tensor estimation and its mutual
  information.
\newblock In {\em 55th Annual Allerton Conference on Communication, Control,
  and Computing}, pages 1056--1063. IEEE, 2017.

\bibitem{bardi1984hopf}
M.~Bardi and L.~C. Evans.
\newblock On {Hopf}'s formulas for solutions of {Hamilton-Jacobi} equations.
\newblock {\em Nonlinear Analysis: Theory, Methods \& Applications},
  8(11):1373--1381, 1984.

\bibitem{barra2013mean}
A.~Barra, G.~Dal~Ferraro, and D.~Tantari.
\newblock Mean field spin glasses treated with {PDE} techniques.
\newblock {\em The European Physical Journal B}, 86(7):1--10, 2013.

\bibitem{barra2010replica}
A.~Barra, A.~Di~Biasio, and F.~Guerra.
\newblock Replica symmetry breaking in mean-field spin glasses through the
  {Hamilton--Jacobi} technique.
\newblock {\em Journal of Statistical Mechanics: Theory and Experiment},
  2010(09):P09006, 2010.

\bibitem{boucheron2013concentration}
S.~Boucheron, G.~Lugosi, and P.~Massart.
\newblock {\em Concentration inequalities: A nonasymptotic theory of
  independence}.
\newblock Oxford university press, 2013.

\bibitem{HB1}
H.-B. Chen.
\newblock {Hamilton-Jacobi equations for nonsymmetric matrix inference}.
\newblock {\em arXiv preprint arXiv:2006.05328}, 2020.

\bibitem{chen2021statistical}
H.-B. Chen, J.-C. Mourrat, and J.~Xia.
\newblock Statistical inference of finite-rank tensors.
\newblock {\em arXiv preprint arXiv:2104.05360}, 2021.

\bibitem{chen2020fenchel}
H.-B. Chen and J.~Xia.
\newblock {Fenchel--Moreau} identities on self-dual cones.
\newblock {\em arXiv preprint arXiv:2011.06979}, 2020.

\bibitem{HBJ}
H.-B. Chen and J.~Xia.
\newblock {Hamilton-Jacobi} equations for inference of matrix tensor products.
\newblock {\em arXiv preprint arXiv:2009.01678}, 2020.

\bibitem{dia2016mutual}
M.~Dia, N.~Macris, F.~Krzakala, T.~Lesieur, L.~Zdeborov{\'a}, et~al.
\newblock Mutual information for symmetric rank-one matrix estimation: A proof
  of the replica formula.
\newblock In {\em Advances in Neural Information Processing Systems}, pages
  424--432, 2016.

\bibitem{evans2010partial}
L.~C. Evans.
\newblock {\em Partial Differential Equations}, volume~19.
\newblock American Mathematical Soc., 2010.

\bibitem{gabrie2019entropy}
M.~Gabri{\'e}, A.~Manoel, C.~Luneau, J.~Barbier, N.~Macris, F.~Krzakala, and
  L.~Zdeborov{\'a}.
\newblock Entropy and mutual information in models of deep neural networks.
\newblock {\em Journal of Statistical Mechanics: Theory and Experiment},
  2019(12):124014, 2019.

\bibitem{genovese2009mechanical}
G.~Genovese and A.~Barra.
\newblock A mechanical approach to mean field spin models.
\newblock {\em Journal of Mathematical Physics}, 50(5):053303, 2009.

\bibitem{guerra2001sum}
F.~Guerra.
\newblock Sum rules for the free energy in the mean field spin glass model.
\newblock {\em Fields Institute Communications}, 30(11), 2001.

\bibitem{lions1986hopf}
P.-L. Lions and J.-C. Rochet.
\newblock Hopf formula and multitime {Hamilton-Jacobi} equations.
\newblock {\em Proceedings of the American Mathematical Society}, 96(1):79--84,
  1986.

\bibitem{luneau2019mutual}
C.~Luneau, J.~Barbier, and N.~Macris.
\newblock Mutual information for low-rank even-order symmetric tensor
  estimation.
\newblock {\em Information and Inference: A Journal of the IMA}, 2019.

\bibitem{luneau2020high}
C.~Luneau, N.~Macris, and J.~Barbier.
\newblock High-dimensional rank-one nonsymmetric matrix decomposition: the
  spherical case.
\newblock {\em arXiv preprint arXiv:2004.06975}, 2020.

\bibitem{mourrat2020free}
J.-C. Mourrat.
\newblock Free energy upper bound for mean-field vector spin glasses.
\newblock {\em arXiv preprint arXiv:2010.09114}, 2020.

\bibitem{HJinfer}
J.-C. Mourrat.
\newblock {Hamilton--Jacobi equations for mean-field disordered systems}.
\newblock {\em Annales Henri Lebesgue}, 4:453--484, 2021.

\bibitem{bipartite}
J.-C. Mourrat.
\newblock Nonconvex interactions in mean-field spin glasses.
\newblock {\em Probability and Mathematical Physics}, 2(2):61--119, 2021.

\bibitem{parisi}
J.-C. Mourrat.
\newblock {The Parisi formula is a Hamilton–Jacobi equation in Wasserstein
  space}.
\newblock {\em Canadian Journal of Mathematics}, page 1–23, 2021.

\bibitem{HJrank}
J.-C. Mourrat.
\newblock Hamilton-{J}acobi equations for finite-rank matrix inference.
\newblock {\em Ann. Appl. Probab.}, \noop{2018}to appear.

\bibitem{mourrat2020extending}
J.-C. Mourrat and D.~Panchenko.
\newblock Extending the {Parisi} formula along a {Hamilton-Jacobi} equation.
\newblock {\em Electronic Journal of Probability}, 25, 2020.

\bibitem{reeves2020information}
G.~Reeves.
\newblock Information-theoretic limits for the matrix tensor product.
\newblock {\em IEEE Journal on Selected Areas in Information Theory},
  1(3):777--798, 2020.

\bibitem{rockafellar1970convex}
R.~T. Rockafellar.
\newblock {\em Convex Analysis}, volume~36.
\newblock Princeton university press, 1970.

\end{thebibliography}
\newcommand{\noop}[1]{} \def\cprime{$'$}

\end{document}